\documentclass[a4paper,10pt]{amsart}
\usepackage{epsf}  
\usepackage{amsfonts, amssymb, amsthm, amsmath}  
\usepackage{hyperref}
\usepackage{bm}
\usepackage{tikz}
\usepackage{enumerate}
\usetikzlibrary{matrix,arrows,decorations.pathmorphing}
\usepackage{tikz-cd}
\usepackage{yhmath}
\usepackage{MnSymbol} 
\usepackage{graphicx}
\usepackage[font=small,labelfont=bf]{caption}
\usepackage{psfrag}
\newtheorem{thm}{Theorem}  
\newtheorem{conj}[thm]{Conjecture}
\newtheorem{cor}[thm]{Corollary}  
\newtheorem{lemma}[thm]{Lemma}  
\newtheorem{remark}[thm]{Remark}  
\newtheorem{defn}[thm]{Definition}  
\newtheorem{prop}[thm]{Proposition}  
\newtheorem{claim}[thm]{Claim}  
\newtheorem{example}[thm]{Example}

\numberwithin{thm}{section}  
\def\pf{\noindent\emph{Proof: }}  
\def\stop{\hfill$\square$}

\providecommand{\totl}[1]{\ensuremath{\lceil #1\rceil }}
\providecommand{\totb}[1]{\ensuremath{\underline{ #1}}}

\DeclareMathOperator{\End}{\mathcal E}

\DeclareMathOperator{\rend}{End}
\DeclareMathOperator{\Aut}{Aut}
\DeclareMathOperator{\Spec}{Spec}
\newcommand{\ro}{{}^{r}\Omega}

\newcommand{\rh}{{}^{r}H}

\newcommand{\ex}{\bold}

\providecommand {\e}[1]{\mathfrak t^{#1}}

\providecommand{\C}[2]{\ensuremath {C^{#1,\underline{#2}}}}

\newcommand{\Lag}[1]{Lag\left(#1\right)}

\DeclareMathOperator{\expl}{Expl}

\newcommand{\dbar}{\bar{\partial}}

\providecommand{\et}[2]{\ensuremath{\bold T^{#1}_{#2}}}
\providecommand{\lrb}[1]{\ensuremath{\left(#1\right)}}
\providecommand{\abs}[1]{\left\lvert #1\right\rvert}  

\author{Brett Parker}
\email{ }  
\thanks{}
  
\title[Holomorphic lagrangian Correspondences]{Gromov--Witten invariants of log Calabi--Yau 3-folds are holomorphic lagrangian correspondences}

\begin{document}
\maketitle

\begin{abstract} 
We introduce a holomorphic version of Weinstein's symplectic category, in which objects are holomorphic symplectic manifolds, and morphisms are holomorphic lagrangian correspondences. We then extend this category to log schemes, and prove that Gromov--Witten invariants of log Calabi--Yau 3-folds are naturally encoded as holomorphic lagrangian correspondences. Gromov--Witten invariants and Donaldson--Thomas invariants are then conjecturally related by a natural unitary lagrangian correspondence. \end{abstract}

\tableofcontents

\section{Introduction}

This paper constructs various holomorphic versions of Weinstein's symplectic category; \cite{WSC}. An object $(X,\omega)$ in such a category consists of a smooth complex algebraic space or log scheme $X$ with a holomorphic symplectic form, $\omega$. A morphism $l:(X_1,\omega_1)\longrightarrow (X_2,\omega_2)$ is a formal linear sum of holomorphic lagrangian subvarieties of $(X_1,-\omega_1)\times(X_2,\omega_2)$; we think of such a lagrangian correpondence as a half-dimensional homology class supported on these holomorphic lagrangian subvarieties.  Composition of morphisms $l_1: X_1\longrightarrow X_2$ and $l_2:X_2\longrightarrow X_3$ is defined using a star product, which is a canonical chain level  version of  pushforward and pullback.
\[\begin{tikzcd}  l_1\star_{X_2}l_2:= \pi_*\iota^!(l_1\times l_2) & X_1\times X_2\times X_3\dar{\iota} \rar{\pi} & X_1\times X_3
\\  & X_1\times X_2\times X_2 \times X_3 \end{tikzcd}\]

Our main result, proved and precisely stated in Theorem \ref{GW lagrangian}, is as follows:

\begin{thm}Gromov--Witten invariants of  log Calabi--Yau 3-folds are canonically encoded by holomorphic lagrangian correspondences. 
\end{thm}
 The importance of this theorem for Gromov--Witten theory is that we obtain a \emph{canonical} chain level virtual fundamental class. Moreover, the canonical, chain-level composition of morphisms in the holomorphic Weinstein category corresponds to the operations required for the tropical gluing formula for Gromov--Witten invariants, \cite{gfgw}. 
 
 \
 
 The plan of the paper is as follows. In Section \ref{symplectic Weinstein}, we give a brief history and motivation of Alan Weinstein's proposed symplectic category.  In Section \ref{holomorphic section}, we introduce a holomorphic version of Weinstein's symplectic category, restricting to smooth complex algebraic spaces to make arguments easier. The importance of this algebraic setting is that intersection of algebraic subvarieties is not as pathological as intersection of smooth submanifolds.  In Section \ref{local intersection section}, we observe that there is a well-defined intersection theory of local homology classes supported on algebraic subvarieties.  We explain this both in terms of Chow theory and algebraic topology. Then in Section \ref{Star product section}, we use this local intersection product to define a composition of lagrangian correspondences, called the star product. We observe that the holomorphic symplectic Weinstein category has many nice properties, with products, duals, adjoints, units and counits all defined.  
 
 In Section \ref{correspondence section}, we give an important example of a  holomorphic lagrangian correspondence $\mathcal L$ related to the Nakajima basis for the homology of the Hilbert scheme of points in a Calabi--Yau surface. This lagrangian correspondence is unitary in the sense that its adjoint $\mathcal L^\dagger$ is equal to its inverse. We return to this $\mathcal L$ at the end of the paper, where we conjecture that Gromov--Witten invariants and Donaldson--Thomas invariants of logarithmic Calabi--Yau 3-folds  are related by this holomorphic lagrangian correspondence.
 
In Section \ref{log section}, we extend the definition of the holomorphic Weinstein category to log schemes, $(X,D)$ with $X$ a smooth complex variety and $D$ a simple normal crossing divisor. Here, we must use the log Chow ring or refined cohomology, and the correct intersection theory requires logarithmic modifications of $(X,D)$.  Nevertheless, a lagrangian subvariety of $(X,D)$ has a simple interpretation as a lagrangian subvariety of the smooth variety $X\setminus D$. To have a geometric interpretation of the star product in this setting, we introduce the exploded and tropical perspectives in Section \ref{exploded section}.

In Section \ref{GW section}, we study the moduli space of holomorphic curves in a log Calabi--Yau 3-fold, $(X,D)$.
We first give an overview using families of smooth curves in the interior of $X$, then use exploded manifolds to construct  a natural holomorphic symplectic evaluation space for holomorphic curves in  the explosion $\expl(X,D)$ of $(X,D)$. The holomorphic symplectic form on this evaluation space is induced as a kind of residue from the holomorphic volume form on $(X,D)$. It is also possible to construct this evaluation space using log schemes, but such a construction requires choices of logarithmic modifications of $(X,D)$, and is slightly less natural. We use algebraic results from the theory of log Gromov--Witten invariants to deduce that the image of the moduli stack of holomorphic curves in our evaluation space is algebraic, and use exploded manifold arguments to deduce that this image is isotropic. A key technical result is Proposition \ref{explodable family}, which relates the moduli stacks of holomorphic curves in the exploded manifold and logarithmic settings. This result relies on a version of resolution of singularities, Proposition \ref{explodable resolution}. Once we have that the image of the moduli stack of holomorphic curves is algebraic and isotropic, we can then deduce that the image of the virtual fundamental class is canonically represented by a holomorphic lagrangian correspondence.

 We conclude  by explaining how to encode the  Gromov--Witten partition function $Z_{GW}$ as a holomorphic lagrangian correspondence, and sketching how to encode Pandharipande--Thomas invariants and Donaldson--Thomas invariants as holomorphic lagrangian correspondences, $Z'_{DT}$ and $Z_{PT}$. Conjecturally, $Z_{PT}=Z'_{DT}$, but  the holomorphic lagrangian correspondences $Z_{GW}$ and $Z_{PT}$ lie in different spaces, related by the natural unitary lagrangian correspondence $\mathcal L$, from Example \ref{HSD2c}.  We then state Conjecture \ref{GWPT}, which asserts that, after a change of variables, $q^\frac 12=ie^{i\hbar}$,  the correspondence between Gromov--Witten invariants and Donaldson--Thomas invariants is  
 \[Z_{GW}=\mathcal L\star Z_{PT} \text{ and } \mathcal L^\dagger \star Z_{GW}=Z_{PT}\ . \]

\subsection{The symplectic Weinstein `category'}\label{symplectic Weinstein}

\

A  symplectic manifold is a smooth manifold $M$ with a non-degenerate, closed $2$-form $\omega$. A lagrangian submanifold of $M$ is a half--dimensional submanifold on which the symplectic form vanishes. 

\begin{example}The graph of a diffeomorphsm $f:(M,\omega_M)\longrightarrow (N,\omega_N)$ is a lagrarangian submanifold of \[M^- N:=(M\times N, -\omega_M+\omega_N)\] if and only if $f^*\omega_N=\omega_M$.
\end{example}

\begin{example}Classical mechanics takes place on phase space, or the cotangent bundle $T^*X$ of a smooth configuration space $X$.  The cotangent space has a natural symplectic structure,  and examples of lagrangian submanifolds of $T^*X$ are cotangent fibres, and the graphs of closed $1$--forms.   
\end{example}

Alan Weinstein's symplectic category has been described as the best idea in symplectic topology that didn't work. 
In the 1970's  Alan Weinstein suggested that a natural domain for geometric quantisation consists of a symplectic category whose morphisms $(M,\omega_M)\longrightarrow (N,\omega_N)$ are lagrangian correspondences, or lagrangian submanifolds of $M^-N$.

Weinstein's idea was partially motivated by pioneering work of H\"ormander on the semiclassical limit of operators; \cite{Hormander}.  H\"ormander constructed  `classical' lagrangian correspondences in $(T^*X)^-  T^*Y$  as wave-front  sets of certain `quantum'  operators between the spaces of functions on $X$ and $Y$. Moreover, when lagrangian correspondences are sufficiently transverse, the classical composition of correspondences coincides with the quantum composition of operators. For a modern symplectic perspective, see \cite{GSSemiclassical}.

\begin{example} The `quantum' Fourier integral operator  \[f\mapsto \int e^{\frac{i}\hbar \phi(x,y)}a(x,y,\hbar)f(x)dx\] is associated  with  a `classical' lagrangian correspondence $\Gamma_\phi\subset (T^*X)^-  T^*Y$ defined using $d\phi$. In particular $\Gamma_\phi$ is  the image of the  graph of $d\phi$ under  the symplectic map $T^*(XY)\longrightarrow (T^*X)^-  T^*Y$ defined by multiplying cotangent fibres of $T^*X$ by $-1$.
\end{example}

 The Weinstein symplectic category is a wonderfully enticing idea, but suffers difficulties: the composition of lagrangian correspondences is only a lagrangian submanifold when the correspondences intersect cleanly.  In this paper, we consider various holomorphic versions of Weinstein's symplectic category. The pathologies  of smooth intersection theory do not occur in the more rigid holomorphic or algebraic setting, so we can overcome the complications plaguing Weinstein's symplectic category. 
 

\section{The holomorphic symplectic Weinstein category}\label{holomorphic section}
 
 In this section, we work with a smooth complex manifold $X$, along with a holomorphic symplectic form $\omega$. To make arguments simpler, we also assume  that $X$ is algebraic; this simplifying assumption is not necessary for the results to hold, but we later use key results from the theory of logarithmic Gromov--Witten theory that are only proved in the algebraic setting.  

 \begin{defn}[Holomorphic symplectic form] A holomorphic symplectic form $\omega$ on $ X$ is a closed, holomorphic $2$--form  which is nondegenerate in the sense that $v\mapsto \iota_v\omega$ is a isomorphism between the holomorphic tangent and cotangent spaces of $X$. 
 \end{defn}

   Throughout the paper $ X^-$ will indicate $X$ with the opposite symplectic form, $-\omega$.

\begin{example}\label{HS1} Suppose that $Y$ is Calabi--Yau 2-fold, meaning that it is an algebraic surface with a holomorphic volume form and hence a holomorphic symplectic form.  Then, $Y^n$ is a holomorphic symplectic manifold. Moreover, the Hilbert scheme of length n, 0--dimensional subschemes  in $Y$ is a smooth holomorphic symplectic manifold $Y^{[n]}$. The holomorphic symplectic form on $Y^{[n]}$ is uniquely determined by its restriction to the dense subset comprised of subschemes supported on  $n$ distinct points. There is a natural $n!$--fold cover of this dense subset by the open subset of $Y^n$ comprised of $n$ distinct points. The holomorphic  symplectic form on $Y^{[n]}$ is the unique form that pulls back to  the holomorphic symplectic form on $Y^n$.  See \cite[Thm 1.17]{Nakajima} and \cite{ Beauville83}. 

 \end{example}

\begin{defn}[Isotropic and Lagrangian subvarieties] A subvariety $V\subset (X,\omega)$ is isotropic if the restriction of $\omega$ to $V$ vanishes; so $f^*\omega=0$ for $f$ any algebraic map $f:U\longrightarrow V$ with $U$ smooth.  Say that $V$ is lagrangian if it is isotropic and $2\dim V=\dim X$.
\end{defn}

\begin{remark} In a $2n$-dimensional symplectic vector space, an isotropic subspace has dimension bounded by $n$. As any subvariety of a smooth space $X$ is generically smooth, isotropic subvarieties of $X$ have dimension bounded by $\frac 12\dim X$.
\end{remark}

\begin{example}\label{HSDT} Suppose that $(X,Y)$ is a log Calabi-Yau 3-fold with $Y\subset X$ a smooth surface. So, there is a non-vanishing holomorphic volume form on $X\setminus Y$ with a simple pole at $Y$. Then $Y$ is Calabi--Yau 2-fold, and  there is an evaluation map from the moduli space of stable rank 2  sheaves on $X$ to the Hilbert scheme of points on $Y$. The image of this evaluation map is a union of lagrangian subvarieties \cite[Thm 1.6.1]{ThomasThesis}.
\end{example}

\begin{example}\label{HS2}  For $Y$ a Calabi--Yau surface, any $1$--dimensional subvariety $\Sigma\subset Y$ is automatically a lagrangian subvariety of  $Y$, moreover, there are interesting lagrangian subvarieties $L^{\mu}\Sigma$ of the Hilbert scheme of points $Y^{[n]}$ for partitions $\mu$ of $n$; see \cite[Section 9.3]{Nakajima}.

 The Hilbert--Chow morphism $\pi: Y^{[n]}\longrightarrow S^nY$ sends a $0$--dimensional subscheme to the $0$--dimensional cycle it represents. The fibres of $\pi$ are isotropic. Moreover,   given any lagrangian subvariety $\Sigma\subset Y$, we have that $\pi^{-1}S^n\Sigma$, is a union of isotropic subvarieties. 
 
 In fact, \cite[Thm 1.17, Thm 5.11, Thm 5.12]{Nakajima}, $\pi^{-1}S^n\Sigma$ is a union of lagrangian subvarieties, $L^\mu\Sigma$, one for each partition $\mu$ of $n$,  defined as follows: Let $S^n_\mu(Y)$ indicate the space of $0$--cycles in the form $\sum \mu_i y_i$ where $y_i$ are distinct points on $Y$. This manifold inherits a natural holomorphic symplectic structure as the quotient of a dense subset of $\prod_i (Y,\mu_i\omega)$. The Hilbert--Chow morphism restricted to the inverse image of $S^n_{\mu}(Y)$ is a smooth submersion with irreducible isotropic fibres of dimension $\sum_i (\mu_i-1)= n-\frac 12\dim S^n_{\mu}(Y)$,  and the holomorphic symplectic form restricted to $\pi^{-1}(S^n_{\mu}(Y))$ is  the pullback of the natural symplectic form.   Similarly, let $S^n_{\mu}(\Sigma)$ indicate those cycles in $S^n_\mu(Y)$ that are supported on $\Sigma$. The closure of $\pi^{-1}S^n_{\mu}(\Sigma)$ is a lagrangian subvariety $L^\mu\Sigma \subset Y^{[n]}$.
\end{example}
%

So long as a lagrangian subvariety of $X^-  Y$ satisfies a properness condition, it will define a morphism from $X$ to  $Y$ in our holomorphic version of the Weinstein category. We will also include formal linear combinations of such morphisms.

\begin{defn}[Lagrangian correspondence] A lagrangian subvariety of $X^-  Y$ is a correspondence from $X$ to $Y$ if its projection to $X^-$ is proper. 
The group of lagrangian correspondences  $ \Lag {X, Y}$ is the $\mathbb Z$--module generated by lagrangian subvarieties $\ell\subset X^-  Y$ that are correspondences from $X$ to $Y$.

Let $p$ be a point, and define  
\[\Lag{X}:=\Lag{p,X}\] and 
 \[\Lag{X}^-:=\Lag{X,p}\ .\] 
\end{defn}

\begin{defn}[Adjoint] If a lagrangian subvariety $\ell$ of $X^-  Y$ has proper projection to both $X^-$ and $Y$, define the adjoint $\ell^\dagger\in Lag(Y, X)$ of $\ell\in Lag(X, Y)$ to be the image of $\ell$ under the map which swaps the order of coordinates. Similarly, define the adjoint of a lagrangian correspondence $\sum c_k\ell_k$ to be $\sum c_k \ell_k^\dagger$ when each $\ell_k$ has proper projection to $Y$.   
\end{defn}

  Note that in this context of subvarieties of algebraic manifolds over $\mathbb C$, the algebraic notion of proper coincides with the topological notion using the classical topology on $X^-  Y$ as a complex manifold. So,  $\Lag{X}$ consists of the $\mathbb Z$--module generated by compact lagrangian subvarieties of $X$ whereas $\Lag{X}^-$ is the $\mathbb Z$--module generated by all lagrangian subvarieites. In the main case of interest for this paper, $X$ is compact so $\Lag{X}$ and $\Lag{X}^-$ coincide, and taking adjoints provides an isomorphism between $\Lag{X,Y}$ and $ \Lag{Y,X}$.
  
In general, we have that
\[\Lag{A,B  C}\subset \Lag{A  B^-,C}\]
because both are generated by lagrangian subvarieties of $A^-  B  C$, but the requirement that these subvarieties have proper projection to $A$ is stronger than the requirement that they have proper projection to $A  B$. 

\begin{example} The graph of the inclusion of an open subset $U\subsetneq Y$ is a lagrangian subvariety of both $U^-  Y$ and $Y^-  U$, and is a correspondence in $\Lag{U,Y}$, but is not a correspondence in $\Lag{Y,U}$. 
\end{example}

\begin{example}\label{HSD} Let $(Y,\omega_Y)$ be a Calabi--Yau 2-fold. For a partition $\mu=(\mu_1,\mu_2,\dotsc)$ of $n$, define a holomorphic symplectic manifold $Y^\mu$ as follows. 
\[Y^\mu:=\prod_i (Y,\mu_i\omega_Y)\] 
There is a cycle map $Y^\mu\longrightarrow S^n Y$ sending $\prod_i y_i$ to $\sum_i \mu_i y_j$. We can define a lagrangian correspondence  $\Delta^\mu\subset (Y^{\mu})^-  Y^{[n]}$, as the set theoretic fibre product of $Y^\mu$ with $Y^{[n]}$ over $S^n(Y)$.

\[\begin{tikzcd}\Delta^\mu\dar \rar & Y^{[n]}\dar 
\\ Y^\mu \rar & S^n(Y)\end{tikzcd}\]

To see that $\Delta^\mu$ is isotropic, note that over $S^n_{\mu}(Y)\subset S^n(Y)$, both maps are submersions, and the restriction of their symplectic forms are just the pullback of the natural holomorphic symplectic form on the symmetric product $S^n_{\mu}(Y)$. It follows that $\Delta^\mu$ is isotropic restricted to the inverse image of $S^n_\mu(Y)$. It is not obvious that $\Delta^\mu$ is the closure of this isotropic subset, however any extra irreducible component must be generically contained in the inverse image of $S^n_{\mu'}(Y)\subset S^n(Y)$ for a partition $\mu'\leq \mu$ created by joining some parts of $\mu$, and this component must  also be isotropic by a similar argument. So, we can conclude that $\Delta^\mu$ is isotropic. 

Over $S^n_{\mu}(Y)$, the dimension of $\Delta^\mu$ is $\dim Y^{\mu}$, plus the dimension of fibres the Hilbert--Chow map, so
\[ \dim \Delta^\mu=\dim Y^\mu + \sum_i(\mu_i-1)=\dim Y^\mu+n-\frac 12\dim Y^\mu =\frac 12\dim(Y^\mu  Y^{[n]})\ .\] So, $\Delta^\mu$ is lagrangian, at least restricted to the inverse image of $S^n_{\mu}$. Moreover any extra irreducible components of $\Delta^\mu$ are isotropic, with dimension strictly less than $\frac 12(\dim Y^\mu  Y^{[n]} )$, because the dimension restricted to the inverse image of $S^n_{\mu'}$ is bounded by $\frac 12(\dim Y^{\mu'}  Y^{[n]})<\frac 12(\dim Y^\mu  Y^{[n]} )$.  Note also that $\Delta^\mu$ restricted to the inverse image of $S^n_{\mu}$ is irreducible, because its projection to $Y^{\mu} $ is a submersion with irreducible fibres and irreducible image. So, $\Delta^\mu$ consists of a unique irreducible lagrangian component, possibly together with some isotropic components of lower dimension.

The analogous subset of $(Y^{[n]})^-  Y^\mu$ is also a correspondence, which we regard as the adjoint 
\[(\Delta^\mu)^\dagger\in\Lag{Y^{[n]},Y^\mu}\] of 
\[\Delta^\mu\in \Lag{Y^\mu,Y^{[n]}}\ .\]

\end{example}
In the following we define the standard composition of relations.

\begin{defn}[Composition of relations] Given  subsets $r_1\subset A^-  X$ and $r_2\subset X^-  B$, define 
\[r_1\circ_X r_2\subset A^-  B\]
as  the set of points $(a,b)\in A^-  B$ such that there exists some $x\in X$ satisfying $(a,x)\in r_1$ and $(x,b)\in r_2$.
\end{defn}

\begin{lemma} \label{isotropic}Suppose that $\ell_1\subset  A^-  X$ and $\ell_2\subset X^-  B$ are lagrangian subvarieties and that $\ell_1\in \Lag{A,X}$; so  the projection of  $\ell_1$ to $A^-$ is proper.  Then 
\[ \ell_1\circ_X \ell_2  \]
is a finite union of isotropic subvarieties $W_i$ of $A^-  B$. Moreover if  $\ell_2\in \Lag{X,B}$,  then the projection of these subvarieties to $A^-$ is proper.
\end{lemma}
\begin{proof}  Consider the fibre product of  $\ell_1$ and $\ell_2$ as sets over the space $X$. This fibre product embeds as a Zarski closed subset of $A^-  X  B$, so is a finite union of subvarieties  $V_i\subset   A^-  X  B$. Moreover $\ell_1\circ_X\ell_2$ is the projection of this fibre product to $A^-  B$. As the projection of $\ell_1$ to $A^-$ is proper,  the projection of this fibre product to $A  B$ is also proper,  so the image of $V_i$ is a subvariety $W_i$  of $A^-  B$. We need to check that the symplectic form $-\omega_A +\omega_B$ vanishes on each of these subvarieties. (Here, we use the notation $\omega_M$ for the symplectic form on $\omega_M$ or its pullback to the product of $M$ with something else.) To check that $-\omega_A+\omega_B$ vanishes on $W_i$, it suffices to check that $-\omega_A+\omega_B$ vanishes on $V_i$, because the projection $V_i\longrightarrow W_i$ is generically a submersion. We have that $\omega_A=\omega_X$ on $\ell_1$, and hence the same holds on the fibre product of $\ell_1$ with $\ell_2$. Similarly, $\omega_X=\omega_B$ on $\ell_2$ and also on the fibre product of $\ell_1$ with $\ell_2$. So, on each component $V_i$ of this fibre product $\omega_A=\omega_X=\omega_B$, so we obtain that $-\omega_A+\omega_B$ vanishes on the image, $W_i$, of $V_i$ in $A^-  B$. We conclude that $\ell_1\circ_X\ell_2$ is a finite union of isotropic subvarieties.

Let us check that the projection of $\ell_1\circ_X\ell_2$ to $A^-$ is proper when both $\ell_1$ and $\ell_2$ are correspondences. As the algebraic and topological notions of properness coincide, we use the topological notion in our argument. As $\ell_1$ is a correspondence, the inverse image of a compact subset  $K\subset A^-$ is a compact subset $K'\subset\ell_1\subset A^-  X$, which has compact image  $K''$ in $X$. If $\ell_2$ is also a correspondence, the inverse image of $K''$ in $\ell_2$ is also compact, and so the inverse image of $K$ within the fibre product of $\ell_1$ and $\ell_2$ is compact. It follows that the projection of $\ell_1\circ_X\circ \ell_2$ to $A^-$ is proper.

\end{proof}

Lemma \ref{isotropic} also implies that if $\ell_i$ consist of a finite union of lagrangian subvarieties, then $\ell_1\circ_X\ell_2$ is a finite union of isotropic subvarieties.


\begin{example}\label{HSD2}Let $\Delta^\mu\subset (Y^\mu)^-  Y^{[n]}$ and $(\Delta^{\mu})^\dagger\subset (Y^{[n]})^-  Y^\mu$ be the lagrangian correspondences from Example \ref{HSD}. Let $\Aut \mu$ be the group of symmetries of the partition $\mu$, acting on  $Y^\mu=\prod_i(Y,\mu_i\omega_Y)$ by permuting coordinates with the same weight $\mu_i$.    Then 
\[\Delta^\mu\circ_{Y^{[n]}}(\Delta^\mu)^\dagger\subset (Y^\mu)^-  Y^\mu\]
consists of the image of the diagonal under the action of $\Aut \mu$. In contrast $\Delta^\mu\circ_{Y^{[n]}}(\Delta^{\mu'})^\dagger$ is isotropic, but has no lagrangian components when $\mu\neq \mu'$. 

The other composition
\[(\Delta^\mu)^\dagger\circ_{Y^\mu}\Delta^\mu\subset (Y^{[n]})^-  Y^{[n]} \] 
is a union of lagrangian subvarieties $\ell^{\mu'}$, one for each partition $\mu'\leq \mu$. To construct $\ell^{\mu'}$, recall that the image of $Y^\mu$ in $S^n(Y)$ is the union of $S^n_{\mu'}(Y)$ for each of these partitions $\mu'\leq \mu$, where $S^n_{\mu'}(Y)$ is  comprised of cycles $\sum \mu'_i y_i$ with $y_i$ distinct. Restricted to the inverse image of $S^n_{\mu'}(Y)$, we have that $(\Delta^\mu)^\dagger\circ_{Y^\mu}\Delta^\mu$ is the inverse image of the diagonal in $S^n_{\mu'}(Y)^-  S^n_{\mu'}(Y)$, and, as in Example \ref{HS2}, we can determine that this  is irreducible and lagrangian, so taking its closure we get an irreducible lagrangian subvariety $\ell^{\mu'}\subset (Y^{[n]})^-  Y^{[n]}$.
\end{example}

\subsection{Local intersection product} \label{local intersection section}

\

In what follows, we use intersection theory, with either the intersection product from differential topology or algebraic geometry sufficient for our purpose. Within algebraic geometry, the set of $k$--dimensional algebraic cycles $Z_k(X)$ is the $\mathbb Z$--module consisting of finite formal $\mathbb Z$--linear combinations of $k$--dimensional subvarieties of $X$. The Chow group $A_*(X)$  is  the quotient of $Z_*(X)$ by the equivalence relation generated by rational equivalence. When $X$ is smooth, the intersection product defined in \cite[section 8]{Fulton} provides a ring structure on $A_*(X)$.

 There is a cycle map from the Chow group to Borel--Moore homology of $X$. Each $n$--dimensional subvariety $V\subset X$  represents a canonical homology class in $H^{BM}_{2n}(X)$,  the $2n$--dimensional homology with closed support. This is because the singular locus  $V^{sing}\subset V$  is a subvariety of strictly lower dimension, and hence has real dimension at most $2n-2$, so both $H_{2n}^{BM}(V^{sing})$ and $H_{2n-1}^{BM}(V^{sing})$ vanish, and  $H_{2n}^{BM}(X)=H_{2n}^{BM}(X\setminus V^{sing})$. We can  therefore define $[V]\subset H_{2n}^{BM}(X)$ as the class represented by the smooth, closed submanifold $V\setminus V^{sing}$ within $H^{BM}_{2n}(X\setminus V^{sing})=H^{BM}_{2n}(X)$. A rational equivalence between subvarieties can be used to construct a cobordism with a real codimension 2 singular locus, so this map $Z_*(X)\longrightarrow H_{2*}^{BM}(X)$ factors through $A_*(X)$. 
 
 When $X$ is compact, $H^{BM}_*(X)=H_*(X)$, and whenever $X$ is smooth or a subvariety, $H^{BM}_*(X)=H_*(X\cup \infty,\infty)$ where $X\cup \infty$ indicates the one point compactification of $X$.  Moreover, when $X$ is smooth, $H^{BM}_*(X)\equiv H^{\dim_{\mathbb R}X -*}(X)$, and using the ring structure from $H^{\dim_{\mathbb R}X -*}(X)$, this cycle map  $A_*(X)\longrightarrow H^{BM}_{2*}(X)\equiv H^{\dim_{\mathbb R}X -2*}(X)$ is a ring homomorphism. 
 
 More generally,  
when $V\subset X$ is a closed algebraic subset, \[H^{BM}_*(V)=H^{\dim_{\mathbb R}X-*}(X,X\setminus V)\ .\]
Moreover, given any open subset $U\subset  X$ retracting onto $V$, we have that 
\[H^{BM}_*(V)=H^{\dim_{\mathbb R}X-*}(X,X\setminus U)\ \]
so we can think of $H^{BM}_k(V)$ as codimension--$k$ cohomology classes on $X$ supported close to $V$.

For $x$ and $y$ subvarieties of a smooth variety $X$, there is a refined intersection product $x\cdot y\in A_*(x\cap y)$ whose pushforward to $A_*(X)$ is the usual intersection product; \cite[Section 8.1]{Fulton}.  An analogue of this refined intersection product can also be defined using  the intersection product in differential topology, but we call this the local intersection product, because a different notion of refinement will occur later in this paper. In what follows, we define the local intersection product.

\begin{defn}[local homology class] A local homology class, $(a,V)$, is a closed algebraic subset $V\subset X$ and a homology class $a\in H^{BM}(V)$. \end{defn}

\begin{remark}An analogous theory works if we replace closed algebraic subsets with closed complex analytic subsets. What is important is that an open neighborhood of $V$ retracts onto $V$, and that intersections and pullbacks of $V$ are still respectively  algebraic or complex analytic. \end{remark}

The following case is of particular interest to us: if each irreducible component of $V$ has dimension bounded by $k$, then  $H^{BM}_{2k}(V)$ is freely generated by the irreducible components of $V$ with dimension $k$. So, each element of $H^{BM}_{2k}(V)$ is represented by a canonical algebraic cycle, and can be interpreted as simply assigning a numerical weight to each $k$--dimensional irreducible component.

The intersection of  two local homology classes $(a,V)$, and $(b,W)$ is the local homology class $(a\cdot b,V\cap W)$, where $a\cdot b$ is the intersection of $a$ and $b$ defined by identifying $H_*^{BM}(V)=H^{\dim_{\mathbb R}X-*}(X,X\setminus V)$, and using the usual product on relative cohomology. 
 \[\begin{tikzcd}H^{\dim_{\mathbb R} X-*_1}(X,X\setminus V)\times H^{\dim_{\mathbb R} X-*_2}(X,X\setminus W)\dar & \lar H^{BM}_{*_1}(V)\times H^{BM}_{*_2}(W)\dar
\\H^{2\dim_{\mathbb R} X-*_1-*_2}(X,X\setminus (V\cap W))\rar  & H^{BM}_{*_1+*_2-\dim_{\mathbb R} X} (V\cap W)\end{tikzcd}\]

Geometrically, we choose transverse smooth cycles representing $a$ and $b$ within neighbourhoods of $V$ and $W$, then the intersection of these cycles is contained in a neighborhood of $V\cap W$. 
Other important operations on homology also work for local homology classes.
\begin{itemize}
\item Given $a_i$ in $H^{BM}_*(V_i)$ for $V_i\subset X$, we have $a+b\in H^{BM}_*(V_1\cup V_2)$.
\item Given $a_i\in H^{BM}_*( V_i)$, for $V_i\subset X_i$, we have $a_1\times a_2\in H^{BM}_*(V_1\times V_2)$. 
\item Given a map $f:X\longrightarrow Y$ so that $f$ restricted to $V$ is proper,  we can push forward $(a,V)$ using   the usual pushforward of Borel--Moore  homology  $f_*: H^{BM}_*(V)\longrightarrow H^{BM}_*(f(V))$. 
\item Given $f:X\longrightarrow Y$, the pullback defined using Poincare duality induces a pullback map 
\[\begin{tikzcd}H^{BM}_*(V)\dar{\equiv}\rar{f^!} & H^{BM}_{*+\dim_{\mathbb R} X-\dim_{\mathbb R} Y}(f^{-1}(V))
\\ H^{\dim_{\mathbb R} Y-*}(Y,Y\setminus V)\rar{f^*} & H^{\dim_{\mathbb R} Y-*}(X, X\setminus f^{-1}(V))\uar{\equiv} \end{tikzcd}\]

\end{itemize}

 The following properties of the local  intersection product follow easily from the analogous properties for the usual intersection product. 

\begin{itemize} 
\item Given local homology classes $a,b$ in $H^{BM}_*(V_i)$, the local intersection product is graded commutative. 
\begin{equation}a\cdot b =(-1)^{(\deg a)(\deg b)}b\cdot a\in H^{BM}_*(V_1\cap V_2)\end{equation}
\item  
\begin{equation}a\cdot(b+c)=a\cdot b+a\cdot c\end{equation}
\item
\begin{equation}a\times (b+c)=a\times b+a\times c\end{equation}
\item Given $a_i$ in $H^{BM}_*(V_i)$,  we have
\begin{equation}\label{pa}(a_1\cdot a_2)\cdot a_3=a_1\cdot (a_2\cdot a_3)\end{equation}  within $H^{BM}_*(V_1\cap V_2\cap V_3)$.
Similarly, 
\begin{equation}(a_1\times a_2)\times a_3=a_1\times(a_2\times a_3)\end{equation}
in $H^{BM}_*(V_1\times V_2\times V_3)$,
and 

\begin{equation}(a_1+a_2)+a_3=a_1+(a_2+a_3)\end{equation}
in $H^{BM}_*(V_1\cup V_2\cup V_3)$.

\item Both pushforward and pullback are functorial:
\[(f\circ g)_*=f_*\circ g_*\]
\[(f\circ g)^!=g^!\circ f^!\]
\item Given even degree local homology classes $a$ and $b$ within $X$  and $c, d$ within  $Y$, we have
\begin{equation}\label{pm}(a\cdot b)\times( c\cdot d)=(a\times c)\cdot (b\times d)\end{equation}
\item If $f:X\longrightarrow Y$ is a proper algebraic map, given local homology classes $a$ within $X$ and $b$ within $Y$, then
\begin{equation}\label{pc}f_*(a)\cdot b=f_*(a\cdot f^! b)\ .\end{equation} 
Moreover, given local homology classes $b$ and $c$ within $Y$, 
\begin{equation}\label{pi} f^!(a\cdot b)=f^!(a)\cdot f^!(b) \ .\end{equation}

\item Given a fibre product diagram of proper algebraic maps:
\[\begin{tikzcd} X\times_Y Z\dar{\pi'} \rar{f'} & Z\dar{\pi}
\\ X\rar{f} &Y
\end{tikzcd}\]
we have 
\begin{equation}\label{fp} f^!\circ \pi_*=\pi'_*\circ (f')^!\ .\end{equation}
\end{itemize}

\begin{defn}[geometric local homology class] Say that a local homology class is geometric if it is $\sum_i c_i[V_i]\in H^{BM}_{2k}(\bigcup_i V_i)$, where each $V_i$ is a different subvariety of dimension $k$, and all coefficients $c_i$ are nonzero. 
\end{defn}

\begin{remark}A lagrangian correspondence $x=\sum_i c_i \ell_i\in\Lag {X,Y}$ represents a unique geometric local homology class $[x]=\sum c_i[\ell_i]\in H^{BM}_{\dim X  Y}(\bigcup_i \ell_i)$. We also say that this cycle represents a (possibly non-geometric) local homology class $\sum c_i[\ell_i]\in H^{BM}_{\dim X  Y}(V)$ when each $\ell_i\subset V\subset X$. So long as $V$ is a finite union of varieties with complex dimension bounded by $\frac 12 \dim X  Y$, each element of $H_{\dim X  Y}(V)$ can be represented by at most one lagrangian cycle.  

Note that it is not necessarily true that $[x]+[y]=[x+y]$ for $x, y\in \Lag{X,Y}$, because $[x]+[y]$ need not be geometric,  but it is always true that $[x]+[y]$ is represented by $x+y$.   \end{remark}

\subsection{Star products}
\label{Star product section}

\

In this section, we define composition of lagrangian correspondences, using an operation on local homology classes called the star product. This operation is a combination of a fibre product and an a pushforward.

\begin{defn}[star product of local homology classes]\label{new star} Suppose that the local homology class $r_1$ on $A  X$ has support $V$ with proper projection to $A$. Given a local homology class $r_2$ within $ X  B$ define the star product of $r_1$ and $r_2$ to be the local homology class $r_1\star_X r_2$ on $A  B$ given by 
\[r_1\star_X r_2:=\pi_*\iota^! (r_1\times r_2) \]
where $\pi$ and $\iota$ are the projection and diagonal inclusion in the following diagram. 
\[\begin{tikzcd} A  X  B\dar{\iota} \rar{\pi} & A  B
\\ (A  X)  (X  B) \end{tikzcd}\]
\end{defn}

In the above definition, we need the support of $r_1$ to be proper over $A$ so that the support of $\iota^!(r_1\times r_2)$ is proper over $A  B$.

\begin{lemma}[star product of lagrangian cycles] Given $\ell_1\in \Lag{A,X}$ and $\ell_2\in\Lag{X  B, C}$, there exists a unique lagrangian cycle $\ell_1\star_X\ell_2\in \Lag{A  B, C}$ representing the local homology class $[\ell_1]\star_X[\ell_2]$.
 \end{lemma}
\pf 
Lemma \ref{isotropic} implies that the local homology class $[\ell_1]\star_X [\ell_2]$ is supported on a finite union $V$ of isotropic subvarieties. Moreover, each of these subvarieties has proper projection to $A  B$.  We also  have that the degree of $[\ell_1]\star_X[\ell_2]$ is $\dim A+\dim X+\dim X+\dim B+\dim C-2\dim X=\dim(A  B  C)$, and  the homology classes of $V$ of this degree in $A  B  C$ are freely generated by lagrangian subvarieties.
\stop

\begin{lemma}\label{asp} Given local homology classes $r_1$ within $A  X$, $r_2$ within $ X  B   Y$, and $r_3$ within $ Y  C$, we have 
\[(r_1\star_X r_2)\star_Y r_3=r_1\star_X(r_2\star_Yr_3)\]
\end{lemma}

\begin{proof}
This is a consequence of the compatiblity of pullbacks and pushforwards with products and fibre products, equations (\ref{pc}), (\ref{pi}) and (\ref{fp}).

Let $M=A  B  C$, and consider the following commutative diagram, containing two fibre product diagrams.
\[\begin{tikzcd} MY^2 & MY\lar{\iota'_Y}\rar{\pi'_Y} & M
\\ MXY^2\uar{\pi_X} \dar{\iota_X} & MXY\uar\lar\rar\dar & MX\uar{\pi'_X} \dar{\iota'_X}
\\ MX^2 Y^2 & MX^2 Y\lar{\iota_Y}\rar{\pi_Y} & MX^2 \end{tikzcd}\]
As pushforward and pullbacks are compatible with fibre products as in Equation (\ref{fp}), the above diagram implies that 
\begin{equation}\label{cfp} (\pi_X')_*(\iota_X')^!(\pi_Y)_*\iota_Y^!(r_1\times r_2\times r_3)=(\pi_Y')_*(\iota_Y')^!(\pi_X)_*\iota_X^!(r_1\times r_2\times r_3)\ .\end{equation}
Similarly, the following  cartesian diagram,  
\[\begin{tikzcd}MX^2 Y^2\dar & MX^2 Y\lar{\iota_Y}\rar{\pi_Y}\dar & MX^2\dar 
\\ B Y^2 C & BYC\lar{\iota}\rar{\pi} & BC\end{tikzcd}\]
and the compatiblity of pullbacks with products and fibre products,  Equations (\ref{pi}) and (\ref{fp}),   implies that
\[\iota_Y^!(r_1\times r_2\times r_3)=r_1\times (\iota^!(r_2\times r_3))\]
and then, compatibility of pushforwards with products and fibre products,  Equations  (\ref{pc}) and (\ref{fp}), imply that 
\[(\pi_Y)_*\iota_Y^!(r_1\times r_2\times r_3)=r_1\times (\pi_*\iota^!(r_2\times r_3))=r_1\times (r_2\star_Y r_3)\]
so, 
\[r_1\star_X(r_2\star_Y r_3)=(\pi_X')_*(\iota_X')^!(\pi_Y)_*\iota_Y^!(r_1\times r_2\times r_3)\ .\]
Similarly, we get that 
\[(r_1\star_X r_2)\star_Y r_3=(\pi_Y')_*(\iota_Y')^!(\pi_X)_*\iota_X^!(r_1\times r_2\times r_3)\]
so Equation (\ref{cfp}) implies our required equality. 
\[(r_1\star_X r_2)\star_Y r_3=r_1\star_X(r_2\star_Yr_3)\]

\end{proof}

\begin{lemma}\label{associative}Given lagrangian cycles $\ell_1\in\Lag{A, X}$, and $\ell_2\in \Lag{X  B,Y}$, and $\ell_3\in\Lag{Y, C}$, we have
\[(\ell_1\star_X \ell_2)\star_Y \ell_3=\ell_1\star_X(\ell_2\star_Y\ell_3)\]
\end{lemma}
\begin{proof} This follows from the associativity of star products of local homology classes; Lemma \ref{asp}.
The only complication is that $[\ell_1]\star_X[\ell_2]$ might not be $[\ell_1\star_X\ell_2]$, and instead we have
\[[\ell_1]\star_X[\ell_2]=[\ell_1\star_X \ell_2]+z\]
where $z$ is the zero homology class supported on some finite union of isotropic subvarieties. Similarly, we have 
\[([\ell_1]\star_X[\ell_2])\star_Y[\ell_3]=[(\ell_1\star_X\ell_2)\star_Y\ell_3]+z\star_Y[\ell_3]+z'\]
where both $z\star_Y[\ell_3]$ and $z'$ are the zero homology class supported on some finite union of isotropic varieties. So, $(\ell_1\star_X\ell_2)\star_Y\ell_3$ is the unique lagrangian cycle representing $([\ell_1]\star_X [\ell_2])\star_{Y}[\ell_3]$. A similar argument shows that $\ell_1\star_X(\ell_2\star_Y\ell_3)$ is the unique lagrangian cycle representing $[\ell_1]\star_X([\ell_2]\star_Y[\ell_3])=([\ell_1]\star_X[\ell_2])\star_Y[\ell_3]$

\end{proof}

A lagrangian cycle $f\in \Lag{ X, Y}$ determines a map $f_\star:\Lag X\longrightarrow \Lag Y$, defined by
\[f_\star (\ell):=\ell\star_X f\in \Lag{Y}\]

As an immediate corollary of Lemma \ref{associative}, we have that, given $f\in \Lag{ X, Y}$ and $g\in \Lag{ Y, Z}$, 
\[ g_\star\circ f_\star= (f \star_Y g)_\star \ .\]
In the case that $f$ is the graph of a holomorphic symplectic map $F:X\longrightarrow Y$, it is easy to verify that $f_\star \ell=F(\ell)$.

\begin{lemma}\label{identity} Given $f\in \Lag{ X, Y}$, 
\[\Delta_X\star_X f=f=f\star_Y\Delta_Y\]
\end{lemma}
\begin{proof} Without loss of generality, assume that $f$ is represented by a holomorphic lagrangian subvariety $\ell$ of $ X^-  Y$. Then, $( X^-   \Delta_X  Y)\subset X^-  X  X^-  Y$  is transverse to $(\Delta_X\times \ell)$.
 The intersection of these two subvarieties is the set of points $(x,x,x,y)$, where $(x,y)\in \ell$, so the projection of this to $ X^-  Y$ is $\ell$, and $\Delta_X\star_X \ell=\ell$. An analogous argument gives that $f\star_Y\Delta_Y=f$. 
\end{proof}

\begin{remark}\label{stratumstructure}
Lemmas \ref{associative} and \ref{identity} imply that there is a category with objects holomorphic symplectic manifolds, and morphisms $X_1\longrightarrow X_2$ defined by $\Lag{ X_1, X_2}$, where $\Delta_X\subset \Lag{ X, X}$ is the identity. This is a holomorphic version of the symplectic category envisioned by Weinstein. 

 Analogous to tensor products, the usual product of holomorphic symplectic manifolds makes this a strict symmetric monoidal category, with the isomorphism $X  Y\longrightarrow Y  X$ defined by the diagonal lagrangian in $ X^-   Y^-  Y  X$. 
 
 As any lagrangian in $ X^-  Y$ is also a lagrangian in $ Y^-  X$, whenever $X$ is compact,  every morphism $f\in \Lag{ X, Y}$ has an `adjoint' morphism $f^\dagger\in \Lag{ Y, X}$. More generally, the adjoint of $f$ exists whenever $f$ is also proper over $Y$.  It is easy to verify that whenever $f\in\Lag{ X, Y}$ is the graph of a holomorphic symplectomorphism $X\longrightarrow Y$, $f$ is unitary in the sense that $ f^\dagger$ is the inverse to $f$. 
 
  Each object $X$ also has a  dual $ X^-$ and the diagonal lagrangian $\Delta_X\subset X^-  X$ can  be interpreted as a counit --- that is,  a morphism $\epsilon_X\in \Lag{X^-  X}^-$ from $ X^-  X$ to the point. Moreover,  when $X$ is compact, this diagonal lagrangian  can also be interpreted as a unit ---  that is, a morphism $\eta_X\in \Lag{X  X^-}$ from the point to $ X   X^-$.  
  
  Restricting to the subcategory of compact holomorphic symplectic manifolds,  all this structure gives our holomorphic symplectic Weinstein category the structure of a dagger compact closed category; see \cite[Section 7]{selinger} for definitions and a convenient graphical calculus.

\end{remark}

\subsection{A unitary correspondence to the Hilbert scheme of points}
\label{correspondence section}

\

Recall that if $(Y,\omega)$ is a Calabi--Yau surface, we use the notation $Y^{[n]}$ for the Hilbert scheme of $n$ points in $Y$, and given a partition $\mu=(\mu_1,\mu_2,\dotsc)$ of $n$, we use the notation
\[Y^\mu:=\prod_k (Y,\mu_k\omega)\ .\]
This holomorphic symplectic manifold $Y^\mu$ occurs as an evaluation space for relative Gromov--Witten invariants in a log Calabi--Yau 3-fold $(X,Y)$, but there is also a more natural evaluation stack\footnote{The reader need not be familiar with stacks, as our lagrangian correspondences in stacks will always have concrete interpretations using the space $Y^\mu$ instead of $\mathcal Y^\mu$. } $\mathcal Y^\mu$. In this section, we define an important lagrangian correspondence $\mathcal L$ between $Y^{[n]}$ and $\coprod_{\mu} \mathcal Y^\mu$. This lagrangian correspondence is related to the Nakajima basis for the homology of $Y^{[n]}$.  

The exact structure of the evaluation stack $\mathcal Y^\mu$ depends on the normal bundle to $Y\subset X$, but our lagrangian correspondence is insensitive to such subtle details. In the simplest case, when the normal bundle to $Y\subset X$ admits roots of order $\mu_k$,  there exists an etal\'e map $Y^\mu\longrightarrow \mathcal Y^\mu$, which is a $z_\mu$--fold cover, where \[z_\mu=\abs {\Aut \mu}\prod_k\mu_k\ ,\] and $\mathcal Y^\mu$ is isomorphic to the quotient of $Y^\mu$ by the action of an extension, $G_\mu$, of $\Aut\mu$.
\[0\longrightarrow \prod_k\mathbb Z_{\mu_k}\longrightarrow G_\mu\longrightarrow \Aut\mu\longrightarrow 0 \] 
The action of $G_\mu$ on $Y^\mu$ is induced by the action of $\Aut\mu$ permuting the factors of $ Y^\mu=\prod_k(Y,\mu_k\omega)$. So, the action of the cyclic group $\prod_k\mathbb Z_{\mu_k}$ is trivial. More generally, such an etal\'e map only locally exists, and $\mathcal Y^\mu$ is only locally isomorphic to such a quotient. In this more subtle case, we have the maps of stacks
\[\begin{tikzcd} \mathcal Y^\mu\dar 
\\ Y^\mu/\Aut\mu & \lar Y^\mu \end{tikzcd}\]
where the downward map has fibres the classifying stack of $\prod_k\mathbb Z_{\mu_k}$.  In any case, we identify lagrangian subvarieties of $\mathcal Y^\mu$ with $(\Aut\mu)$--orbits of lagrangian subvarieties of $Y^\mu$, so our definition of lagrangian correspondences does not actually depend on the exact structure of the stack $\mathcal Y^\mu$.

\begin{defn}[Lagrangian correspondence to or from $\mathcal Y^\mu$]\label{stack star product}Define a lagrangian correspondence  $\ell_1\in \Lag{A, \mathcal Y^\mu}$ to be an $(\Aut\mu)$--invariant correspondence   $\ell_1'\in \Lag{ A, Y^\mu}$ that is divisible by $\prod_k\mu_k$, and define a lagrangian correspondence  $\ell_2\in\Lag{\mathcal Y^{\mu},B}$ as an $(\Aut\mu)$--invariant correspondence $\ell_2'\in \Lag{Y^\mu,B}$. Given such correspondences,  define the composition 
\[\ell_1\star_{\mathcal Y^\mu}\ell_2=\frac 1{z_\mu}\ell'_1\star_{Y^\mu}\ell'_2\in\Lag{A, B}\]
\end{defn}

The above definition is good enough for our purposes, but we also give the following tentative definition for Lagrangian correspondences in Deligne--Mumford stacks. 

\begin{defn}[Lagrangian correspondences in Deligne--Mumford stacks] Suppose that $\mathcal X$ is a Deligne--Mumford stack with a holomorphic symplectic structure. An element of $\Lag{\mathcal X}^-$ is an element $\ell_U\in \Lag{U}^-$ for each etal\'e map $U\longrightarrow \mathcal X$. These $\ell_U$ must be compatible with pullbacks in the sense that given any commutative diagram of etal\'e maps,
\[\begin{tikzcd}U_1\rar{ f}\ar{dr} &U_2\dar
\\  & \mathcal X\end{tikzcd}\]
$\ell_{U_1}$ is the pullback of $\ell_{U_2}$ using $f$.
Given another Deligne--Mumford stack $\mathcal Y$ with a holomorphic symplectic structure, an element of $\Lag{\mathcal X, \mathcal Y}$ is an element $\ell$ of $\Lag{\mathcal X^-  \mathcal Y}^-$ satisfying the additional condition that there exists a finite collection of integers $n_k$ and  maps of Deligne--Mumford stacks $\mathcal W_k\longrightarrow \mathcal X  \mathcal Y$ that are proper and representable over $\mathcal X$ such that, given an etal\'e map $U\longrightarrow \mathcal X  \mathcal Y$, we have that $\ell_U$  consists  the pushforward of $\sum_k n_k\left[\mathcal W_k\times_{\mathcal X \mathcal Y} U\right]$. 
\end{defn}

\begin{example} The correspondence  $\Delta^\mu\subset (Y^\mu)^-  Y^{[n]}$  from Example \ref{HSD} defines a correspondence in $\Lag{\mathcal Y^\mu,Y^{[n]}}$, whereas its adjoint $(\Delta^\mu)^\dagger$ only defines a correspondence in $\Lag{Y^{[n]}, \mathcal Y^\mu}\otimes \mathbb Q$, because it is not divisible by $\prod_k\mu_k$.\end{example}

 In what follows, we show that
\[\mathcal L:=\sum_\mu (\prod_k i^{\mu_k-1})\Delta^\mu\in \Lag{\coprod_\mu \mathcal Y^{\mu}, Y^{[n]}}\otimes\mathbb Z[i]\]
is unitary, in the sense that the inverse of $\mathcal L$ is $\mathcal L^\dagger$. To interpret this statement, we extend the star product to be complex bilinear, and define $\mathcal L^\dagger:=\sum_\mu (\prod_k i^{\mu_k-1})(\Delta^\mu)^\dagger$, so we are not taking complex conjugates of coefficients.

\begin{example}In this example, we describe the identity $\Delta_{\mathcal Y^\mu}\in\Lag{\mathcal Y^\mu,\mathcal Y^{\mu}}$. This can be regarded as the image of the diagonal map $\mathcal Y^\mu\longrightarrow(\mathcal Y^\mu)^-  \mathcal Y^\mu$. Define $\Delta_{\mathcal Y^\mu}'\subset \Lag{Y^\mu, Y^{\mu}}$ to be $\prod_k\mu_k$ times the image of the diagonal under the action of $\Aut\mu$ on the first factor. In the case that there is an etal\'e map $ Y^\mu\longrightarrow \mathcal Y^\mu$, this can be thought of as the image of $Y^\mu \times_{\mathcal Y^\mu}Y^\mu$. As this lagrangian correspondence is $\Aut\mu$ invariant with respect to the action on both factors, and divisible by $\prod_k\mu_k$, it represents a lagrangian correspondence $\Delta_{\mathcal Y^\mu}\in\Lag{\mathcal Y^\mu,\mathcal Y^{\mu}}$.

Given any lagrangian correspondence $\ell_1\in \Lag{A, \mathcal Y^\mu}$, $\Delta_{\mathcal Y^\mu}'$ is transverse to the lift $\ell_1'$ of $\ell_1$ in $A^-  Y^\mu$, so it is easy to verify that
\[\ell_1'\star_{Y^\mu}\Delta_{\mathcal Y^\mu}'= z_\mu \ell_1'\ .\]
So,
\[\ell_1\star_{\mathcal Y^\mu}\Delta_{\mathcal Y^\mu}=\ell_1 \]
and similarly
\[\Delta_{\mathcal Y^\mu}\star_{\mathcal Y^\mu}\ell_2=\ell_2\ ,\]
so $\Delta_{\mathcal Y^\mu}$ is the identity on $\mathcal Y^\mu$.

\end{example}

\begin{example}\label{HSDc} Let $\Delta^\mu\in\Lag {Y^\mu, Y^{[n]}}$ be the lagrangian correspondence from Example \ref{HSD}. Given a compact lagrangian subvariety $\Sigma\subset Y$, there is a  lagrangian correspondence  $\Sigma^\mu\in \Lag {Y^\mu}$ given by $\Sigma^\mu:=\prod_k\mu_k[\Sigma]$.  Then 
\[\Sigma^\mu \circ_{Y^\mu}\Delta^\mu=\coprod_{\mu'\leq \mu} L^{\mu'}\Sigma\]
where $L^\mu\Sigma\subset Y^{[n]}$ is the lagrangian subvariety from Example \ref{HS2}, and $\mu'\leq \mu$ indicates that a partition $\mu'$ can be created by joining parts of $\mu$. To compute $\Sigma^\mu \star_{Y^\mu}\Delta^\mu$, it remains to compute the coefficient of each subvariety $L^{\mu'}\Sigma$. Away from the big diagonal in $Y^\mu$, $\Delta^\mu$ is transverse to $\Sigma^\mu$, and we get that $z_\mu=\abs{\Aut \mu}\prod_k\mu_k$ is the coefficient of $L^{\mu}\Sigma$. 

The coefficients of $L^{\mu'}\Sigma$ for $\mu'\neq \mu$ vanish for dimension reasons. We can perturb $\Sigma^{\mu}$   to a smooth submanifold $N\subset Y^\mu$  transverse to the big diagonal and therefore transverse to $\Delta^{\mu}$. We then have that $N\times_{Y^\mu} \Delta^\mu\longrightarrow N$ is surjective, and a submersion over each stratum of $N$ (stratified by intersection with the big diagonal in $Y^\mu$).  
 When the length of $\mu$ is $k$ greater than the length of $\mu'$, the intersection of $N$  with the $\mu'$-component of the big diagonal has real codimension $4k$, whereas the fibre of $\Delta^\mu$ over this component has extra real dimension $2k$. It follows that  the coefficient of $L^{\mu'}\Sigma$ is $0$ for $\mu'\neq \mu$. 
  So, we obtain
\[ \Sigma^\mu \star_{Y^\mu}\Delta^\mu= z_\mu L^\mu\Sigma  \ .\]
If we think of $\Sigma^\mu$ as defining a lagrangian corrrespondence in $\Lag {\mathcal Y^\mu}$ and  $\Delta^\mu$ as in $\Lag{\mathcal Y^\mu,Y^{[n]}}$, we have 
\[\Sigma^\mu \star_{\mathcal Y^\mu}\Delta^\mu=L^\mu\Sigma\ .\]

\end{example}

\begin{example} \label{HSD2c}In this example, we construct a unitary Lagrangian correspondence  $\mathcal L\in \Lag{\prod_\mu \mathcal Y^\mu,Y^{[n]}}\otimes \mathbb Z[i]$.

 The correspondence $\Delta^\mu\subset (Y^\mu)^-  Y^{[n]}$ from Example \ref{HSD} defines an element of both $\Lag{Y^\mu, Y^{[n]}}$ and  $\Lag{\mathcal Y^\mu,Y^{[n]}}$. As observed in Example \ref{HSD2}, the set theoretic composition $\Delta^\mu\circ_{Y^{[n]}}(\Delta^\mu)^\dagger$, of $\Delta^\mu$ with its adjoint consists of the image of the diagonal under the action of permuting coordinates in $Y^\mu  Y^\mu$. The fibre product $\Delta^\mu \times_{Y^{[n]}} (\Delta^{\mu})^\dagger$ is not transverse, however it can be computed over the inverse image of $S^n_\mu Y$ using \cite[Thm 9.20]{Nakajima}, getting a factor of $\prod_k \mu_k(-1)^{\mu_k-1}$. This local computation works on a dense subset of $\Delta^\mu\circ_{Y^{[n]}}(\Delta^\mu)^\dagger$, so we obtain
\[\Delta^\mu\star_{Y^{[n]}}(\Delta^\mu)^\dagger=\left(\prod_k \mu_k(-1)^{\mu_k-1}\right)\Delta^\mu\circ_{Y^{[n]}}(\Delta^\mu)^\dagger =\lrb{\prod_k(-1)^{\mu_k-1} }\Delta_{\mathcal Y^\mu}\ .\]
Moreover, $\Delta^\mu\star_{Y^{[n]}}(\Delta^{\mu'})^\dagger=0$ for $\mu'\neq \mu$, because the set theoretic composition has lower dimension than lagrangian cycles. So, we have

\begin{equation}\label {Deltac2}\Delta^{\mu'}\star_{ Y^{[n]}}(\Delta^\mu)^\dagger=\begin{cases}(-1)^{\sum_k(\mu_k-1)}\Delta_{\mathcal Y^\mu} & \text{ if $\mu=\mu'$}\\ 0  & \text{ if $\mu\neq \mu'$}\end{cases}\end{equation}
The sign in the above equation is the sign of permutations of cycle type $\mu$.

From Example \ref{HSD2}, we also have that
\[(\Delta^\mu)^\dagger\circ_{Y^\mu} \Delta ^\mu=\coprod_{\mu'\leq \mu} \ell^\mu\]
where  the lagrangian subvariety $\ell^{\mu'}\subset (Y^{[n]})^-  Y^{[n]}$ is the closure of the inverse image of the diagonal within $S^n_{\mu'}Y$.  Away from the big diagonal in $Y^{\mu}$, the fibre product $(\Delta^\mu)^\dagger\times_{Y^\mu} \Delta ^\mu$ is transverse, so the coefficient of $\ell^\mu$ in $(\Delta^\mu)^\dagger\star_{Y^\mu} \Delta ^\mu$ is readily calculated as $\abs{\Aut \mu}$. From this, we can conclude that 
\begin{equation}\label{Deltac1}(\Delta^\mu)^\dagger\star_{Y^\mu} \Delta ^\mu= \abs{\Aut \mu}\ell^\mu+\sum_{\mu'<\mu} c_{\mu',\mu}\ell^{\mu'} \end{equation}
for some integers $c_{\mu',\mu}$. In particular, this implies that there is a unique linear combination of the above terms summing to $n!\ell^{(1,1,\dotsc)}$, which is the diagonal $\Delta_{Y^{[n]}}\subset (Y^{[n]})^-  Y^{[n]}$.  So,  there are unique coefficients $a_\mu\in\mathbb Q$ such that
\begin{equation}\label{Deltalc}\sum_\mu a_\mu (\Delta^\mu)^\dagger\star_{\mathcal Y^\mu} \Delta ^\mu=\Delta_{Y^{[n]}}\end{equation}
Equations \ref{Deltac2} and \ref{Deltalc}  (and associativity of the star product, Lemma \ref{associative}) then give
\[\Delta^\nu=\Delta^\nu \star_{ Y^{[n]}} \sum_\mu a_\mu (\Delta^\mu)^\dagger\star_{\mathcal Y^\mu} \Delta ^\mu= a_\nu \left(\prod_k(-1)^{\nu_k-1}
\right) \Delta^\nu\]
so $a_\mu$ is the sign of permutations of cycle type $\mu$.

We now define an important lagrangian correspondence between $Y^{[n]}$ and $\coprod_{\mu} \mathcal Y^\mu$.
\[\mathcal L:=\sum_\mu (\prod_k i^{\mu_k-1})\Delta^\mu\in \Lag{\coprod_\mu \mathcal Y^{\mu}, Y^{[n]}}\otimes\mathbb Z[i]\]
\[\mathcal L^\dagger:=\sum_\mu (\prod_k i^{\mu_k-1})(\Delta^\mu)^\dagger \in \Lag{ Y^{[n]},\coprod_\mu \mathcal Y^{\mu}}\otimes\mathbb C\]
We then have 

\[\mathcal L \star_{Y^{[n]}}\mathcal L^\dag=\sum_\mu \Delta_{\mathcal Y^{\mu}}\]
and 
\[\mathcal L^\dagger\star_{\coprod_\mu \mathcal Y^\mu}\mathcal L= \Delta_{Y^{[n]}}\]

\end{example}

\section{Logarithmic Weinstein category }\label{log section}

In this section, we extend the holomorphic Weinstein category to include log schemes.
Let $\mathrm{X}=(X, D_{\mathrm X})$ indicate the log scheme corresponding to a smooth variety $X$ with a simple normal crossing divisor $D_{\mathrm X}$. We will refer to such log schemes as smooth, noting that this is significantly stronger than requiring that $\mathrm X$ be log smooth. Geometrically, $\mathrm X$ can be thought of as the non-compact smooth variety $X\setminus D_{\mathrm X}:=\rm X^\circ $ with some kind of smooth structure at infinity. We shall see that lagrangian subvarieties of $\mathrm X$ correspond to lagrangian subvarieties of $\rm X^\circ$, however, to define composition of lagrangian correspondences, we will need a different intersection theory to account for intersections  at infinity --- geometrically, this intersection theory can be realised by translating subvarieties by generic smooth  global sections of the logarithmic tangent bundle.


  The logarithic structure on $\mathrm{X}$ is encoded in the sheaf $M_{\mathrm{X}}$ of holomorphic functions that are invertible when restricted to $X^\circ =X\setminus D_{\mathrm X}$. The operation of multiplication makes $M_{\mathrm{X}}$ a sheaf of monoids; the formal structure of $\mathrm{X}$ consists of $M_{\mathrm X}$ together with the obvious homomorphism $\alpha:M_{\mathrm X}\longrightarrow O_{X}$.  
  
   Locally, $\mathrm X$ is isomorphic to a Zarski open subset of $\mathbb C^{n}$ with the divisor defined by  $z_1z_2\dotsb z_{n}=0$, and locally, each element of $M_{\mathrm{X}}$ consists of a non-vanishing algebraic function times a monomial in $z_1,\dotsc, z_n$.   A local basis for the (holomorphic) log tangent bundle of $\mathrm X$ is provided by $z_k\partial _{z_k}$, and a basis for the (holomorphic) logarithmic cotangent space is given by the meromorphic $1$-forms $dz_k/z_k$. A  holomorphic symplectic form on $\mathrm X$ is a closed meromorphic $2$--form on $X$ locally given by
\[\omega= \sum h_{ij}\frac{dz_i}{z_i}\wedge \frac{dz_j}{z_j}\]
where the coefficient functions $h_{ij}$ are algebraic and satisfy the non-degeneracy condition  that the top exterior power of $\omega$ is an invertible function times $\frac {dz_1}{z_1}\wedge \dotsb \wedge \frac{dz_{n}}{z_{n}}$. More formally, we have the following definition

\begin{defn}[logarithmic holomorphic symplectic form] A holomorphic symplectic form on $\rm X$ is is a closed holomorphic $2$--form $\omega$ such that $v\mapsto \iota_v\omega$ is an isomorphism between the holomorphic logarithmic tangent and cotangent spaces.\end{defn}

\begin{example}Let $(X,D)$ be an even dimensional toric variety with its toric boundary divisor. In toric coordinates,  $\omega =\sum _k\frac {dz_{2k}}{z_{2k}}\wedge \frac{dz_{2k+1}}{z_{2k+1}}$ is a holomorphic symplectic form. Lagrangian subvarieties of $(X,D)$ correspond to lagrangian subvarieties of $\rm X^\circ=(\mathbb C^*)^{2n}$, and our different intersection theory can be realised by translating subvarieties generically using the $(\mathbb C^*)^{2n}$ action.
\end{example}
Our most important examples will be products of log Calabi--Yau surfaces, as these will arise naturally as evaluation spaces for the moduli stack of curves in log Calabi--Yau 3-folds.

\begin{example} A standard log point $\rm p$ is a point, $\Spec \mathbb C$, with the monoid $\mathcal M_{\rm p}=(\mathbb C^*,\times)\times(\mathbb N,+)$ and  homomorphism $\alpha:\mathcal M_{\rm p}\longrightarrow \mathbb C$ defined as follows. 
\[\alpha(c,n):=\begin{cases}c\text{ if }n=0
\\ 0 \text{ if }n>0\end{cases}\]
The standard log point is not log smooth, but can be obtained by restricting $(\mathbb C,0)$ to the point $0\subset \mathbb C$. This standard log point should be regarded as $1$--dimensional. The vector field $z\partial_z$ restricts from $(\mathbb C,0)$ to $\rm p$, and generates a free action of $\mathbb C^*$ on $\rm p$, which acts on $M_{\rm p}$ by $t*(c,n)=(t^nc,n)$.
\end{example}

\begin{remark} \label{stratumlogstructure} A closed stratum of $(X,D_{\mathrm X})$ is the intersection of some collection of $k$ components $D_1,\dotsc, D_k$ of $D_{\mathrm X}$. Each such closed stratum $S$ is smooth, with a simple normal crossing divisor $D_S$ comprised of the intersection of $S$ with other strata of $D_{\mathrm X}$. This closed stratum also inherits a log structure from $\mathrm X$, but this log structure is not log smooth, and not the log structure from $(S,D_S)$. Instead, $\mathrm S\subset \mathrm X$ should be thought of as a $\mathrm p^k$--bundle  
\[\begin{tikzcd}\mathrm S\rar\dar & \mathrm X 
\\ (S,D_S)\end{tikzcd}\] encoding the extra information from the normal bundles  to $D_1,\dotsc, D_k$. This is because locally, when $D_j$ is the vanishing locus of $z_j$,  the sections $z_1,\dotsc, z_k$ of  $M_{\mathrm X}$ restrict to independent sections of $M_{\mathrm S}$. Moreover, $M_{\mathrm S}$ is the monoid created from $M_{(S,D_S)}$ by adding these $k$ extra sections. Accordingly, the fibres of the bundle $\mathrm S\longrightarrow (S,D_S)$ are log points $\rm p^k$ with monoid  $M_{\rm p^k}=(\mathbb C^*,\times)\times(\mathbb N,+)^k$. 

There is a well--defined $(\mathbb C^*)^k$ action on $\mathrm S$ locally given by scaling these $k$ extra sections of $M_{(S,D_S)}$;  this $(\mathbb C^*)^k$--action is the flow generated by $k$ globally defined sections of $T\mathrm S$ locally given by the restriction of   $z_i\partial _{z_i}$ to $\mathrm S\subset \mathrm X$.       As a log scheme, $\mathrm S$ has the same dimension as $\mathrm X$, and a local basis for the logarithmic cotangent space is given by the restriction of the basis $\frac{dz_j}{z_j}$ from $\mathrm X$. 

 The restriction of a holomorphic symplectic form $\omega$ to $\mathrm S\subset\mathrm X$ is invariant under the $(\mathbb C^*)^k$--action, as are all differential forms on $\mathrm S$. Accordingly, $\iota_{z_i\partial _{z_i}}\omega$ is a nonvanishing closed holomorphic $1$--form on $\mathrm S$ for $i=1,\dotsc, k$. Moreover, we can create a non-vanishing holomorphic volume form on $\mathrm S$ from the wedge product of these 1-forms with a power of $\omega$. When $\omega$ vanishes on some subvariety $V$ of $S$, each of these $1$-forms $\iota_{z_i\partial_{z_i}}\omega$ also vanishes on $V$.

\end{remark}

We need a homology and intersection theory for $\mathrm X$ that properly reflects transverse fibre products within the category of (fine, saturated) log schemes --- this is not usual intersection theory on $X$. In particular, properties we want are as follows:

\begin{enumerate}
\item Given a proper map $f:\rm X\longrightarrow \rm Y$ from a log smooth domain $\rm X$,  $f_*[\rm X]$ should be defined within our homology theory. 

\item If $g:\rm Z\longrightarrow \rm Y$ is transverse to $f$ (with transversality defined using the log tangent bundle), we obtain a fibre product diagram.
\begin{equation}\label{fpd}\begin{tikzcd} \rm X\times_{\rm Y}\rm Z\rar{f'}\dar{g'} & \rm Z\dar{g}
\\ \rm X\rar{f } & \rm Y\end{tikzcd}\end{equation}
Moreover, $\rm X\times_{\rm Y}\rm Z$ is log smooth if $\rm Z$ is, and $f'$ is proper if $f $ is.  In a situation such as this, $g^!(f_*[\rm X])$ should be defined, and equal to $f'_*[\rm X\times_{\rm Y}\rm Z]$. Moreover, in the case that $g$ is also proper,  the intersection of $f_*[\rm X]$ with $g_*[\rm Z]$ should be defined, and equal to $f_*g'_*[\rm X\times_{\rm Y}\rm Z]$.
\item Furthermore, if $g$ is a submersion, and $h:\rm N\longrightarrow \rm X$ is a proper map from a log smooth domain, the following fibre product diagram
\[\begin{tikzcd} \rm N\times_{\rm Y} \rm Z \rar \dar & \rm X\times_{\rm Y}\rm Z\rar{f'}\dar{g'} & \rm Z\dar{g}
\\ \rm N\rar{h} & \rm X\rar{f } & \rm Y\end{tikzcd}\]
implies that 
\[f'_*(g')^!(h_*[N])=g^!f_*(h_*[N])\]
so we should expect that, in the case of fibre product diagrams (\ref{fpd}) 
\begin{equation}f'_*\circ (g')^!=g^!\circ f_* \ .\end{equation}

\end{enumerate}

The analogous compatibility with fibre product diagrams was essential for proving that the star product is associative. Moreover, a homology theory satisfying this property is also essential for gluing Gromov--Witten invariants.

We can construct homology theories satisfying the above properties using either the log Chow ring or refined differential forms; see \cite[Section 9]{dre}. 

The above properties become alarming in light of particularly nice submersions called logarithmic modifications. An algebraic map $(X',D_{\rm X'})\longrightarrow (X, D_{\rm X})$ is a logarithmic modification when it is proper and a bijection between $(\rm X')^\circ$ and $\rm X^\circ$. For example, if $(X,D)$ is toric, any toric blowup gives a logarithmic modification. Given any logarithmic modification $\pi:\rm X'\longrightarrow \rm X$, the following is a fibre product diagram
\[\begin{tikzcd}\rm X'\rar{id}\dar{id} & \rm X'\dar{\pi}
\\ \rm X'\rar{\pi} & \rm X\end{tikzcd}\]
so, assuming the above properties hold, we obtain that $\pi^!\circ\pi_*=id_*\circ id^!$. Under the reasonable assumption that $id^!$ is the identity, this implies that the homology of $\rm X$ must contain the homology of every logarithmic modification $X'$.

\begin{remark}\label{logresolution} If a subvariety $V\subset X$ intersects $\rm X^\circ$,   we can resolve the singularities of $V$ to obtain a proper surjective map  $\pi:N\longrightarrow V$ with $N$ smooth; \cite{hironaka}. We can even choose this resolution so that $\pi^{-1}D_{\mathrm X}\subset N$  is a simple normal crossing divisor, so $\mathrm N=(N,\pi^{-1}D_{\mathrm X})$ is log smooth.  In particular, $V$ is the image of a proper  map $\mathrm N\longrightarrow \mathrm X$ from a log smooth domain, and $V$ is a suitable cycle for a homology theory.

On the other hand, when $V$ does not intersect $\rm X^\circ$, it is not the image of a proper map from a log smooth domain. 
When $V$ is contained in some closed codimension--k stratum $\mathrm S\subset \mathrm X$, and intersects the interior of  this stratum, the image of $V$ within $(S,D_S)$ is the image of a proper  map $\mathrm N\longrightarrow (S,D_S)$ with smooth domain. As explained in Remark \ref{stratumstructure},  $\mathrm S$ is actually a $\mathrm p^k$--bundle over $(S,D_S)$, where $\mathrm p$ is the standard log point. Taking fibre products, we get that $V$ is the image of a $\mathrm p^k$--bundle map over a map with smooth domain.

\[\begin{tikzcd} \mathrm p^k\rtimes \mathrm N\dar \rar &  \mathrm S\rar\dar &   \mathrm X\dar
\\ \mathrm N \rar& V\subset (S,D_S) \rar &V\subset X  \end{tikzcd}\]

This is $\mathrm p^k$ bundle is not log smooth, and it is in fact impossible for such a $V$ to be the image of a log smooth map when $k\neq 0$. 
\end{remark}



\begin{defn}[complete subvariety] \label{complete} Say that a   subvariety $V\subset X$ is complete in  $\mathrm X=(X,D_{\mathrm X})$ if $V$ is not contained in $D_{\mathrm X}$.  Say that $V$ intersects $D_{\mathrm X}$ nicely if it is complete and its intersection with each closed stratum is a union of complete subvarieties.\end{defn}

\begin{remark} Complete subvarieties $V$ of $(X, D_{\mathrm X})$ correspond to  subvarieties $V^\circ\subset \rm X^\circ$. In particular, $V^\circ$ is the intersection of $V$ with the Zarski--open subset $\rm X^\circ\subset X$, and the closure $V$ of any  subvariety $V^\circ$ of $\rm X^\circ\subset X$ is a complete subvariety. 
\end{remark}

\begin{defn}[isotropic, lagrangian subvarieties, lagrangian correspondence] Given a holomorphic  symplectic form $\omega$ on $\mathrm X$, a subvariety  $V\subset X$ is isotropic if the restriction of $\omega$ to $V$ vanishes in the sense that $f^*\omega=0$  for any map $f:\mathrm p^k\rtimes\mathrm N\longrightarrow \mathrm X$ from a $\mathrm p^k$--bundle over a connected smooth domain $N$, with image a dense\footnote{We can remove the requirement that the image of $f$ is dense if $V$ intersects the divisor nicely. See Corollary \ref{allpullbacks}. This definition is complicated by the problem that subvarieties of the divisor $D$ do not carry enough logarithmic information. We later solve this problem by working with exploded manifolds; see Definition \ref{exploded isotropic def}. } subset of $V$.   Say that $V$ is lagrangian if it is isotropic and $2\dim V=\dim X$.

The group of lagrangian cycles, $\Lag{\mathrm X}^-$ is the $\mathbb Z$--module generated by  lagrangian subvarieties of $\mathrm X$. The group of lagrangian correspondences $\Lag{\mathrm X, \rm Y}$ is the $\mathbb Z$ module generated by lagrangian subvarieties of $\rm X^-Y$ with proper projection to $\mathrm X^-$.
\end{defn}

\begin{remark}In the above definition $\dim V$ refers to the dimension of $V$ as a subvariety of $X$ without considering any log structure. As a consequence, $V$ must be complete in order to be lagrangian. To see this, note that when $V$ is contained in a codimension $k$ stratum of $S$, the regular locus of $V$ intersected with the interior of $S$ has a canonical log structure pulled back from $\mathrm S\subset \mathrm X$, and its dimension as a log scheme is $\dim V+k$. Moreover, $\omega$ can only vanish on $V$ if this logarithmic dimension  is at most half the dimension of $\rm X$.
\end{remark}

\begin{remark}\label{densecheck}
To check that a subvariety is isotropic, it suffices to check that $\omega$ vanishes on a dense subset of the regular locus of $V$.   This then implies that given any map $f:\mathrm p^k\rtimes\mathrm N\longrightarrow \mathrm X$ with image dense in $V$, $f^*\omega$ vanishes on a dense subset of $\mathrm p^k\rtimes\mathrm N$, and hence vanishes everywhere.  If $V\subset (X,D)$ is complete, it follows that $V$ is isotropic if and only if the restriction $V^\circ\subset X^\circ$ is isotropic.
\end{remark}

The following is an interesting example of a lagrangian correspondence, realising a non-toric blowup as a morphism in our holomorphic Weinstein category.

\begin{example} \label{ntbc} Let $(X,D)$ be a 2-dimensional toric manifold with its toric boundary divisor. Choose a point $p$ in the interior of a smooth component of $D$, and let $\pi:Bl_p X\longrightarrow X$ be the projection from the (non-toric) blowup of $X$ at $p$. There are two possible divisors we could use to define a log structure on $Bl_p X$, the inverse image $\pi^{-1}D$ of $D$, or the strict transform $D'$ of $D$, which removes the component $\pi^{-1}(p)$ from $\pi^{-1}D$. We get the following diagram of logarithmic maps
\[\begin{tikzcd} & (Bl_p X,\pi^{-1}D)\ar{dr}{\pi_2}\ar{dl}[swap]{\pi_1}
\\ (X,D) && (Bl_p X,D') \end{tikzcd}  \]
however, this can not be completed to a commutative triangle of logarithmic maps. 
Both $(X,D)$  and $(Bl_p X,D')$ are log Calabi--Yau surfaces: In toric coordinates on $X$, $\omega= \frac{dz_1}{z_1}\wedge \frac{dz_2}{z_2}$ is a holomorphic volume form, and hence a holomorphic symplectic form. Moreover there exists a unique holomorphic volume form $\omega_2$ on $(Bl_p X,D')$ such that $\pi_1^*\omega=\pi_2^*\omega_2$. We therefore get a holomorphic lagrangian correspondence $b\in \Lag{(X,D),(Bl_p(X),D')}$ as the image of 
\[(\pi_1,\pi_2):(Bl_p X,\pi^{-1}D)\longrightarrow ( X,D)^-  (Bl_p X,D') \ .\]

\end{example}

\begin{example} Given a $\mathbb C^*$ action on $\rm X$ preserving the holomorphic symplectic form, and an isolated critical point $p$, the unstable manifold of $p$ consists of the points $x$ such that $p=\lim _{t\to -\infty} t*x$. The closure of the unstable manifold of $p$ is a lagrangian subvariety, and similarly, the closure of the stable manifold of $p$ is a lagrangian subvariety. A simple example of such a situation is given by the non-toric blowup of $\mathbb CP^2$ at $[1,1,0]$. The $\mathbb C^*$ action on the last coordinate has a fixed point. The stable manifold of this fixed point is the exceptional locus, and the inverse image of the points $[1,1,z_3]$ breaks up into the stable and unstable manifolds of this point.

 More generally, given a $\mathbb C^*$--action preserving the holomorphic symplectic form, each component of the fixed point set has a holomorphic symplectic structure. Given any lagrangian subvariety $L$ of this fixed point set, we can define a lagrangian subvariety as the closure of the set of points $x$ such that $\lim_{t\to -\infty}t*x\in L$.
\end{example}

Note that a complete subvariety $V$ of $\mathrm X$ is determined by its intersection $V^\circ $ with $\rm X^\circ $.   By blowing up $\mathrm X$ at closed strata, we can obtain many other compactifications $(X',D_{\rm X'})$ of $\rm X^\circ$, and the blowdown maps $\rm X'\longrightarrow \rm X$ are examples of logarithmic modifications. Logarithmic modifications are submersions, in the sense that their derivatives are surjective when we use the logarithmic tangent space.  Given a holomorphic symplectic form $\omega$ on $\mathrm X$, the pullback of $
\omega$  is also a holomorphic symplectic form $\pi^*\omega$ on $\mathrm X'$. Moreover, as complete subvarieties of $\mathrm X$ correspond to subvarieties of $X\setminus D=X'\setminus D_{\mathrm X'}$, compete subvarieties of $\mathrm X$ correspond to complete subvarieties of $\mathrm X'$.  
\begin{defn}[pullback of a complete subvariety] Given a complete subvariety $V$ of  $\mathrm X$ and a proper submersion $f:\mathrm Y\longrightarrow\mathrm X$, define the pullback $f^*V$ of $V$ to be the closure of $f^{-1}(V^\circ)$.
\end{defn}
In general, $f^*V$ is smaller than $f^{-1}V$.

 \begin{remark}\label{propermod} If $V \subset \rm X^-  \rm Y$ is a complete subvariety, and $\pi:\rm M\longrightarrow \rm X^-  \rm Y$ is a logarithmic modification, $V$ is proper over $\rm X$ if and only if $\pi^*V$ is proper over $\rm X$. From Remark \ref{densecheck}, we also have that $\pi^*V$ is a lagrangian subvariety if and only if $V$ is.
 \end{remark}
 
 \begin{example} If $\rm Y$ is a log Calabi--Yau 2-fold, the Hilbert scheme of points $\rm Y^{[n]}$ still makes sense and has a natural holomorphic symplectic structure, \cite{logHS,logDT} but has a log structure that is only defined up to logarithmic modification, and is a Deligne--Mumford stack locally modelled on the quotient of  $(\mathbb C,0)^{2n}$ by a finite group action that is free on the interior. Moreover, the interior of $\rm Y^{[n]}$ is $(\mathrm Y^\circ )^{[n]}$, so lagrangian subvarieties $L^{\mu}\Sigma$, $\Delta^\mu$ and $\ell^\mu$ from examples \ref{HS2}, \ref{HSD}, and \ref{HSD2} can be defined by taking the closure of the corresponding lagrangian subvarieties of $(\mathrm Y^\circ)^{[n]}\subset \rm Y^{[n]}$.
 \end{example}

\begin{lemma}\label{lmodification} Given a complete subvariety $V$ of $\mathrm X$, there exists a logarithmic modification $\pi: \mathrm X'\longrightarrow \mathrm X$ such that the pullback $\pi^*V$ of $V$  intersects $D_{\mathrm X'}$ nicely.  Moreover, given any further logarighmic modification of $\mathrm X'$, the pullback of $\pi^*V$ coincides with the inverse image of $\pi^*V$, and  also  intersects the divisor nicely.
\end{lemma}

We defer the proof of Lemma \ref{lmodification} until Remark \ref{complete connection}, where we use tropical techniques for constructing the required logarithmic modification.

\subsection{Star products using the log Chow ring}

\

The log Chow ring of a smooth log scheme $\mathrm X$ is often defined as the limit of the Chow ring of $X'$ over log modifications of $\mathrm X'\longrightarrow \mathrm X$ with $\mathrm X'$ smooth; see, for example \cite{logarithmicdoubleramificationcycles}. This definition provides a homology theory that obeys  our expectations of compatibility with transverse fibre products in the category of fine saturated log schemes. 

A complete subvariety $V$ of $\mathrm X$ represents a class in the log Chow ring as follows: Choose a smooth log modification $\pi:\mathrm X'\longrightarrow \mathrm X$ such that $V$ intersects the divisor nicely. The class represented by $V$ is then the class in $A_*(X')$ represented by $\pi^*(V)$.  In fact, the log Chow ring is generated by complete subvarieties, because Chow's moving lemma implies that each subvariety in $X'$ is rationally equivalent to a linear combination of varieties that intersect the boundary divisor with the expected dimension. Two complete subvarieties represent the same class if their pullbacks are  rationally equivalent in every log modification of $\mathrm X$ --- in practice, it is enough to check that their pullbacks are rationally equivalent in a log modification where they intersect the divisor nicely.  

Chain level intersection theory in the log Chow ring is not as clean as in the Chow ring of a smooth scheme, because the intersection of complete subvarieties need not be complete. To take the intersection of two complete subvarieties, pull them back to a log modification  $\mathrm X'$ where they intersect the divisor nicely, then take their intersection in the Chow ring of $X'$. As the geometric intersection of these cycles may be in the divisor, we can not always represent this intersection as a sum of complete subvarieties  contained in the geometric intersection. Nevertheless, we will still get a well-defined star product on lagrangian correspondences. 

Complete subvarieties should be regarded as transverse only if they are the image of transverse proper maps from log smooth schemes, where transversality is checked using the log tangent space. In this case, the intersection is always a complete subvariety, and satisfies our expected compatibility with fibre products. When we have transversality, intersections can be determined by the transverse intersection of these subvarieties restricted to $\rm X^\circ\subset \rm X$; note however that transversality can not be checked by restricting to $\rm X^\circ $, as there are non-transverse complete subvarieties that intersect only in $D$.

\begin{lemma}\label{niceisotropic} If an isotropic subvariety $V\subset \mathrm X$ intersects the divisor nicely, then the intersection of $V$ with each stratum of $\mathrm X$ is a finite union of  isotropic subvarieties.
\end{lemma}
\begin{proof} We will prove this lemma by induction on the codimension of  strata $S\subset (X,D_{\mathrm X})$. It holds trivially for the codimension $0$ stratum, so assume that it holds for all strata of codimension $k-1$, and let $S$ be a closed stratum of codimension $k$. As $V$ intersects the divisor nicely,  given any codimension $k-1$ stratum $S'$ containing $S$, each irreducible component of $V\cap S$ is a codimension $1$ stratum of some irreducible component $V'$  of $V\cap S'$, which is an isotropic subvariety, by our inductive hypothesis.  

Following Remark \ref{logresolution}, choose a map $f:\mathrm p^{k-1}\rtimes\mathrm N\longrightarrow \mathrm X$ whose image is $V'$. We have that $f^{-1}(S)\subset \mathrm p^{k-1}\rtimes\mathrm N$ is some union of codimension $1$ strata, and at least one such stratum has image that is dense in our component of $V\cap S$. The restriction of $f$ to such a stratum is a submersion with dense image in our component of $V\cap S$, so the fact that $f^*\omega=0$ implies that $\omega$ vanishes on a dense subset of the regular locus of our component of $V\cap S$, so Remark \ref{densecheck} implies that our component is isotropic. The required result follows by induction.

\end{proof}

\begin{cor}\label{allpullbacks} If an isotropic subvariety $V\subset \mathrm X$ intersects the divisor nicely, then $f^*\omega=0$ for every map $f:\mathrm p^k\rtimes\mathrm N\longrightarrow \mathrm X$ whose image is contained in $V$. 
\end{cor}
\begin{proof} Without loss of generality, $N$ is irreducible, and the image of $f$ is contained in some stratum  $\mathrm S\subset \mathrm X$, and intersects the interior $\mathrm S$. As the inverse image of the interior of this stratum is dense in $\mathrm p^k\rtimes\mathrm N$, we can further assume that the image of $f$ is contained in the interior of $\mathrm S$. By Lemma \ref{niceisotropic}, we have that $V\cap S$ is a union of isotropic varieties, and, as $N$ is irreducible, the image of $f$ must be contained in one of them, $V'$. Using resolution of singularities as in Remark \ref{logresolution}, we can choose a proper map $\pi:\mathrm p^m\rtimes\mathrm M\longrightarrow \mathrm S$ with image $V'$, and $\mathrm M$ smooth. We have that $\pi^*\omega=0$. As with any algebraic map between smooth varieties, there exists an algebraic stratification of $M$ such that the interior of each stratum is regular, and $\pi$ restricted to the interior  is constant rank. This induces a stratification of $V'$ so that the interior of each stratum is regular, and $\pi$ restricted to the inverse image of the interior of each stratum is a submersion. As a consequence, $\omega$ vanishes on the interior of each stratum. The restriction of $f$ to a dense subset of $\mathrm p^k\rtimes\mathrm N$ is sent to the interior of some stratum, and therefore $f^*\omega$ vanishes on this dense subset, and therefore $f^*\omega=0$, as required.

\end{proof}

\begin{lemma}\label{logisotropic} 
Suppose that $\ell_1\subset  \rm A^-  \rm X$ and $\ell_2\subset \rm X^-  \rm B$ are lagrangian subvarieties and that  the projection of  $\ell_1$ to $A^-$ is proper.

Let  $\pi:\mathrm M\longrightarrow \rm A^-  \mathrm X  \mathrm B$ be a smooth logarithmic modification, and consider the following maps.
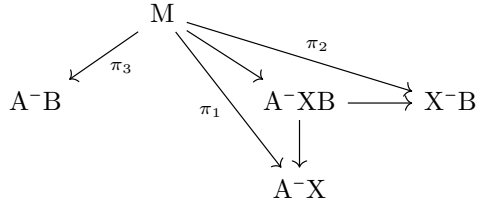
\[\begin{tikzcd} & \ar{dl}{\pi_3} \mathrm M\ar{dr}\ar{drr}{\pi_{2}} \ar{ddr}[swap]{\pi_{1}} 
\\ \rm A^-  \rm B & & \rm A^-  \mathrm X  \mathrm B\dar \rar & \mathrm X^-  \rm B
\\ & & \rm A^-  \mathrm X\end{tikzcd}\]

 Suppose that within $\mathrm M$,  $\pi_{i}^{*}\ell_{i}$  intersects the divisor nicely.  Then
\[\pi_{3}\left(\pi_{1}^{*}\ell_{1}\cap \pi_2^*\ell_2\right) \subset \mathrm A^-  \rm B\]
is a finite union of  isotropic subvarieties, possibly together with some extraneous (not-complete) subvarieties of dimension strictly less that $\frac 12 \dim (\rm A  \rm B)$.

Moreover, if $\ell_2$ is proper over $X$, each of these subvarieties is proper over $\rm A^-$.

\end{lemma}
\begin{proof} As in the proof of Lemma \ref{isotropic},  $\pi_{3}\left(\pi_{1}^{*}\ell_{1}\cap \pi_2^*\ell_2\right) \subset \mathrm A^-  \rm B$ is a finite union of subvarieties of $\rm A  \rm B$, because $\ell_1$ is proper over $\rm A^-$.  Moreover, if $\ell_2$ is also proper over $X$,  we have that that $\pi_{1}^{*}\ell_{1}\cap \pi_2^*\ell_2$ is proper over $\rm A^-$, so the image of this in $\rm A^-  \rm B$ is a finite union of subvarieties, each of which is proper over $\rm A^-$.

 Lemma \ref{lmodification} implies that we can choose a logarithmic modification $\rm M_1$ of $\rm A^-  \rm X$ such that the pullback of $\ell_{1}$ intersects the divisor nicely, and similarly, we can choose a logarithmic modification $\rm M_2$ of $\rm X^-  \rm B$ so that the pullback of $\ell_2$ intersects the divisor nicely. By making a further logarithmic modification of $\mathrm M$, we can also ensure that $\pi_{1}$ and $\pi_2$ factor through these logarithmic modification. Lemma \ref{lmodification} implies that, for this further logarithmic modification, pullbacks coincide with inverse images, so the image of our intersection in $\rm A^-  \rm B $ is unaffected by this further logarithmic modification.

Given any map $f:\mathrm p^k\rtimes\mathrm N\longrightarrow \mathrm M$ with image in $\pi_{1}^{*}\ell_{1}\cap \pi_2^*\ell_2$, Corollary \ref{allpullbacks} implies that the pullback of both $\omega_{\mathrm X}-\omega_{\rm A}$ and $\omega_{\mathrm B}-\omega_{\mathrm X}$ vanish, so $f^*\omega_{\rm A}=f^*\omega_{\rm B}$, and in particular, $\omega_{\rm A}=\omega_{\rm B}$  on the regular locus of  $\pi_{1}^{*}\ell_{1}\cap \pi_2^*\ell_2$. It follows that $\omega_{\rm A}=\omega_{\rm B}$  on a dense subset of every complete irreducible component of $\pi_{3}\left(\pi_{1}^{*}\ell_{1}\cap \pi_2^*\ell_2\right)$, so all complete irreducible components are isotropic. It remains to check that any extraneous irreducible components that are not complete also have dimension strictly less than $\frac 12\dim (\rm A  \rm B)$. 

Let $V$ be an extraneous irreducible component  of $\pi_{1}^{*}\ell_{1}\cap \pi_2^*\ell_2$ whose image in $\rm A^-  \rm B$ does not intersect the interior of $\rm A^-  \rm B$. The complication we need to deal with is as follows: $\pi_3V$ might not be isotropic in $\rm A^-  \rm B$, even though it is the image of an isotropic subvariety of some logarithmic modification of $\rm A^-  \rm B$. At each regular point of $V$, the image of the logarithmic tangent space of $V$ is isotropic within the logarithmic tangent space of $\rm A^-  \rm B$, and therefore has dimension bounded by $\frac 12\dim(\rm A  \rm B)$. As the image of $V$ does not intersect the interior of $\rm A^-  \rm B$, the projection of this image to the usual tangent space of the smooth scheme $A  B$ has nontrivial kernel. It follows that the image within the tangent space of the smooth scheme has dimension strictly less than $\frac 12\dim (\rm A  \rm B)$, and therefore, $\dim \pi_{3}V<\frac 12 \dim (\rm A  \rm B)$.

\end{proof}

As a corollary of Lemmas \ref{lmodification} and \ref{logisotropic}, we get the following.

\begin{cor}[star product using log Chow ring] Given  lagrangian correspondences $\ell_1\in \Lag{  \rm A, \rm X}$ and $\ell_2\in\Lag{ \rm X  \rm B, \rm C}$,  there exists a unique lagrangian correspondence 
$\ell_{1}\star_{\rm X} \ell_2 :=\sum_i m_iV_i\in \Lag{\rm A \rm B,C}$ satisfying the following condition. Given any smooth logarithmic modification $\pi:\mathrm M\longrightarrow \rm A^-  \mathrm X  \mathrm B  C$, 
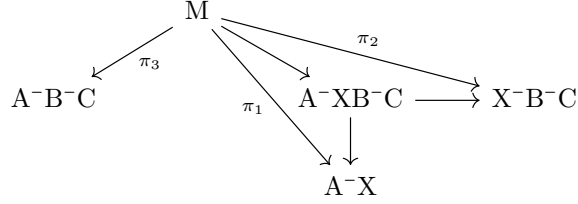
\[\begin{tikzcd} & \ar{dl}{\pi_3} \mathrm M\ar{dr}\ar{drr}{\pi_{2}} \ar{ddr}[swap]{\pi_{1}} 
\\ \rm A^- \rm B^- C & & \rm A^- \mathrm X \mathrm B^- C\dar \rar & \mathrm X^- \rm B^- C
\\ & & \rm A^- \mathrm X\end{tikzcd}\]
 with the property that within $\mathrm M$,  $\pi_{i}^{*}\ell_{i}$  intersects the divisor nicely,  we have
 \[[\ell_1\star_{\rm X}\ell_2]= (\pi_3)_*(\pi_1,\pi_2)^![\ell_1\times \ell_2]\] 

%
%

\end{cor}

\begin{example}Let $b\in\Lag{(X,D),(Bl_p X,D')}$ be the lagrangian correspondence from Example \ref{ntbc}. Then, $b^\dagger$ is a right inverse to $b$
\[b\star_{(Bl_p X,D')}b^\dagger=\Delta_{(X,D)}\]
but 
\[b^\dagger\star_{(X,D)}b=\Delta_{(Bl_p X,D')}+E\times E\  \in \Lag{(Bl_p X,D'),(Bl_p X,D')}\] where  $E\in \Lag{(Bl_p X,D')}$ the exceptional sphere. 

\end{example}

This logarithmic star product also has a more geometric interpretation in terms of exploded manifolds. We defer proving the associativity of this logarithmic  star product until we have this more geometric interpretation. Sometimes, however, the logarithmic star product can be interpreted concretely and simply by restricting to the interior of our log schemes.  The restriction of $\ell\in \Lag{\rm A, \rm X}$ need not be a lagrangian correspondence in $\Lag{\rm A^\circ,\rm X^\circ }$, because it may not have proper projection to $\rm A^\circ$. In the case that the restriction of $\ell_1$ is proper over $A^\circ$,  the composition of $\ell_1$ and $\ell_2$ is determined by the composition of the restrictions.

\begin{lemma} \label{restricted computation} Suppose that $\ell_1\in \Lag{\rm A, \rm X}$ has restriction 
$\ell_1^\circ\in \Lag{\rm A^\circ,\rm X^\circ }$. Then, for any lagrangian cycle $\ell_2$ in  $\Lag{\rm X, \rm B}$, we have
\[(\ell_1\star_{\rm X}\ell_2)^\circ= \ell_1^\circ\star_{X^\circ } \ell_2^\circ\]
\end{lemma}

\begin{proof} By assumption, $\ell_1^\circ $ has proper projection to $\rm A^\circ$.  In particular, this implies that $\ell_1^\circ$ restricted to the inverse image of a compact subset $K_1$ of $\rm A^\circ$ is compact, and hence agrees with $\ell_1$ restricted to the inverse image of $K_1$. It follows that over $K_1\times B^\circ $ the geometric intersection of $\ell_1$ with $\ell_2$ coincides with the geometric intersection of $\ell_1^\circ$ and $\ell_2^\circ$. As the virtual intersection is local to each component of the geometric intersection,  it follows that the virtual intersection also coincides over $K_1\times B^\circ $, so $\ell_1\star_{\rm X}\ell_2$ agrees with $\ell_1^\circ \star_{\rm X^\circ}\ell_2^\circ $ over $K_1\times B^\circ$. It follows that $\ell_1\star_{\rm X}\ell_2$ restricted to $\rm A^\circ\times B^\circ$ agrees with $\ell_1^\circ\star_{\rm X^\circ }\ell_2^\circ$, and therefore $(\ell_1\star_{\rm X}\ell_2)^\circ= \ell_1^\circ\star_{X^\circ } \ell_2^\circ$, because $\ell_1\star_{\rm X}\ell_2$ is a weighted sum of complete subvarieties, so has no irreducible components in the complement of $A^\circ   B^\circ$.

\end{proof}

\begin{example} Lemma \ref{restricted computation} implies that the computations from examples \ref{HSDc} and \ref{HSD2c} are also valid for $\rm Y$ a log Calabi--Yau surface. 
\end{example}

\section{Exploded manifold perspective}\label{exploded section}

To understand lagrangian cycles in log schemes as subsets, it is helpful to use exploded manifolds. 
See \cite{scgp} for a short introduction to exploded manifolds, and \cite{elc} for an explanation of the relationship between log schemes and exploded manifolds. In terms of log schemes, an exploded point  $\mathrm p^+$ is a point with the log structure given by the monoid $\mathbb C^*\e{[0,\infty)}:=(\mathbb C^*,\times)\times ([0,\infty),+)$,  and the homomorphism
\[ \mathbb C^*\e{[0,\infty)}\longrightarrow \mathbb C\]
\[  c\e a\mapsto \totl{c\e a}:= \begin{cases} c\text{ if }a=0
\\ 0 \text{ if }a>0\end{cases}  \]
called the smooth part homomorphism. In fact, the monoid structure on $\mathbb C^*\e{[0,\infty)}$ extends to a semiring structure on $\mathbb C\e{[0,\infty)}$ such that the above is a homomorphism of semirings. The addition in this semiring is defined by
\[c_1\e{a_1}+c_2\e{a_2}=\begin{cases}c_1\e{a_1}\ \ \ \ \ \ \ \ \text{ if } a_1<a_2
\\ (c_1+c_2)\e{a_1}\text{ if }a_1=a_2
\\ c_2\e{a_2}\ \ \ \ \ \ \ \  \text{ if }a_1>a_2\end{cases}\]
An exploded point in a log scheme $\mathrm X$ is a map $\mathrm p^+\longrightarrow \mathrm X$ of log schemes.

 For example, the set of exploded points in a standard log point $\mathrm p=(\text{Spec} \mathbb C,\mathbb C^*\times\mathbb N)$  is $\mathbb C^*\e{(0,\infty)}$. The $\mathbb C^*$--action from Remark \ref{stratumlogstructure} is the obvious free $\mathbb C^*$ action on this set of points, and the log tangent space corresponds to the  tangent space of this disjoint union of infinitely many copies of  $\mathbb C^*$. 
 
 When we explode $\mathrm X$, we can interpret a section of $M_{\mathrm X}$ as a $\mathbb C^*\e{[0,\infty)}$--valued function on the set of exploded points in $\mathrm X$. The explosion\footnote{When $\rm X$ is not an explodable log scheme, this version of the explosion of $\mathrm X$ may forget some structure of $\rm X$ --- for example, it forgets any non-reduced structure in the underlying scheme, and any non-saturated structure in the sheaf of monoids $M$.}, $\expl \rm X$ of a log scheme is the set 
 \[\expl \mathrm X:=\hom(\mathrm p^+,\rm X)\] of exploded points in $\mathrm X$, with the topology induced from the underlying scheme $X$ using the natural map $\expl \mathrm X\longrightarrow X:=\totl{\expl X}$, and  a sheaf of $\mathbb C^*\e{\mathbb R}$--valued functions, obtained from $M_{\mathrm X}$ by allowing inverses and multiplication by elements of $\mathbb C^*\e{\mathbb R}$. The explosion $\expl \rm X$ makes sense for any log scheme, but is only an exploded manifold when $\rm X$ is locally isomorphic to a stratum of something log smooth\footnote{A log scheme is log smooth if it is locally isomorphic to an open subset of a toric scheme with defining sheaf of   monoids the sheaf of holomorphic functions that are non-vanishing off the toric boundary divisor. So,  it is locally isomorphic to an open subset of $\Spec (Q\longrightarrow\mathbb C[Q])$ for $Q$ a toric monoid; see \cite[Definition 1.1.8]{Ogus_2018}. }; following \cite[Definition 4.1]{elc}, we call such log schemes explodable. 
 
 \begin{defn}[explodable log scheme] A log scheme is explodable if it is locally isomorphic to an open subset of $\Spec(Q\longrightarrow \mathbb C[Q/Q'])$ for $Q$ a toric monoid, and $Q'\subset Q$ an ideal.
 \end{defn}
  
It is easy to check that $\expl$ is a functor from the category of log schemes over $\mathbb C$ to the category of abstract exploded spaces \cite[Definition 3.1]{iec}, which is a certain category of topological spaces with sheaves of $\mathbb C^*\e{\mathbb R}$--valued functions. In particular, the explosion of a map  $f:\rm X\longrightarrow \rm Y$ of log schemes induces a map $\expl f:\expl \rm X\longrightarrow \expl \rm Y$.

\begin{defn}[logarithmically proper]  A logarithmically proper subset $\ex V\subset \expl \rm X$ is the set of points in the image of the explosion of some proper map $f:\mathrm Y\longrightarrow\mathrm X$ where $\rm Y$ is explodable. It is isotropic if $f^*\omega=0$. Given a map $\pi:\rm X\longrightarrow \rm Z$,  say that  $\ex V$ is proper over $\expl \rm Z$ if $\pi\circ f$ is proper.
\end{defn}

\begin{defn}[lagrangian subvariety, lagrangican cycles] A lagrangian subvariety is an isotropic, logarithmically proper subset $\ex V\subset \expl \rm X$ that is the image of the explosion of a generically injective map from an irreducible smooth \footnote{Requiring that $\ex V$ is the image of the explosion of a proper map $f$ from a smooth log scheme ensures that $\expl f$ is complete, which is analogous to proper in the category of exploded manifods. See \cite[Definition 3.15]{iec}. It also ensures that $\totl{V}\subset  X=\totl{\expl X}$ is complete in the sense of Definition \ref{complete}.} log scheme $\rm Y$, with $\dim Y=\frac 12\dim \rm X$. The group of Lagrangian cycles $\Lag{\expl X,\expl Y}$ is the $\mathbb Z$--module generated by lagrangian subvarieties of $\expl X^-  \expl Y$, whose projection to $\expl X^-$ is proper. 
\end{defn}

Using resolution of singularities, the image of any proper map from a fs log scheme\footnote{A log scheme $\rm X$ is fs if the ghost sheaf $M/O^*$ is a sheaf of finitely generated, saturated monoids, and $\rm X$ is coherent in the sense that it admits standard logarithmic charts. See \cite[Section 2]{kk} or \cite[Chapter 2]{Ogus_2018}. } is also a logarithmically proper subset. This is useful because, in the category of fs log schemes, all fibre products exist; see \cite[Section 2.4]{Ogus_2018}. So, we can take pullbacks and finite intersections of logarithmically proper sets.

\begin{prop}\label{explodable resolution} Suppose that $\mathrm X$ is a fs log scheme. Then, there exists a proper map $\pi:\mathrm X_{res}\longrightarrow \mathrm X$ from an explodable log scheme $\mathrm X_{res}$ such that $\expl \pi$  is surjective.
\end{prop}

\begin{proof} 
Let $X_1$ be  the reduced part of the underlying scheme $X$ of $\rm X$, and let $\rm X_1$ be $X_1$ with the pulled back log structure. All maps $p^+\longrightarrow \rm X$ factor through $\rm X_1$.   
 The fs log structure on $\mathrm X_1$ induces a stratification of $X_1$, with the ghost sheaf\footnote{The ghost sheaf $\bar M_{\rm X_1}$ is the quotient of $M_{\rm X_1}$ by the invertible functions  $\mathcal O^*(X_1)\subset M_{\rm X_1}$. } of the log structure constant on each stratum. Let $ X_2$ be the normalisation of  the disjoint union of the closure of each stratum, and let $\rm X_2$ be the corresponding log scheme with the pulled back log structure. Now we have a proper map $\pi_2:\rm X_2\longrightarrow \rm X$ with the property that each point in $\expl X$ is in the image of $\expl \pi_2$ restricted to the interior strata of $\rm X_2$.

  Let $\mathrm X_0$ be the log scheme with underlying scheme $X_2$, but sheaf of monoids $M_{\mathrm X_0}\subset M_{\mathrm X_2}$ comprised of elements of $M_{\mathrm X_2}$ sent to generically nonzero functions on $X_2$. So, $\mathrm X_0$ is normal, reduced, and generically  smooth. Applying logarithmic resolution of singularities \cite[Theorem 1.2.4]{ATW2020}, we obtain  a proper  map
\[\mathrm X'_{res}\longrightarrow \mathrm X_0\]
where $\mathrm X'_{res}$ is log smooth. As $\mathrm X_2$ is normal with fs log structure,  the sheaf $M_{\mathrm X_2}/M_{\mathrm X_0}$ is a locally constant sheaf of toric monoids. As such, the map $\rm X_2\longrightarrow \rm X_0$ can be thought of as a fibration with fibres explodable log points. In particular, taking a fibre product, we obtain our desired explodable log scheme $\mathrm X_{res}$.

\[\begin{tikzcd}\mathrm X_{res}\dar \rar & \mathrm X_2\dar
\\ \mathrm X'_{res}\rar & \mathrm X_0\end{tikzcd}\]

As $\mathrm X_{res}\longrightarrow \mathrm X_2$ is proper, we get that $\mathrm X_{res}\longrightarrow \mathrm X$ is proper. As $\mathrm X_{res}\longrightarrow \mathrm X_2$ is surjective on underlying schemes,  and has the pullback log structure when restricted to the inverse image of the interior strata of $X_2$, we get that the explosion of this map is surjective onto the explosion of the interior strata of $\mathrm X_2$.  As each point in $\expl X$ is in the image of $\expl \pi_2$ restricted to the interior strata of $\rm X_2$, we therefore
obtain that $\mathrm X_{res}\longrightarrow \mathrm X$ is proper, and its explosion is surjective.

\end{proof}

\begin{lemma}\label{fp comparison} The explosion of the fibre product of fs log schemes is the fibre product of the explosions
\[\expl\left(\rm A\times_B C\right)=\rm \expl  A\times_{\expl B}\expl C\ ,\]
where on the righthand side, we take the fibre product in the category of sets. 
\end{lemma}
\begin{proof} 
This lemma follows from the observation that the fibre product in the category of fs log schemes satisfies the usual universal property for fibre products in the category of saturated log schemes. As $\mathbb C^*\e{[0,\infty)}$ is saturated, it follows that a point in $\expl \rm A\times_B C$ is equivalent to a pair of points in $\expl \rm A$ and $\expl \rm{B}$ send to the same point in $\expl \rm C$.   Further details are given below.

 Because any map of $p^+$ into a fs log scheme factors through a map of a fs log point, it is enough to restrict to the case that $\rm A$, $\rm B$ and $\rm C$ are log points. The construction of $\rm A\times_B C$, explained in \cite[section 2.4]{Ogus_2018}, proceeds in three steps. First, we take the fibre product in the category of all log schemes. In this case, the underlying scheme is the fibre product of the underlying schemes (and hence a point), and  the monoid defining the log structure on the fibre product is determined by  a pushout diagram of monoids
\[\begin{tikzcd} P & M_{\mathrm A}\lar
\\ M_{\rm C}\uar & M_{\rm C}\lar\uar \uar 
\end{tikzcd}\]
The homomorphisms defining the log structure on $\rm A, B$ and $\rm C$ induce a homomorphism $\alpha:P\longrightarrow \mathbb C$ defining a log structure on this first fibre product, $\rm X$.  The universal property of fibre products gives that 
\[\expl \rm X=\expl  A\times_{\expl B}\expl C\ ,\]
Next, to obtain the fibre product in the category of integral log schemes, we take $\mathrm X^{int}$; see \cite[Proposition 2.4.5]{Ogus_2018}. We obtain the monoid $P^{int}$ as the image of $P$ under the homomorphism $P\longrightarrow P^{gp}$. If the homomorphism $P\longrightarrow \mathbb C$ factors through $P^{int}$, $\mathrm X^{sat}$ is the log scheme corresponding to this homomorphism $P^{int}\longrightarrow \mathbb C$, and otherwise $\mathrm X^{int}$ is empty. This $\mathrm X^{int}$ clearly has the property that any morphism from an integral log scheme to $\rm X$ factors uniquely through $\mathrm X^{int}$. As  $p^+$ is an integral log scheme, we have
\[\expl \mathrm X^{int}=\expl X \ .\]
Finally, we construct $\rm A\times_B C$ as $(\mathrm X^{int})^{sat}$, following \cite[Proposition 2.4.5]{Ogus_2018}.   Let $P^{sat}$ be the saturation of $P^{int}\subset P^{gp}$, and let $n$ be the number of homomorphisms $\alpha_i:P^{sat}\longrightarrow (\mathbb C,\times)$ such that the following diagram commutes.
\[\begin{tikzcd} P^{sat}\rar{\alpha_i} & (\mathbb C,\times)
\\ P\uar\ar{ur}{\alpha}\end{tikzcd}\]
The fibre product $(\mathrm X^{int})^{sat}$ consists of these $n$ points with the log structure $\alpha_i$. As $\mathbb C^*\e{[0,\infty)}$ is saturated, any homomorphism $P\longrightarrow \mathbb C^*\e{[0,\infty)}$ factors uniquely through $P^{sat}$. So, given any map $p^+\longrightarrow \rm X$, there exists a unique commutative diagram
\[\begin{tikzcd} \mathbb C^*\e{[0,\infty)}\ar{dr}{\totl{\cdot}}
\\ \uar P^{sat}\rar{h} & (\mathbb C,\times)
\\ \ar[bend left]{uu}P\uar\ar{ur}{\alpha}\end{tikzcd}\]
so, this homomorphism factors uniquely through $(\mathrm  X^{int})^{sat}$, in particular the component that is  the log point with log structure $\alpha_i=h$. We conclude that 
\[\expl X=\expl (X^{int})^{sat}\]
so
\[\expl \rm A\times_B C=\expl  A\times_{\expl B}\expl C\ .\]
\end{proof}

\begin{remark} When a fibre product of explodable log schemes is explodable, its explosion is the fibre product in the category of exploded manifolds. This is because, when this fibre product exists,  a map to $\rm \expl A \times_{\expl B}  \expl C$ is equivalent to a map to $\rm \expl A  \expl B$ whose image is within the set of points send to the same point in $\expl \rm C$; see \cite[Section 9]{iec}.
\end{remark}

\begin{cor}\label{inverse image}Given a map of fs log schemes $f:\rm X\longrightarrow Y$, the inverse image of a logarithmically proper subset of $\expl Y$ under $\expl f$ is a logarithmically proper subset of $\expl X$.
\end{cor}

\begin{proof} Any logarithmically proper subset $V\subset \expl \rm Y$ is the image the explosion of a proper map $h:\rm Z\longrightarrow Y$ from an explodable log scheme $\rm Z$. Taking fibre products in the category of fs log schemes, we get a proper map $h': \rm Z\times_Y X\longrightarrow Y$. Moreover Lemma \ref{fp comparison} implies that the image of $\expl h'$  is $\expl f^{-1}V$. We can compose $h'$ with the resolution of $\rm Z\times_Y X$ by an explodable log scheme, from Proposition \ref{explodable resolution}, obtaining that $\expl f^{-1}V$ is the image of the explosion of a proper map from an explodable log scheme.
\end{proof}
\

\begin{cor} The finite intersection or union of logarithmically proper subsets is logarithmically proper. 
\end{cor}

\begin{proof} The finite union of logarithmically proper subsets is logarithmically proper, because it is  still the image of a proper map from an explodable log scheme. If logarithmically proper subsets $V_i$ are the image of $\expl f_i$ for  proper maps $f_i:\mathrm X_i\longrightarrow \mathrm Y$ from explodable log schemes, Lemma \ref{fp comparison} implies that $V_1\cap V_2$ is the image $\expl \mathrm X_1\times_{\mathrm Y}\mathrm X_2$. Proposition \ref{explodable resolution} then implies that $V_1\cap V_2$ is logarithmically proper.  \end{proof} 

For the following corollary, we use the notation $\ex X$ for an exploded manifold $\expl \rm X$.
\begin{cor} If $\ex V \subset \ex X^-  \ex Y$  and $\ex W\subset \ex Y^-  \ex Z$ are logarithmically proper,  and $\ex V$ is proper over $\ex X$ then
\[\ex V\circ_{\ex Y}\ex W:=\pi(\ex V\times_\ex Y \ex W)\]
is a logarithmically proper subset of $\ex X  \ex Z$, where $\pi$ is the projection in the following diagram.
\[\begin{tikzcd} \ex X  \ex Y\dar & \lar \ex X  \ex Y  \ex Z\dar \rar{\pi} & \ex X  \ex Z
\\ \ex Y & \lar \ex Y  \ex Z\end{tikzcd}\]
Moreover if $\ex V$ and $\ex W$ are isotropic, then $\ex V\circ_{\ex Y}\ex W$ is also isotropic, and if $\ex W$ is proper over $\ex Y$, $\ex V\circ_{\ex Y}\ex W$ is also proper over $\ex X$.
 \end{cor}

\begin{proof}Proposition \ref{explodable resolution} together with Lemma \ref{fp comparison} give that $\ex V\times_{\ex Y}\ex W$ is the image of the explosion of a proper map 
\[\expl f:\expl (\rm Y)\longrightarrow \ex V\times_{\ex Y}\ex W\]  from an explodable log scheme $\rm Y$. As $\ex V$ is proper over $\ex X$, the map $\pi$ is proper restricted to $\ex V\times_{\ex Y}\ex W$, so $\ex V\circ_{\ex Y}\ex W$ is also the image of the explosion of a proper map from the explodable log scheme $\rm Y$.

Now, suppose that $\ex V$ and $\ex W$ are isotropic. This implies  $\expl f^*\omega_{\ex X}=\expl f^*\omega_{\ex Y}=\expl f^*\omega_{\ex Z}$, so $(\pi\circ \expl f)^*(\omega_Z-\omega_X)=0$, and $\ex V\circ_{\ex Y}\ex W$ is isotropic. 

\end{proof}

\subsection{Tropical perspective}

\

There is a functor assigning each exploded manifold $\ex X$ a tropical part $\totb {\ex X}$ related to the tropicalisation of a log scheme; see \cite[Section 4]{iec}. For example, the tropical part of $\expl (X,D)$ is the dual intersection complex, and the tropical part of $\mathrm p^k$ is $(0,\infty)^k$. There is a surjective map of sets $\ex X\longrightarrow \totb{\ex X}$, related to the tropical part homomorphism $\mathbb C\e{\mathbb R}\longrightarrow \e{\mathbb R}$ from the exploded semiring $\mathbb C\e{\mathbb R}$ to the tropical semiring $\e{\mathbb R}$ 
\[c\e a\mapsto \totb{c\e a}:= \e{a}\]
In particular, when $\rm X$ is an affine log scheme,  each point in $\expl \rm X$ determines a homomorphism $M_{\rm X}\longrightarrow ([0,\infty),+)$ by composing the homomorphism $M_{\rm X}\longrightarrow M_{\rm p^+}=\mathbb C^*\e{[0,\infty)}$ with the tropical part homomorphism $\mathbb C^*\e{[0,\infty)}\longrightarrow \e{[0,\infty)}=([0,\infty),+)$. In this case $\totb{\expl {\rm X}}$ is the set of such homomorphisms $M_{\rm X}\longrightarrow [0,\infty)$, and the map $\expl {\rm X}\longrightarrow \totb{\expl \rm X}$ sends the point $\rm p^+\longrightarrow \rm X$ to the corresponding homomorphism.

 Note the usual definition of the tropicalisation of an affine $\rm X$ is that it is the cone $\hom(M_{\rm X},[0,\infty))$, but in the case that $\rm X$ is not log smooth,  $\totb{\expl {\rm X}}$ is a subset of this tropicalisation  obtained by removing some faces. In particular, if $m\in M_{\rm X}$ is sent by the log structure map to $0$, we remove the face comprised of those homomorphisms sending $m$ to $0\in[0,\infty)$. 
 
 If $\rm U\subset X$ is an open subset, then $\totb{\expl \rm U}\subset \totb{\expl \rm X}$ is a face.    When $\rm X$ is not affine, $\expl X\longrightarrow \totb {\expl X}$ can be defined by taking a quotient by the smallest equivalence relation such that, for affine $\rm U\subset \rm X$,   points sent to the same point in $\totb{\expl \rm U}$ are equivalent. In nice cases, such as when $\mathrm X=(X,D)$ with $D$ a simple normal crossing divisor,  the corresponding maps $\totb {\expl \rm U}\longrightarrow \totb{\expl \rm X}$ are injective, and $\totb{\expl \rm X}$ is a union of cones glued along faces. In general however, the map $\totb {\expl \rm U}\longrightarrow \totb{\expl \rm X}$ may factor through a quotient of the cone $\totb {\expl \rm U}$ by a group of automorphisms.

\begin{example} \label{surjective point map} If  $\rm q$ is a fs log point, the tropical part $\totb{\expl \rm q}$ of $\expl \rm q$ is the open cone $Q$ comprised of sharp\footnote{A sharp homomorphism is one that does not send any non-invertible elements to invertible elements.} homomorphisms $M_{\rm q}\longrightarrow [0,\infty)$, whereas the usual definition of tropicalisation of $\rm q$ is the closure of this cone. Moreover,  $\expl \rm q$ is the set of sharp monoid homomorphisms $M_{\rm q}\longrightarrow \mathbb C^*\e{[0,\infty)}$ that are the identity on $\mathbb C^*\subset M_{\rm q}$. As $\rm q$ is fine and saturated, $\bar M_{\rm q}:=M_{\rm q}/\mathbb C^*$ is a toric monid, and $\bar M_q^{gp}$ is isomorphic to $\mathbb Z^n$. We have that  $Q:=\totb{\expl \rm q}$ is an open cone within $\hom(\bar M_{\rm},\mathbb R)=\mathbb R^n$, and there is a free and transitive action of $\hom(\bar M_q^{gp},\mathbb C^*)=(\mathbb C^*)^n$ on the fibres of the map $\expl \rm q\longrightarrow \totb{\expl \rm q}$. The exploded manifold $\expl \rm q$ is the standard exploded manifold $\et nQ$; see \cite[Example 3.8]{iec}.   A map of fs log points
\[f:\rm q\longrightarrow \rm r\]
is equivalent to a sharp monoid homomorphism $M_{\rm r}\longrightarrow M_{\rm q}$ that is the identity on $\mathbb C^*\subset M_{\rm r}$, and both 
\[\expl f: \expl \rm q\longrightarrow \expl \rm r\]
and 
\[\totb{\expl f}:\totb{\expl \rm q}\longrightarrow \totb{\expl \rm r}\]
are given by pullback via this homomorphism. 

A useful fact is that, for such maps of log points,  $\expl f$ is surjective if and only if $\totb{\expl f}$ is surjective. This follows from the observation that $\expl f$ is equivariant with respect to the action of $\hom(\bar M_{\rm q},\mathbb C^*)$ on $\expl q$ and the induced action on $\expl \rm r$ using the homomorphism $\hom(\bar M_{\rm q}^{gp},\mathbb C^*)\longrightarrow \hom(\bar M^{gp}_{\rm r},\mathbb C^*)$. In particular, for $\totb{\expl f}$ to be surjective, we need the map $\hom(\bar M_{\rm q}^{gp},\mathbb R)\longrightarrow \hom(\bar M^{gp}_{\rm r},\mathbb R)$ to be surjective, which is the case if and only if the homomorphism $\bar M_{\rm r}^{gp}\longrightarrow \bar M_{\rm q}^{gp}$ is injective, so the map $\hom(\bar M_{\rm q}^{gp},\mathbb C^*)\longrightarrow \hom(\bar M^{gp}_{\rm r},\mathbb C^*)$ is surjective. We conclude that when $\totb{\expl f}$ is surjective, there is at least one point in $\expl \rm q$ in each fibre of $\expl \rm r\longrightarrow\totb{ \expl \rm r}$, and  $\hom(\bar M_{\rm q}^{gp},\mathbb C^*)$ acts transitively on each fibre, and therefore $\expl f$ is also surjective.

\end{example}

\begin{remark}It is not true that an infinite intersection of logarithmically proper subsets is logarithmically proper. For example, the set of exploded points in $\expl \mathrm p^2$ is $(\mathbb C^*\e{(0,\infty)})^2$, and we have coordinates $z_1$ and $z_2$ taking values in  $\mathbb C^*\e{(0,\infty)}$. For any two nonzero integral vectors $(a,b)$ and $(c,d)$ in $\mathbb N^2$, there is a map $\mathrm p^2\longrightarrow \mathrm p^2$ whose explosion is given by $(z_1,z_2)\mapsto (z_1^a z_2^b,z_1^c z_2^d)$, and the image of this map consists of those points with tropical part in the interior of the cone spanned by $(a,b)$ and $(c,d)$.
Thus, given any open cone in $(0,\infty)^2=\totb{\expl \mathrm p^2}$ defined by rational inequalities,  there is a logarithmically proper subset of $\expl\mathrm p^2$ consisting of those points with tropical part in this cone. An infinite intersection of such cones can give a cone which is partially closed, or which is defined using an irrational inequality, and the set of exploded points over such a cone is not a logarithmically proper subset. 
\end{remark}

The explosion of a log modification $\mathrm X'\longrightarrow \mathrm X$ is a bijection on the set of exploded points, and an example of a refinement of exploded manifolds; see \cite[Section 10]{iec}.  A logarithmic modification of $\rm X$ is equivalent to a subdivision of $\totb{\expl X}$ into rational cones --- such a logarithmic modification is always log smooth, and is smooth if these cones are all isomorphic to $[0,\infty)^k$ using some invertible $\mathbb Z$--linear transformation.

 Corollary \ref{inverse image} implies  that the notion of being logarithmically proper is the same in both $\mathrm X$ and $\mathrm X'$. In particular, a subset $V\subset \mathrm X$ is logarithmically proper if and only if its inverse image in $\mathrm X'$ is logarithmically proper.

We will have need of logarithmically proper subsets of the exploded manifold $\ex T^n$. This exploded manifold has refinements isomorphic to the explosion of compact toric $n$-folds. As the notion of being logarithmically proper is invariant under refinements, we can make the following definition.

\begin{defn}[logarithmically proper subset of exploded manifold] \label{exploded logarithmically proper}Let $\ex X$ be a holomorphic exploded manifold whose tropical part is a cone over some chosen point.  A logarithmically proper subset of $\ex X$ is a subset which is the image of a logarithmically proper subset of a refinement of $\ex X$, where this refinement subdivides $\totb{\ex X}$ into cones over $0\in\totb{\ex X}$ and is isomorphic to the explosion of an explodable log scheme. 

 \end{defn}
 
 \begin{defn}[lagrangian subvariety of exploded manifold] \label{exploded isotropic def} A logarithmically   proper subset $\ex V\subset \ex X$ is isotropic if, for all maps of exploded manifolds $f:\ex Y\longrightarrow \ex V\subset \ex X$, $f^*\omega=0$.  A logarithmically proper subset $\ex V\subset \ex X$ is a lagrangian subvariety if it is isotropic, and there exists a connected log smooth $\rm Y$, with $\dim \mathrm Y=\frac 12 \dim \ex X$ and a surjective, generically injective map $\expl {\mathrm Y}\longrightarrow \ex V\subset \ex X$.
 
 \end{defn}

Remark \ref{complete connection}  implies that, when $\ex X=\expl \rm X$, lagrangian subvarieties of $\ex X$ correspond to lagrangian subvarieties of $\rm X$.

%

\begin{remark}From the perspective of exploded manifolds, it is unnatural to specify a cone structure on $\totb{\ex X}$ and unnatural to restrict to refinements that preserve this cone structure in Definition \ref{exploded logarithmically proper}. This unnatural restriction ensures that, given two cone-structure-preserving refinements of $\ex X$ to $\expl\rm Y_1$ and $\expl\rm Y_2$, there is a common refinement that is the explosion of a logarithmic modification of $\rm Y_1$ and $\rm Y_2$.

 The explosion of a map between log schemes $f: \mathrm A\longrightarrow\mathrm B$ always has tropical part $\totb{f}:\totb{\expl \rm A}\longrightarrow \totb{\expl \rm B}$ that is a map of cones. There are more maps from $\expl \rm A$ to  $\expl \rm B$ as  exploded manifolds: such a map is the explosion of a map of log schemes if and only it is holomorphic and its tropical part is a map of cones. \end{remark}

\begin{defn}[combinatorially flat] A map $f:\rm X\longrightarrow \rm Y$ of explodable log schemes is combinatorially flat if, the interior of each cone in $\totb{\expl \rm X}$ is sent surjectively onto the interior of a cone in $\totb{\expl \rm  Y}$.
\end{defn}

In light of Example \ref{surjective point map}, a combinatorially flat map has the useful property that the image of $\expl f$ consists of those points in $\expl \rm Y$ that project to the image of $f$ within $Y$.

The following is a simplified version of \cite[Proposition 1.7.2]{logGWDT}

\begin{lemma}\label{combinatorially flat} If $f:\rm X\longrightarrow \rm Y$ is a proper map of explodable log schemes with $\rm Y$ log smooth, there exists a modification $\rm Y'\longrightarrow \rm Y$ with $Y$ smooth such that in the induced fibre product diagram
\[\begin{tikzcd} \rm X'\rar{f'}\dar  & \rm Y'\dar
\\ \rm X \rar{f} & Y \end{tikzcd}\]
the map $f'$ is combinatorially flat. 
\end{lemma}

\begin{proof} First choose a subdivision of $\totb{\expl \rm Y}$ into cones such that the image of each cone in $\totb{\expl \rm X}$ is some union of cones in this subdivison. By toric resolution of singularities, we can take a further subdivision so that all these cones are standard, so the corresponding log modification $\rm Y'\longrightarrow \rm Y$ has $\rm Y'$ smooth. This logarithmic modification has the required properties. 
\end{proof}

\begin{lemma}\label{Dense subset} If a logarithmically proper subset $V\subset \expl( X,D)$ has dense image $\totl{ V}\subset X=\totl{\expl (X,D)}$, then $V=\expl(X,D)$.

\end{lemma}
\begin{proof} Without losing generality, by using Lemma \ref{combinatorially flat}, we can assume that $V$ is the image of the explosion of a proper combinatorially flat map $f:\rm Y\longrightarrow X$. As this map is proper, and has dense image in $X$, it surjects onto $X$. As $f$ is combinatorially flat, restricting $\expl f$ to the explosion of each log point  gives a surjective map, as explained in \ref{surjective point map}. So $\expl f$ is surjective. 

\end{proof}
\begin{remark}\label{complete connection} Resolution of singularities implies that each complete subvariety $V\subset  (X,D)=\rm X$  is the image of a proper logarithmic map $f:\rm Y\longrightarrow \rm X$ with $\rm Y$ log smooth. The image $\ex V\subset \expl \rm X$ of $\expl f$ is a logarithmically proper subset whose  image $\totl{\ex V}\subset X$ is $V$. Moreover, if $\ex W\subset \expl \rm X$ is any logarithmically proper subset with $\totl {\ex W}=V$, we have that $\expl f^{-1}(\ex W)$ is  logarithmically proper so Lemma \ref{Dense subset} implies that this is all of $\expl Y$. We conclude that $\ex V$ is the unique minimal logarithmically proper subset with $
\totl{\ex V}=V$.

Given any logarithmic modification $\pi :\rm X'\longrightarrow \rm X$, the pullback $\pi^*V$ of $V$ may be smaller than $\pi^{-1}V$, but using exploded manifolds, the pullback coincides with the inverse image. In particular, we have that
\[\pi^*V=\totl{\expl \pi^{-1}\ex V}\ .\]
Lemma \ref{lmodification} can now be proved using Lemma \ref{combinatorially flat}. Our goal is to find a logarithmic modification in which $\pi^*V$ intersects the divisor nicely. Choose compatible logarithmic modifications 
\[\begin{tikzcd}\rm Y'\dar\rar{f'}&\rm X'\dar{\pi}
\\ \rm Y\rar{f}&\rm X\end{tikzcd}\]
so that $\expl f'$ is combinatorially flat. As $\rm Y'$ is log smooth, every cone in $\totb{\expl \rm Y'}$ is closed, so combinatorial flatness implies that  image of $f'$ intersects the divisor in $\rm X'$ nicely. It also implies that the image of $\expl f'$ is the set of points with smooth part $f'(Y')=\pi^*V$, so pulling back via further logarithmic modifications will coincide with taking inverse images. Moreover, taking the fibre product of $f'$ with any further logarithmic modification of $\rm X'$ gives a map $f''$ satisfying the same conditions as $f'$, so the pullback of $\pi^*V'$ under further logarithmic modifications will also intersect the divisor nicely. 
\end{remark}

\subsection{Star product using refined cohomology} \label{refined cohomology section}

\

We can construct an appropriate homology theory in the category of exploded manifolds using refined differential forms. 
Refined differential forms on exploded manifolds are defined in \cite[Section 9]{dre}. All the usual operations on differential forms work as usual for refined differential forms: pullbacks, wedge products, exterior derivatives, integration, and pushforwards are defined, and satisfy the usual properties.

Of particular interest to us is the following fact: if  $f: \rm A\longrightarrow B$ is a surjective, generically injective map of smooth log schemes, then it is also true that $\expl f$ is surjective and generically injective, and $\int_{\expl \rm A}\expl f^*\theta=\int_{\expl \rm B} \theta$. In particular, if $\ex V\subset \ex X$ is is the image of $\expl f$ with $f$ generically injective, we can define
\[\int_{\bf V}\theta:=\int_{\expl \rm Y}\expl f^*\theta\]
and this definition does not depend on the choice of $f$. So, given a lagrangian correspondence $\ell$ in the form $\ell=\sum_i n_i \ex V_i$, we can define
\[\int_{\ell} \theta:= \sum_i n_i \int_{\ex V_i}\theta\]

\begin{defn}[refined neighbourhood] A refined neighbourhood $U$ of $\ex V\subset \ex X$ is a subset  $U\subset \ex X$ containing $\ex V$, such that $U$ is open in some refinement of $\ex X$.
\end{defn}

In the special case that $\ex V$ is the image of $\expl f$ for $f$ a combinatorially flat map, any refined neighbourhood $U$ of $\ex V$  contains  an open neighborhood of $\ex V$.

When $\ex V\subset \ex X$ is a lagrangian subvariety,  we can  construct the Poincare dual $PD(\ex V)$ of $\ex V$ as a closed refined differential form on $\ex X$ supported within an arbitrarily small refined neigbourhood  $U$ of $\ex V\subset \ex X$; see \cite[Lemma 9.5]{dre}.  Then given  any compactly supported refined differential form $\theta$ on $\ex X$ which is closed in $U$, we have
\[\int_{\ex V}\theta=\int_{\ex X}\theta\wedge PD(\ex V)\]

We can consider $PD(\ex V)$ as an element in  $\rh^{2(\dim X-\dim V)}(\ex X,\ex X\setminus U)$, where $\rh^*(\ex X,\ex X\setminus U)$ indicates the de Rham cohomology defined using refined differential forms on $\ex X$ supported within $U$. We can similarly define the Poincare duals of lagrangian correspondences.

\begin{lemma}\label{free generation} Suppose that $\ex V\subset \expl \rm X$ is a logarithmically proper isotropic subset. Then every refined neighborhood $U$ of $\ex V\subset \expl \rm X$ has a refined subneighbourhood $U'$ such that $\rh^{\dim X}(\expl\mathrm  X,\expl\mathrm X\setminus U')$ is freely generated by the Poincare duals of the lagrangian subvarieties in $\ex V$.
\end{lemma}
\begin{proof}
\

The proof proceeds by removing any problematic locus in $\ex V$. Let the complex dimension of $X$ be $2k$. Using Lemma \ref{combinatorially flat}, and a logarithmic modification of $\rm X$, we can assume that $\ex V$ is the image of the explosion of a combinatorially flat map. This ensures that $\ex V$ is fully recorded in $\totl{\ex V}\subset X$.  Moreover, it also ensures that there exists some neighbourhood of $\totl{\ex V}\subset X$, whose explosion is contained in the refined neighborhood $U$ of $\ex V$.

Let $V_0\subset \totl{\ex V}\subset X$ be the intersection of $\totl {\ex V}$ with the divisor in $\rm X$, together with every irreducible component of dimension less than $k$, and the singular locus of $\totl {\ex V}$.  We have that $V_0$ is a closed algebraic subset with dimension bounded by $(k-1)$. There therefore exists an open neighbourhood $N$ of $V_0$ that retracts onto $V_0$. There is no loss of generality in assuming that the explosion of $N$ is contained in $U$.  Moreover, we can choose $N$ so that its closure $\bar N$ is a manifold with boundary and also retracts onto $V_0$, we can also ensure that  its boundary $\bar N\setminus N$ is a smooth codimension 1 submanifold whose explosion is a smooth codimension $1$ exploded submanifold of $\expl \rm X$. By choosing $N$ and the retraction starting with the lowest dimensional strata, we can also ensure that this retraction is compatible with the stratification of $\rm X$, in the sense that it restricts a retraction of $N$ and $\bar N$ intersected with the closure of each stratum in $\rm X$. We shall see that removing $\bar N$ from $\rm X$ will not affect our cohomology in degree $2k$.

Choose  an open neighbourhood $O$ of $\totl{\ex V}$ such that
\begin{itemize}\item $\expl O\subset U$, 
 \item  $O\supset N$,
 \item  $O\setminus \bar N$ retracts onto $\totl {\ex V}\setminus \bar N$,
 \item and such that, restricted to lower dimensional strata of $X$,  $O$ coincides with $ N$. 
 \end{itemize}

\begin{claim} \label{set removal} 
\[\rh^{2k}(\expl \mathrm X, \expl \mathrm X\setminus \expl O)=\rh^{2k}(\expl (\mathrm X\setminus \bar N),\expl (\mathrm X\setminus O))\ .\]

\end{claim}

Once this claim is proved, our lemma follows easily, because the locus in $\expl (\mathrm X\setminus \bar N)$ where our refined forms are supported is  $\expl (O\setminus \bar  N)=O\setminus \bar N$, which is  a smooth complex manifold, retracting on to the proper $k$-dimensional complex submanifold  $\totl{\ex V}\setminus \bar N$, and this complex submanifold has one component for each lagrangian subvariety of $\ex V$. On this locus, refined differential forms are just ordinary smooth differential forms, so the degree $2k$ cohomology is generated freely generated by the Poincare duals of these components.

It remains to prove Claim \ref{set removal}. Consider the following short exact sequence of chain complexes 
\[\begin{tikzcd} \ro^*(\expl \mathrm X, \expl X\setminus \expl N) \rar & \ro^*(\expl \mathrm X, \expl (\mathrm X\setminus O))\dar
\\ & \ro^*(\expl (\mathrm X\setminus N),(\expl (\mathrm X\setminus  O)  )\end{tikzcd}\]
where $\ro^*(\expl (\mathrm X\setminus N),(\expl (\mathrm X\setminus  O)) $ indicates the chain complex of refined differential forms on $\expl X$, supported within $\expl O$, up to the equivalence relation that two of these are equivalent if they agree on the closed subset $ \expl X\setminus \expl N$. Because $X\setminus N$ a manifold with boundary and $\expl (\bar N\setminus N)$ is a codimension 1 exploded submanifold, there is a smooth vector field on $\expl X$ whose flow  sends  $\expl (\mathrm X\setminus N)$ inside $\expl (\mathrm X\setminus \bar N)$, so the chain map 
\[\ro^*(\expl (\mathrm X\setminus N),(\expl (\mathrm X\setminus  O)  )\longrightarrow \ro^*(\expl (\mathrm X\setminus\bar  N),\expl (\mathrm X\setminus (\bar N\cup O)  )\]
induces an isomorphism on homology,  and  therefore the homology of the chain complex $\ro^*(\expl (\mathrm X\setminus N),(\expl (\mathrm X\setminus  O)  )$ is $\rh^{2k}(\expl (\mathrm X\setminus \bar N),\expl (\mathrm X\setminus O))$.

Therefore, the short exact sequence above implies that, to check Claim \ref{set removal},  it suffices to check that $\rh^{d}(\expl \mathrm X, \expl \mathrm X\setminus \expl N)=0$ for $d=2k$ and $2k-1$. The de Rham cohomology of $\expl \rm X$ coincides with the ordinary cohomology of $ X$ with real coefficients; see \cite[Corollary 4.2]{dre}, and similarly, $H^*(\expl (\mathrm X\setminus \bar N))=H^*(X\setminus \bar N,\mathbb R)$. As argued above, $H^*(\expl (\mathrm X\setminus \bar N) )$ can also be computed as the homology of the chain complex $\Omega^*(\expl (\mathrm X\setminus N))$ comprised of differential forms on $\mathrm X$ up to the equivalance relation of agreeing on the closed subset $\expl (\mathrm X\setminus N)$. We then have the following short exact sequence for relative de Rham cohomology
\[ \Omega^*(\expl \mathrm X, \expl (\mathrm X\setminus N))\longrightarrow \Omega^*(\expl X)\longrightarrow\Omega^*(\expl(\mathrm X\setminus N) )\]
and  \cite[Corollary 4.2]{dre} then also implies the equivalence of relative de Rham cohomology with ordinary relative cohomology.

\[H^*(\expl \mathrm X, \expl (\mathrm X\setminus N))=H^*( X,  X\setminus N ;\mathbb R)\]
with the lefthand side being de Rham cohomology, and the righthand side being usual cohomology with real coefficients. As $N$ retracts onto $V_0$, we have $H^d(X,X\setminus N)=H^d(X,X\setminus V_0)=H_{4k-d}^{BM}(V_0)$, which vanishes for $d=2k$ and $2k-1$, because the real dimension of $V_0$ is $2k$. So, we conclude that $H^d(\expl \mathrm X, \expl O)=0$ for $d=2k$ and $(2k+1)$.

We can reach the same conclusion for any smooth logarithmic modification $\rm X'\longrightarrow X$. Because $V_0$ is the image of a combinatorially flat map, its inverse image $V'$ in $\rm X'$ still has dimension bounded by $(2k-1)$, and the stratum-preserving  retraction of $N$ and $\bar N$ onto $V_0$ lifts to a stratum preserving retraction of the inverse images,  $N'$ and $\bar N'$ onto $V_0'$. So, as above, we can conclude that $H^d(\expl \mathrm X', \expl (\mathrm X'\setminus N'))=0$ for $d=2k$ and $(2k+1)$. 

The above vanishing implies the analogous vanishing of $\rh^d(\expl \mathrm X, \expl (\mathrm X\setminus N))$, because $\rh^d(\expl \mathrm X, \expl (\mathrm X\setminus N))$ coincides with limit of $H^d(\expl \mathrm X', \expl (\mathrm X'\setminus N'))$ over such smooth refinements. To see this, note that refined cohomology of $\expl \rm X$ coincides with taking the limit of the de Rham cohomology of $\expl \rm X$ over all refinements of $\rm  X$. Not every refinement corresponds to a smooth logarithmic modification of $\rm X$, but each refinement has a further refinement, that is connected by a family of refinements to $\expl \rm X'$ for $\rm X'$ a smooth logarithmic modification of $\rm X$. Invariance in families; \cite[Proposition 11.4]{dre}, then implies the required vanishing result for an arbitrary refinement. We can then conclude Claim \ref{set removal}, completing the proof of Lemma \ref{free generation}.

\end{proof}

\begin{prop}[star product using refined cohomology]\label{exploded star} Suppose that  $\ex V\in \Lag{\ex A,\ex X}$ and $\ex W\in \Lag{\ex X  \ex B ,\ex C}$. Then,  there exists a unique lagrangian correspondence 
\[\ex V\star_{\ex X}\ex W\in \Lag{\ex A  \ex B,\ex C}\]
satisfying the conditions below.  

Consider the following maps.
\[\begin{tikzcd} (\ex A^-  \ex X)  (\ex X^-  \ex B^-  \ex C)&\lar{\iota} \ex A^-  \ex X  \ex B^-  \ex C\rar{\pi} & \ex A^-  \ex B^-  \ex C\end{tikzcd}\]
For each choice of Poincare dual $PD(\ex V\times \ex W)$, supported in a sufficiently small refined neighbourhood $U$ of $\ex V\ex W$, 
\[\int_{\ex V\star_{\ex X}\ex W} \theta=\int_{\ex A^-\ex B^-\ex C}\theta \wedge   \pi_!\iota^*PD(\ex V\times \ex W)\]
for each compactly supported refined form $\theta$ on $\ex A^-\ex B^-\ex C$ that is closed when restricted to $\pi \iota^{-1} U$. 

In particular, 

\[ \pi_! \iota^*PD(\ex V\times \ex W)=PD(\ex V\star_{\ex X}\ex W) \in \rh^*(\ex A^-\ex B^-\ex C, \ex A^-\ex B^-\ex C\setminus \pi \iota^{-1} U) \]


\end{prop}

\begin{proof}
Without losing generality, we can assume that $\ex V$ and $\ex W$ are lagrangian subvarieties, with the general case following by linearity. 

For $\pi!\iota^*PD(\ex V\times \ex W)$ to be defined using \cite[Theorem 9.2]{dre}, we require that the support of $\iota^*PD(\ex V\times \ex W)$ is proper over $\ex A^-\ex B^-\ex C$. Choose complete metrics on $\ex A$,  $\ex X$, $\ex B$ and $\ex C$. Use the notation $\ex V_r$ and $\ex W_r$  for the set of points of distance at most $r$ from $\ex V$ or $\ex W$ respectively. As $\ex V$ is proper over $\ex A$, we have that $\ex V_r$  is also proper over $\ex A$. We can choose our refined neighborhood $U$ of $\ex V\ex W$ to be contained in $\ex V_r \ex W_r$, and then $\iota^{-1}(U)$ is contained in $\ex V_r \ex B^-\ex C$, which is proper over $\ex A^-\ex B^-\ex C^-$. So, $\pi!\iota^*PD(\ex V\times \ex W)$ is well defined. 

We claim that,  for any refined neighborhood $O$ of $\pi\iota^{-1}\ex V\ex W:=\ex V\circ_{\ex X}\ex W\subset \ex A^-\ex B^-\ex C$, and compact subset $K\subset \ex A^-\ex B^-\ex C$, there exists a refined neigbourhood $U$ of $\ex V\ex W$ such that $\pi^{-1}O\supset \pi^{-1}K\cap \iota^{-1}(U)$.  To prove this claim,  note that $(\iota^{-1}\ex V_r\ex W_r)^c$ for $r>0$ forms an open cover of the complement of $\iota^{-1}\ex V\ex W$, and therefore an open cover of the compact set  $\iota^{-1}(\ex V_{1}\ex W_{1})\cap  \pi^{-1}(K\setminus O)$. It follows that $(\iota^{-1}\ex V_r\ex W_r)^c$ contains  $\pi^{-1}(\ex V_{1}\ex W_{1})\cap  \pi^{-1}(K\setminus O)$ for some $0<r<1$, and therefore $\iota^{-1}\ex V_r\ex W_r\cap \pi^{-1}K$ is contained in $O$. 
 
 As $\ex A^-\ex B^-\ex C$ has an exhaustion by compact subsets, it follows that for any refined open neighborhood  $O$ of $\ex V\circ_{\ex X}\ex W$, there exists an open neighborhood $U$ of of $\ex V\times \ex W$ such that $\iota^{-1}U\subset \pi^{-1}O$, and in particular, $\pi\iota^{-1}U\subset O$. So, by constructing $PD(\ex V\times \ex W)$ supported in $U$, we have that $\pi_!\iota^*PD(\ex V\times \ex W)$ is supported in $O$.

Lemma \ref{free generation} implies that, for $O$ small enough, $\rh^*(\ex A^-\ex B^-\ex C, \ex A^-\ex B^-\ex C\setminus O)$ is freely generated by the Poincare duals of the lagrangian subvarieties of  $\ex V\circ_{\ex X}\ex W$. So, in particular, there exists a unique $\ex V\star_{\ex X}\ex W\in \Lag{\ex A  \ex B,\ex C}\otimes\mathbb R$ such that 
\[PD(\ex V\star_{\ex X}\ex W)=\pi_!\iota^*PD(\ex V\times \ex W) \text{ in }\rh^*(\ex A^-\ex B^-\ex C, \ex A^-\ex B^-\ex C\setminus O)\ .\]
This implies that, \[\int_{\ex V\star_{\ex X}\ex W} \theta=\int_{\ex A^-\ex B^-\ex C}\theta \wedge   \pi_!\iota^*PD(\ex V\times \ex W)\]
for each compactly supported refined form $\theta$ on $\ex A^-\ex B^-\ex C$ that is closed when restricted to $\pi \iota^{-1} U$. Moreover, choosing $\theta$ to locally agree with the Poincare dual of a half-dimensional smooth surface $S$ transversely intersecting $\ex V\circ_{\ex X}\ex W$ exactly once in chosen Lagrangian subvariety, we obtain the weight associated to that subvariety, and therefore this condition uniquely specifies $\ex V\star_{\ex X}\ex W$. Moreover, this weight is integral because it coincides with the virtual intersection of $\pi^{-1}S$ with $\iota^{-1}(\ex V\times \ex W)$, so $\ex V\star_{\ex X}\ex W\in \Lag{\ex A\ex B,\ex C}$.

\end{proof}

\begin{prop}The star product is associative. In particular, 
\[(\ex U\star_{\ex X} \ex V)\star_{\ex Y}\ex W=\ex U\star_{\ex X}(\ex V\star_{\ex Y}\ex W)\ .\]
\end{prop}

\

\begin{proof}

The proof follows the proof of Lemma \ref{asp}, except now using properties of refined cohomology instead of local homology classes. 

Suppose that $\ex U\in \Lag{\ex A,\ex X}$, $\ex V\in\Lag{\ex X\ex B,\ex Y}$ and $\ex W\in\Lag{\ex Y\ex C,\ex D}$. Let $\ex M=\ex A^-\ex B^-\ex C^-\ex D$,  and consider the following commutative diagram, containing two fibre product diagrams.
\[\begin{tikzcd} \ex M\ex Y^2 & \ex M\ex Y\lar{\iota'_{\ex Y}}\rar{\pi'_{\ex Y}} & \ex M
\\ \ex  M\ex X\ex Y^2\uar{\pi_X} \dar{\iota_{\ex X}} & \ex M\ex X\ex Y\uar\lar\rar\dar & \ex M\ex X\uar{\pi'_{\ex X}} \dar{\iota'_{\ex X}}
\\ \ex M\ex X^2 \ex Y^2 &\ex  M\ex X^2 \ex Y\lar{\iota_{\ex Y}}\rar{\pi_{\ex Y}} & \ex M\ex X^2 \end{tikzcd}\]
As pushforward and pullbacks of refined differential forms are compatible with fibre products, \cite[Lemma 9.3]{dre}, the above diagram implies that 
\begin{equation}\label{ncfp} (\pi_{\ex X}')_!(\iota_{\ex X}')^*(\pi_{\ex Y})_!\iota_{\ex Y}^*PD(\ex U\ex V\ex W)=(\pi_{\ex Y}')_!(\iota_{\ex Y}')^*(\pi_{\ex X})_!\iota_{\ex X}^*PD(\ex U\ex V\ex W)\ .\end{equation}
Similarly, the following  cartesian diagram,  
\[\begin{tikzcd}\ex M\ex X^2 \ex Y^2\dar & \ex M\ex X^2 \ex Y\lar{\iota_{\ex Y}}\rar{\pi_{\ex Y}}\dar & \ex M\ex X^2\dar 
\\ \ex B \ex Y^2 \ex C & \ex B\ex Y\ex C\lar{\iota}\rar{\pi} & \ex B\ex C\end{tikzcd}\]
and the compatibilty of Poincare duals with fibre products, \cite[Lemma 9.5]{dre}   implies that
\[\iota_{\ex Y}^*PD(\ex U\ex V\ex W)=PD(\ex U)\wedge \iota^*PD(\ex V\ex W))\]
and then, using Proposition \ref{exploded star} we get that
\[(\pi_{\ex Y})_!\iota_{\ex Y}^*PD(\ex U\ex V\ex W)=PD(\ex U)\wedge (\pi_!\iota^*PD(\ex V\ex W))=PD(\ex U)\wedge PD(\ex V\star_{\ex Y}\ex W )\]
so, 
\[PD(\ex U\star_{\ex X}(\ex V\star\ex W))=(\pi_{\ex X}')_!(\iota_{\ex X}')^*(\pi_{\ex Y})_!\iota_{\ex Y}^*PD(\ex U\ex V\ex W)\ .\]
Similarly, we get that 
\[PD((\ex U\star_{\ex X}\ex V)\star\ex W)=(\pi_{\ex Y}')_!(\iota_{\ex Y}')^*(\pi_{\ex X})_!\iota_{\ex X}^*PD(\ex U\ex V\ex W)\]
The above equalities hold in refined cohomology supported on an arbitrarily small neighborhood of $\ex U\circ_{\ex X}\ex V\circ_{\ex Y}\ex W$, so Proposition \ref{exploded star} and  Equation (\ref{ncfp}) imply our required equality. 
\[(\ex U\star_{\ex X} \ex V)\star_{\ex Y}\ex W=\ex U\star_{\ex X}(\ex V\star_{\ex Y}\ex W)\]

\end{proof}

\section{Gromov--Witten invariants of log Calabi--Yau 3-folds} \label{GW section}

Let $(\rm X,\Omega)$ be a compact log smooth  Calabi--Yau 3-fold. In other words, $\rm X$ is compact, log smooth, and 3 dimensional, and  $\Omega$ is a non-vanishing holomorphic section of $\bigwedge^{(3,0)}T^*\rm X$, where $T^*\rm X$ indicates the logarithmic cotangent bundle. A simple example of a log Calabi--Yau 3-fold is given by a toric 3-fold with the log structure determined by its toric boundary divisor, and $\Omega=\frac{dz_1}{z_1}\wedge \frac{dz_2}{z_2}\wedge \frac{dz_3}{z_3}$. A more interesting example is provided by a non-toric blowup of this, analogous to Example \ref{ntbc}.   

Gromov--Witten invariants of $\rm X=(X,D_{\rm X})$ are defined using the moduli stack of holomorphic curves in the logarithmic or exploded categories, however we can gain some understanding by considering the moduli stack of smooth holomorphic curves in $X$ whose intersection with $D_X$ consists of some finite collection of marked points. Consider a holomorphic family of such curves parametrised by $F$.
\begin{equation}\label{smooth family}\begin{tikzcd}C\dar{\pi}\rar{f}& X
\\ \uar[bend left]{s_k} F\end{tikzcd}\end{equation}
We have that $f^*\Omega$ is a closed meromorphic $3$--form on the domain $C$,  with poles at the image of the marked point sections $s_k(F)$. By taking the residue of this meromorphic $3$--form at a marked point section, we can define a holomorphic 2-form $Res_{s_k}\Omega$ on $F$ such that
\[Res_{s_k}\Omega=\pi_! d\rho_k\wedge \Omega \]
where $\rho_k:C\longrightarrow \mathbb C$ is any smooth function that is $\frac i{2\pi}$ on $s_k(F)$ and $0$ on all other marked point sections. This 2-form does not depend on the choice of $\rho_k$ because $\Omega$ and the projection $\pi$ are both holomorphic. In particular, given any other choice $\rho'_k$ of $\rho_k$, we have that $\pi_!(\rho_k-\rho_k')\Omega=0$, because we are integrating a $(3,0)$--form along holomorphic fibres. Then, as exterior derivatives commute with integration along the fibre, we get that $\pi_!d\rho_k\wedge\Omega=\pi_!d\rho_k'\wedge \Omega$.   Similarly we have that 
\[\sum_k Res_{s_k}\Omega=0\ .\]

In this section, we will generalise the above argument and interpret $\sum_kRes_{s_k}\Omega$ as the pullback of a natural holomorphic symplectic form from a natural evaluation stack. So, the image of the moduli stack of holomorphic curves will be isotropic. In this setting, the virtual dimension of the moduli space is the number of marked points, and half the dimension of the natural evaluation stack. So, the pushforward of the virtual fundamental class will be a lagrangian cycle.

\begin{remark}
Something can also be said in the case of holomorphic curves with boundary. In this case we require $\rho_k$ to vanish on the boundary $\partial C$ of our holomorphic curves, and Stokes's theorem gives that $\sum_k Res_{s_k}\Omega=-\pi_! \Omega\rvert_{\partial C}$. If, for example, the real part of $\Omega$ vanishes on $\partial C$ because these curves have boundary on a special lagrangian within $\rm X$, this implies that the real part of $\sum_k Res_{s_k}\Omega$ vanishes. As the real part of a holomorphic symplectic form is a real symplectic form, we still have the interpretation that the image of the moduli stack is real isotropic. 
\end{remark}

\subsection{Holomorphic symplectic evaluation space}

\

Suppose that $(X,D,\Omega)$ is a log Calabi--Yau 3-fold. In this section, we construct a holomorphic symplectic evaluation space for the moduli stack of holomorphic curves in $(X, D)$ with specified contact with the divisor $D$. 

The easiest case to describe is the moduli space of curves with labeled marked points, each of which has simple contact with a component $D_\nu$ of the divisor. In this case, there is an evaluation map at such a marked point to the log scheme $\mathrm X_\nu:=(D_\nu,D\cap D_\nu)$. Recall, from Remark \ref{stratumlogstructure}, that we don't have an inclusion of $\rm X_\nu$ into $\rm X$, and instead the log structure of $\rm X$ restricted to $D_\nu$ is actually a $\rm p$--bundle over $X_\nu$
\[\begin{tikzcd} \mathrm S_\nu\rar{\iota}\dar{\pi} & \mathrm X
\\ \mathrm X_\nu \end{tikzcd}\]
The residue of $\Omega$ at $\rm X_\nu$ is a holomorphic symplectic form $\Omega_\nu$ on $\rm X_\nu$ satisfying the following condition. Locally, where $D_\nu$ is the vanishing of a coordinate $z$, 
\[\frac{dz}{z}\wedge\pi^*\Omega_\nu=\iota^*\Omega\ .\]
Given a smooth family of curves in the form (\ref{smooth family}), along with a specified marked point section $s_k$ with simple contact with $D_{\nu(k)}$, we have a natural evaluation map $ev_k:F\longrightarrow \mathrm X_{\nu(k)}$, and we have that
\[ev_k^*\Omega_{\nu(k)}=Res_{s_k}\Omega\ .\]
Similarly, if we have contact of order $d_k$ with $X_{\nu(k)}$, we have
\[ ev_k^*d_k\Omega_{\nu(k)}=Res_{s_k}\Omega\ .\]
Let $\omega$ be the holomorphic symplectic form on the product $\prod_k(X_{\nu(k)},d_k\Omega_{\nu(k)})$, and let 
\[ev:=\prod_k ev_k: F\longrightarrow \prod_k \mathrm X_{\nu(k)}\]
be the product evaluation map. We have
\[ev^*\omega = \sum_k Res_{s_k}\Omega=0 \ .\]
The above equation still holds for log smooth families of logarithmic holomorphic curves with analogous contact data, once we have a suitable generalisation of $Res_{s_k}\Omega$.

 If we stay within the category of log schemes, it becomes harder to describe the evaluation space when curves have points contacting multiple divisors. One fix is to choose a logarithmic modification $\rm X'$ of $\rm X$ so that the moduli space of curves under study have each marked point only contacting a single divisor. The moduli stack of curves in $\rm X'$ can be regarded as a logarithmic modification of the moduli stack of curves in $\rm X$, so we then gain information about the original moduli stack of curves in $\rm X$ by studying the curves in $\rm X'$. This approach provides appropriate holomorphic symplectic evaluation spaces that are defined only up to logarithmic modification; they also suffer the defect that evaluation maps to these evaluation spaces are defined on the moduli stack of curves in a suitable logarithmic modification $\rm X'$ of $\rm X$. With that said, there is an appropriate evaluation stack playing the role of $\prod_k (\mathrm X_{\nu(k)},d_k\Omega_{\nu(k)})$.  Instead of sticking with log schemes, we will work with exploded manifolds.

Let $\ex X$ be an exploded Calabi--Yau 3-fold (such as $\expl \mathrm X$) such that $\totb{\ex X}$ is a cone around some point $0\in\totb{\ex X}$.\footnote{We make the assumption that $\totb{\ex X}$ is a cone around some point $0$ for convenience. All arguments still work for the general case that $\ex X$ is an Calabi--Yau 3-fold, but the construction of the evaluation space is less convenient to describe.}

Consider a family of $\C\infty 1$ exploded curves\footnote{$\C\infty 1$ indicates a level of regularity, which plays the role of `smooth' for exploded manifolds. In particular, we are allowing the possibility that these curves are not holomorphic. See \cite[Section 7 and Definition 11.1]{iec} } in $\ex X$.
\[\begin{tikzcd}\ex C\rar{\hat f}\dar {\pi} & \ex X
\\ \ex F\end{tikzcd}\]

If this family of curves has marked ends, in place of a section $s_k$, we have a subset $\ex C_k\subset \ex C$ with a canonical free $\mathbb C^*\e{[0,\infty)}$--action such that $\pi:\ex C_k\longrightarrow \ex F$ is invariant.  This $\ex C_k$ is a union of strata of of $\ex C$, and the $\ex C^*\e{[0,\infty)}$ action restricts to each stratum.   In the case that this family of curves is the explosion of some family of log curves, $\ex C_k$ is the explosion of the image of a marked point section\footnote{It is also inaccurate to talk about a marked point section in the case of a family of logarithmic curves. There is a section on the level of underlying schemes, but the image of this section has the log structure of a $\mathrm p$--bundle over the base, and there is accordingly no `marked point section' as a logarithimic map. }. We will refer to $\ex C_k\longrightarrow \ex F$ as a $\mathbb C^*\e{[0,\infty)}$--bundle. In the case that $\ex X=\expl \rm X$, fibres of this bundle are always isomorphic to  $\et 1{(0,\infty)}$, but more generally, it is possible that fibres are isomorphic to  $\ex T=\et 1{\mathbb R}$. The map $\hat f$ restricted to $\ex C_k$ lands in some stratum  $\ex  S_k\subset \ex X$.
\[\begin{tikzcd}\ex C\supset\ex C_k\rar{\hat f\rvert_{\ex C_k}}\dar {\pi} & \ex S_k\subset \ex X
\\ \ex F\end{tikzcd}\]
Restricted to each fibre $\totb f:\totb{\ex C_k}\longrightarrow \totb{\ex X}$ gives a map of $(0,\infty)$ or $\mathbb R$ into $\totb{\ex X}$. Using that $\totb{\ex X}$ is a cone over $0\in\totb{\ex X}$, we can identify the derivative of this map with an integral vector $\nu(k)$ in $\totb{\ex X}$. When $\ex X=\expl(X,D)$, $\nu(k)$ encodes the contact order of this family of curves with the divisor $D\subset X$. As this stratum  $\ex S_k\subset \ex X$ depends on $\nu(k)$, we will use the notation  $\ex S_k=\ex S_{\nu(k)}$.  There is then a canonical exploded manifold $\ex X_{\nu(k)}$, constructed in \cite[Section 3]{gfgw} and a canonical projection
\[\pi_{\nu(k)}: \ex S_{\nu(k)}\longrightarrow \ex X_{\nu(k)}\]
which is a submersion with fibres $\et 1{(0,\infty)}$ or $\ex T$. This projection also satisfies the following universal property:
 given any $\mathbb C^*\e{[0,\infty)}$--bundle $\ex A\longrightarrow \ex B$ with a map $h:\ex A\longrightarrow \ex S_\nu$, whose tropical part has derivative $\nu$ in fibre directions, there exists a unique commutative diagram
\[ \begin{tikzcd}\ex A\dar \rar{h} & \ex S_\nu\dar{\pi_k}
\\ \ex B\rar{ev} &\ex X_{\nu}\end{tikzcd}\]
So,   there is an  evaluation map $ev_k$ making the following diagram commute.
\[\begin{tikzcd}\ex C_k\rar{\hat f\rvert_{\ex C_k}}\dar {\pi} & \ex S_{\nu(k)}\dar{\pi_{\nu(k)}}
\\ \ex F\rar{ev_k} & \ex X_{\nu(k)}\end{tikzcd}\]
So long as $\nu$ is a primitive integral vector, the space $\ex X_{\nu}$ can be regarded as a quotient of $\ex S_\nu\subset \ex X$ by the $\mathbb C^*\e{[0,\infty)}$ action with weight $\nu$, however this construction must be treated with some care as $\mathbb C^*\e{[0,\infty)}$ is not a group. Instead this action extends uniquely to an action of  $\ex T=\mathbb C^*\e{\mathbb R}$ on the tropical completion\footnote{See \cite[Section 7]{vfc} for tropical completion.} of $\ex S_\nu$, and $\ex X_{\nu}$ is the quotient by this $\ex T$--action;  for details see \cite[Section 3]{gfgw}.
 
When $\nu(k)$ is a multiple $d_k$ of a primitive  integral vector, this evaluation map $ev_k$ factors through a map $\mathfrak{ev}_k$ to an evaluation stack $\mathcal X_{\nu(k)}$, where the map $\mathcal X_{\nu(k)}\longrightarrow \ex X_{\nu(k)}$ has  fibres the classifying stack of the cyclic group $\mathbb Z_{d_k}$ of order $d_k$. One complication is that this evaluation stack is not the classifying stack for appropriate maps of $\mathbb C^*\e{[0,\infty)}$--bundles, because this classifying stack is not sufficiently well behaved. Instead, it is the classifying stack for maps of $\ex T$--bundles into the tropical completion of $\ex S_{\nu(k)}$, and is the quotient of this tropical completion by a $\ex T$--action with weight $\nu(k)$. Our evaluation map $ev_k$ then factorises as follows.
\[\begin{tikzcd}\ex F\rar{\mathfrak{ev}_k} \ar{dr}[swap]{ev_k}&\mathcal X_{\nu(k)}\dar
\\ & \ex X_{\nu(k)}\end{tikzcd}\]
To construct $\mathfrak {ev}$, one extends a $\mathbb C^*\e{[0,\infty)}$--bundle to a $\ex T$--bundle using a relative version of tropical completion, then extends the original map to $\ex S_{\nu(k)}$ to a map of this $\ex T$--bundle to the tropical completion, obtaining a family in the stack $\mathcal X_{\nu(k)}$.

\begin{example}Consider $\mathbb CP^3$, with its toric boundary divisor. We have coordinates $\tilde z_1$, $\tilde z_2$, $\tilde z_3$ on $\ex X=\expl \mathbb CP^3$ taking values in $\mathbb C^*\e{\mathbb R}$. Corresponding to the vector $\nu=(1,1,2)$, we have the stratum $\ex S_\nu\subset \ex X$ where $\totb{z_i}>0$, whose smooth part is the point $(0,0,0)\in \mathbb CP^3$. To construct the evaluation space $\ex X_\nu$, sit this stratum $\ex S_\nu$ inside its tropical completion,  $\ex T^3$, which has the same coordinates $\tilde z_i$. Then  define $\ex X_\nu$ as the quotient of $\ex T^3$ by the $\ex T$--action of weight $(1,1,2)$, and let $\pi_\nu$ be the restriction of this quotient map to our statum. So, $\ex X_\nu$ is isomorphic to $\ex T^2$, and the projection from our stratum given by functions $\tilde z_2/\tilde z_1$ and $\tilde z_3/\tilde z_1^2$. 
\[\begin{tikzcd} \ex S_\nu\rar[hook]\ar{dr}{\pi_\nu} & \ex T^3\dar{(\tilde z_2/\tilde z_1,\tilde z_3/\tilde z_1^2)}
\\ & \ex T^2=\ex X_\nu
\end{tikzcd}\]

For a multiple  $\nu'$ of $\nu$, such as $(3,3,6)$, we define $\ex X_{\nu'}=\ex X_{\nu}$, but we also have the evaluation stack $\mathcal X_{\nu}$ which is the quotient of $\ex T^3$ by the $\ex T$--action of weight $(3,3,6)$.
\end{example}

The construction of $\mathcal X_{\nu}$ as a quotient gives a canonical map $\varpi_\nu: \ex S_\nu\longrightarrow \mathcal X_{\nu}$ factorising $\pi_\nu$.  

\[\begin{tikzcd} \ex C\dar \rar{\hat f} & \ex S_{\nu(k)}\dar[swap]{\varpi_{\nu(k)}}\ar[bend left]{dd}{\pi_{\nu(k)}} 
\\ \ex F\rar{\mathfrak{ev}_k} \ar[swap]{dr}{ev_k}&\mathcal X_{\nu(k)}\dar
\\ & \ex X_{\nu(k)}\end{tikzcd}\]

We can define a holomorphic symplectic form $\omega_{\nu}$ on $\mathcal X_{\nu}$ as follows:
\[\omega_{\nu}:=(\varpi_{\nu})_!d\rho\wedge \Omega \text{ on }\mathcal X_\nu\]
where $\rho$ is any smooth $\mathbb C$--valued function on a refinement of $\ex S_\nu$ that, restricted to each fibre, is $0$ where $\totb z$ is small, and $\frac i{2\pi}$ where $\totb z$ is sufficiently large (here, $z$ is a standard coordinate on a fibre of $\varpi_\nu$, which is isomorphic to either $\et 1{(0,\infty)}$ or $\ex T$). We will represent $\omega_{\nu}$ by a holomorphic symplectic form on the exploded manifold $\ex X_{\nu}$, which we will also call $\omega_{\nu}$. So that these two forms have the same pullback to $\ex F$, we define
\[\omega_{\nu}:=\abs\nu (\pi_\nu)_!d\rho\wedge\Omega\]
where $\nu$ is $\abs\nu$ times a primitive integral vector, so $\ex X_\nu$ is locally a $\abs\nu$--fold cover of $\mathcal X_\nu$. In other words,   $\omega_\nu$ on the stack $\mathcal X_\nu$ is the pullback of the form $\omega_\nu$ on the exploded manifold $\ex X_\nu$.

\begin{lemma}As defined above, $\omega_\nu$ does not depend on the choice of $\rho$, and is a holomorphic symplectic form on $\ex X_{\nu}$. 
\end{lemma}
\begin{proof} We can define the pushforward $\pi_!$ of $d\rho\wedge\Omega$ using \cite[Thm 9.2]{dre}. Given any other choice of  function $\rho'$ satisfying the same conditions as $\rho$, $(\rho-\rho')\Omega$ is also a refined form that can be pushed forward using $\pi$, so, using that the pushforward is a chain map, we get that 
\[\pi_!d\rho\wedge\Omega-\pi_! d\rho'\wedge\omega=d\pi_!(\rho-\rho')\Omega\ .\]
As $\Omega$ is holomorphic and fibres of $\pi$ are holomorphic, $\pi_!(\rho-\rho')\Omega=0$, and we can conclude that $\omega_\nu=\abs\nu\pi_!d\rho\wedge \Omega$ is independent of the choice of $\rho$.

The projection $\pi_\nu:\ex S_\nu\longrightarrow \ex X_\nu$ is a submersion that is also surjective on integral vectors, so we can use \cite[Thm 6.1]{dre} to push forward differential forms, without resorting to refined forms. Despite this,  $\rho$ is a function on some refinement $\ex S'$ of $\ex S_\nu$, so we only know that the pushforward of $d\rho\wedge \Omega$ coincides with a form on $\ex X_\nu$ in regions  where the projection from $\ex S'$ is surjective on integral vectors. Nevertheless, for any point $p$ in $\totb{\ex X_\nu}$, it is easy to construct an appropriate $\rho$ on a refinement of $\ex S_\nu$ whose projection to $\ex X_\nu$ satisfies the required condition in a neighbourhood of $p\in \totb{\ex X_\nu}$. So, we can conclude that $\omega_\nu$ is actually a form on $\ex X_\nu$.

Moreover, the definition the push forward implies that $\omega_\nu$ can also be characterised by
\begin{equation}\label{omeganui}\pi_\nu^*\omega_\nu= \frac 1i \iota_{\partial_{\theta_\nu}}\Omega \end{equation}
where $\partial_{\theta_\nu}$ is the vector field generating the action of rotations using our   $\mathbb C^*\e{[0,\infty)}$ action of weight $\nu$. As this action is holomorphic and preserves the closed holomorphic volume form $\Omega$,  we get that $\pi^*\omega_\nu$ is holomorphic and closed, and therefore $\omega_\nu$ is also holomorphic and closed. Moreover, as $\Omega$ is a holomorphic volume form, it follows that $\omega_\nu$ is also nondegenerate, and therefore is a holomorphic symplectic form. 

\end{proof}

\subsubsection{Pullbacks using evaluation maps as residues.}

\

To summarise, we now have a holomorphic symplectic form $\omega_{\nu}$ on the stack $\mathcal X_\nu$ and the exploded manifold $\ex X_\nu$. We can pull back  $\omega_{\nu}$ using the evaluation map $\mathfrak {ev}$ or $ev$, with the result not depending on which evaluation map we choose. Next, we will identify this pullback as a kind of residue on our family of curves. 
\[\begin{tikzcd}\ex C\dar{\pi} \rar{\hat f}& \ex X
\\ \ex F \end{tikzcd}\]
Define
\[Res_{k}\Omega:=\pi_! d\rho_k \wedge {\hat f}^*\Omega\]
where $\rho_k: \ex C\longrightarrow \mathbb C$ is any $\C\infty 1$ function on some refinement of $\ex C$, supported on $\ex C_k$, and satisfying the condition that  $\rho_k=\frac i{2\pi}$ sufficiently far into $\ex C_k$, so, using the $\mathbb C^*\e{[0,\infty)}$ action on $\ex C_k$, for all points $p\in \ex C_k$,  $\rho_k(c\e a*p)=\frac i{2\pi}$ for $a$ large enough.  In the special case that $\ex C$ is a family of curves with domain $\ex T$, we also require that $\rho_k(c\e a*p)=0$ for $a<0$ sufficiently negative. Applying \cite[Thm 9.2]{dre}, we have that $Res_k\Omega$ is a form on a refinement of $\ex F$, but we will see in Lemma \ref{Res is pullback}, that $Res_{k}\Omega$ is also a form on $\ex F$. Moreover, as $\hat f$ is holomorphic restricted to each fibre of $\ex C_k$, $\hat f^*\Omega$ is holomorphic on fibres of $\ex C_k$, so $Res_k\Omega$ is independent of the choice of $\rho_k$.  An alternate characterisation of $Res_k\Omega$ is that
\begin{equation}\label{reski}(\pi^* Res_k\Omega)\rvert_{\ex C_k}= \frac 1 i \iota_{\partial_{\theta_k}}\hat f^*\Omega\end{equation}
where $\partial_{\theta_k}$ is the vector field on $\ex C_k$ generating the action of the circle within the $\mathbb C^*\e{[0,\infty)}$--action on $\ex C_k$. A second equivalent characterisation of $Res_k$ is that locally on $\ex C_k$ 
\[\frac{dz}{z}\wedge \pi^* Res_k\Omega= \hat f^*\Omega\]
where $z$ is any local $\mathbb C^*\e{\mathbb R}$--valued  holomorphic coordinate on $\ex C$ such that the $\mathbb C^*\e{[0,\infty)}$ action on $\ex C_k$ corresponds to the standard action on $z$.

\begin{lemma}\label{Res is pullback} Given a $\C\infty1$ family of curves \[\begin{tikzcd}\ex C\rar{\hat f}\dar {\pi} & \ex X
\\ \ex F\end{tikzcd}\]
with marked ends,
\[ev_k^*\omega_{\nu(k)}=\mathfrak{ev}_k^*\omega_{\nu(k)}=Res_k\Omega \ .\]
\end{lemma}

\begin{proof} 
The commutative diagram
\[\begin{tikzcd}\ex C_k\dar{\pi} \rar{\hat f\rvert_{\ex C_k}} & \ex  S_{\nu(k)}\dar{\varpi_{\nu(k)}}
\\ \ex F\rar{\mathfrak{ev}_k} & \mathcal X_{\nu(k)}\end{tikzcd}\]
is not quite a fibre product diagram unless we apply tropical completion to $\ex C_k$ and $\ex S_k$, but it is similar to a fibre product diagram so that we could expect the result to follow from the compatibility of pushforwards with fibre products, \cite[Lemma 9.3]{dre}. 

Instead of fleshing out the above argument, we use the characterisation of $\omega_{\nu(k)}$ and $Res_k\Omega$ from 
 equations (\ref{omeganui}) and (\ref{reski}). Then, we have
 \[\pi^*Res_k\Omega=\frac 1 i \iota_{\partial_{\theta_k}}\hat f^*\Omega =\hat f^*\left(\frac 1 i \iota_{\partial_{\theta_{\nu(k)}}\Omega}\right)=\hat f^*\varpi_k^*\omega_{\nu(k)}=\pi^*\mathfrak {ev}_k^*\omega_{\nu(k)}\]
 so, 
 \[Res_k\Omega=\mathfrak{ev}^*_k\omega_{\nu(k)} \ .\]
 As the pullback of $\omega_\nu$ from $\ex X_\nu$ agrees with the pullback from $\mathcal X_\nu$, we also have that $Res_k\Omega=ev_k^*\omega_\nu$.
 
\end{proof}

\begin{lemma}\label{Res vanishes} If $\hat f$ is a $\C\infty 1$ family of holomorphic curves in $\ex X$, with ends labeled by $k\in \{1,\dotsc, m\}$, then
\[\sum_{k=1}^m Res_k\Omega=0\]\
\end{lemma}

\begin{proof}
There exists a refinement $\ex C'$ of the domain of our family of curves so that each of the functions $\rho_k$ used to define $Res_k$ is a function on $\ex C'$. We have that 
\[\begin{split}\sum_k Res_k\Omega&=\pi_!\left(d\sum_k \rho_k\right)\wedge \hat f^*\Omega
\\ &= \pi_!d\left( \frac 1{2\pi i}+\sum_k \rho_k\right)\hat f^*\Omega
\\ &= d\pi_!\left( \frac 1{2\pi i}+\sum_k \rho_k\right)\hat f^*\Omega
\end{split}\]
Note that $\left( \frac 1{2\pi i}+\sum_k \rho_k\right)\hat f^*\Omega$ vanishes on the ends of the curves in $\ex C'$, so we can apply \cite[Thm 6.1, Thm 9.2]{dre} to push it forward. Moreover, as $\Omega$ is holomorphic, and $\hat f$ is holomorphic restricted to fibres of $\pi$,  we have that $\left( \frac 1{2\pi i}+\sum_k \rho_k\right)\hat f^*\Omega$ vanishes whenever we insert two vectors tangent to the fibres of $\pi$. So 
\[\pi_!\left( \frac 1{2\pi i}+\sum_k \rho_k\right)\hat f^*\Omega=0\]
and we can conclude that $\sum_k Res_k\Omega=0$.
\end{proof}

\subsubsection{Evaluation stacks without labelling contact data, and the ring of lagrangian cycles}

\

We now describe the evaluation spaces and stacks for the moduli stack of curves in $\ex X$.

Given a family of curves in $\ex X$, parametrised by $\ex F$,  with $m$ labeled ends with contact data $\nu(1),\dotsc,\nu(m)$ we have the evaluation map
\[ev=(ev_1,\dotsc, ev_m): \ex F\longrightarrow \prod_{k=1}^m (\ex X_{\nu(k)},\omega_{\nu(k)})\ .\]
Lemmas \ref{Res is pullback} and \ref{Res vanishes} imply that, if this is a family of holomorphic curves, then $ev^*(\sum_k \omega_{\nu(k)})$ vanishes. 

Without choosing contact data, the evaluation space for a curve with a single end is
\[\rend (\ex X):=\coprod_\nu (\ex X_\nu,\omega_\nu) \]
with corresponding stack
\[\End (\ex X):=\coprod_\nu (\mathcal  X_\nu,\omega_\nu)\ .\]
With $m$ labeled ends, the corresponding evaluation space or stack is the holomorphic symplectic exploded manifold $\rend(\ex X)^m$ or stack $\End(\ex X)^m$, so with $m$ unlabelled ends we quotient by the action of the symmetric group $S_m$, and have the natural evaluation maps
\[ev:\ex F\longrightarrow \rend(\ex X)^m/S_m\]
\[\mathfrak{ev}: \ex F\longrightarrow \End(\ex X)^m/S_m\]
whereas, if we don't specify the number of ends of our family of curves, we have the evaluation maps
\[ev: F\longrightarrow \exp\lrb{ \rend(\ex X) }:=\coprod_{m=0}^\infty \rend(\ex X)^m/S_m  \]
\[\mathfrak{ev}: F\longrightarrow \exp\lrb{ \End(\ex X) }:=\coprod_{m=0}^\infty \End(\ex X)^m/S_m  \]
The connected components of $\exp(\rend \ex X)$ are labeled by the number of copies of $\ex X_\nu$ appearing in the product. Let $\mathbf p(\nu)\in\mathbb N$ be the number of copies of $\ex X_\nu$. Contact data consists of a choice of number $\mathbf p(\nu)$ for each $\nu$, such that all but finitely many of these numbers are $0$. Then we can define
\[(\ex X^{\mathbf p},\omega_{\mathbf p}):=\prod_{\nu} (\ex X_{\nu},\omega_\nu)^{\mathbf p(\nu)}\] 
\[(\mathcal X^{\mathbf p},\omega_{\mathbf p}):=\left(\prod_{\nu} (\mathcal X_{\nu}^{\mathbf p(\nu)},\omega_{\mathbf p})^{\mathbf p(\nu)}\right)/\Aut\bf p\]
and we have that 
\[\exp(\rend \ex X)=\coprod_{\mathbf p}\ex X^{\mathbf p}/\Aut\mathbf p:=\coprod_{\mathbf p} \prod_\nu \ex X^{\mathbf p(\nu)}_\nu/S_{\mathbf p(\nu)}\]
and similarly
\[\exp(\End (\ex X))=\coprod_{\mathbf p} \mathcal X^{\bf p}\ .\]

Lagrangian cycles on $\exp(\End(\ex X))$ naturally form a ring. Following Definition \ref{stack star product} for the definition of lagrangian cycles in the stack $\exp(\End(\ex X))$, a lagrangian cycle  $\alpha\in \Lag{\exp\End(\ex X)}^-$ is, for all choices $\bf p$ of contact data,  an $\Aut \bf p$--invariant lagrangian cycle 
\[\alpha_{\bf p}\in  \Lag{\ex X^{\bf p}}\]
Moreover, such a lagrangian cycle $\alpha$ is also in $\Lag{\exp\End(\ex X)}$ if $\alpha_{\bf p}=0$ for all but finitely many $\bf p$, and $\alpha_{\bf p}$ is divisible by $\prod_v \abs v^{\mathbf p(v)}$. We can define multiplication of lagrangian cycles by

\begin{equation}(\alpha  \beta)_{\bf r}:=\sum_{\bf p+q=r}\frac 1{\abs{\Aut \bf p}\abs{\Aut{\bf q}}}\sum_{\sigma\in\Aut{\bf r}}\sigma(\alpha_{\bf p}\times \beta_{\bf q})\end{equation}
where $\sigma\in \Aut\bf r$ acts on $\ex X^{\bf r}$ by permuting coordinates. This multiplication turns both $\Lag{\exp\End(\ex X)}^-$ and $\Lag{\exp\End(\ex X)}$ into commutative rings.

\subsubsection{Logarithmic evaluation space is compatible with log modifications.}

\

When $\ex X=\expl (X,D)$, we have a logarithmic description of $\ex X_\nu$ whenever $\nu$ is a multiple of the primitive integral vector corresponding to a component of the divisor $D$, so we have a logarithmic description of $\ex X^{\bf p}$ for contact data $\bf p$ if and only if $\mathbf p(\nu)=0$ whenever $\nu$ is not in this form. When contact data is not in this form, we need to choose a logarithmic modification of $(X,D)$. 

Given any refinement  $\ex X'\longrightarrow \ex X$, respecting the cone structure on $\totb{\ex X}$,  the evaluation space $\ex X_{\nu}'$ is a refinement of $\ex X_\nu$, so we get corresponding refinement maps $\rend(\ex X')\longrightarrow \rend(\ex X)$ and $\End(\ex X')\longrightarrow \End(\ex X)$. Given any family of curves in $\ex X'$, composing with the refinement map $\ex X'\longrightarrow \ex X$ gives a family of curves in $\ex X$, so we have a natural map from the moduli stack of curves in $\ex X'$ to the moduli stack of (possibly unstable) curves in $\ex X$. For an individual curve, we can then stabilise this curve, \cite[Lemma 4.8]{evc}, to obtain a stable curve in $\ex X$. This provides a bijection between isomorphism classes of stable holomorphic curves in $\ex X$ and $\ex X'$.

\begin{lemma}\label{moduli refinement}  Given a refinement $\ex X'\longrightarrow \ex X$, there is a natural bijection between the set of isomorphism classes of stable holomorphic curves  in $\ex X$ and  $\ex X'$. Each stable holomorphic curve $f$ in $\ex X$ corresponds to a unique stable holomorphic curve $f'$ such that the following is a fibre product diagram. 
\begin{equation}\label{curve refinement}\begin{tikzcd}
\ex C' \rar{ f'}\dar & \ex X'\dar
\\ \ex C\rar{f} & \ex X 
\end{tikzcd}\end{equation}

\end{lemma}

\begin{proof} 

Given a stable holomorphic  curve $f:\ex C\longrightarrow \ex X$, taking fibre products  as in (\ref{curve refinement}) gives a stable holomorphic curve $f': \ex C\longrightarrow \ex X'$. The map $\ex C'\longrightarrow \ex C$ is a refinement map, obtained by subdividing the edges of the tropical part $\totb{\ex C}$ of $\ex C$ using the inverse image of the subdivision of $\totb{\ex X}$ given by the refinement of $\ex X'\longrightarrow \ex X$; see \cite[Lemma 10.7]{iec}. The corresponding map $\ex C'\longrightarrow \ex X$ is unstable if this refinement $\ex C'\longrightarrow \ex C$ is non-trivial; see \cite[Definition 10.9]{iec}. Moreover,  the stabilisation from \cite[Lemma 4.8]{evc} is exactly the original curve  $f:\ex C\longrightarrow \ex X$.

Similarly, given a stable curve $f':\ex C'\longrightarrow \ex X'$, the composition of this curve with the refinement map $\ex X'\longrightarrow \ex X$ has no extra automorphisms, so is unstable if and only if it is a nontrivial refinement of a stable curve  $f:\ex C\longrightarrow\ex X'$. Let $\ex C''$ be the fibre product of $\ex C$ with $\ex X'$. Then $f'$ factorises through $\ex C''$, so we get that $\ex C'$ is isomorphic to $\ex C''$, and the stable curve we obtain using the fibre product is the original stable curve $f'$.

\end{proof}

\subsection{Gromov--Witten invariants as lagrangian correspondences}

\

We now have a holomorphic symplectic evaluation space, and lemmas \ref{Res is pullback} and \ref{Res vanishes} tell us that $ev^*\omega$ vanishes on each family of holomorphic curves in $\ex X$. To argue that the image of the moduli stack of curves is isotropic, we also need some algebraic properties of the moduli stack of curves. 

\begin{prop}\label{explodable family} Suppose that $\mathrm X$ is a log smooth, proper log scheme. Then, there exists an explodable family of stable log curves 
\[\begin{tikzcd} \rm C\dar \rar{f} & \mathrm X
\\ \rm F \end{tikzcd}\]
in $\mathrm X$ such that the corresponding family of exploded curves
\[\begin{tikzcd} \expl \rm C\dar \rar{\expl f} & \expl \mathrm X
\\ \expl \rm F \end{tikzcd}\]
contains every isomorphism class of stable exploded curve in $\expl \rm X$. Moreover, the corresponding map of $ F$ to the moduli stack of stable basic log curves in $\rm X$ is proper. 
\end{prop}

\begin{proof} Our proof uses the algebraic properties of the moduli stack of basic log curves in $\rm X$,  \cite{GSlogGW}. Applying the explosion functor to a basic stable log curve in $\rm X$ gives a family of stable exploded curves in $\expl \rm X$ with universal tropical structure; see \cite[Section 7]{elc}, \cite{uts}. Moreover, every stable holomorphic curve in $\expl  X$ is contained in a family of stable holomorphic curves with universal tropical structure, which can also be obtained by applying the explosion functor to a basic log curve. Use the notation $\mathcal M(\rm X)$ for the moduli stack of  basic stable log curves in $\rm X$, and $\mathcal M(\expl X)$ for the moduli stack of stable holomorphic curves in $\expl X$. This moduli stack $\mathcal M(\rm X)$ is a Deligne-Mumford stack over the category of ordinary complex schemes, \cite[Corollary 2.8]{GSlogGW}, with some extra structure encoding log structure, \cite[Corollary 2.6]{GSlogGW}. We will regard our moduli stack  $ \mathcal M(\expl \rm X)$ as the explosion of $\mathcal M (\rm X)$. One complication is that $\mathcal M(\expl \rm X)$ is a stack over the category of $\C\infty 1$ exploded manifolds, whereas $\mathcal M(\rm X)$ is a stack over the category of schemes. 

As $\mathcal M(\rm X)$ is Deligne--Mumford, there exists a proper surjective map from a scheme $F_0$ to $\mathcal M(\rm X)$; \cite{Olsson2005}. From this map, $F_0$ inherits a fs log structure $\rm F_0$ and parametrises a family of basic stable log curves in $\rm X$.
\[\begin{tikzcd} \mathrm C_0\dar\rar{f_0} & \mathrm X
\\ \mathrm F_0 \end{tikzcd}\]
Moreover, the above family contains every isomorphism class of basic stable log curve in $\mathrm X$. Applying Proposition \ref{explodable resolution}, there exists a  proper map $\mathrm F\longrightarrow \mathrm F_0$ from an explodable log scheme $\mathrm F$ whose explosion is surjective. 

Pulling back our family of basic stable log curves, we get an explodable family of stable\footnote{Because our log morphism $\mathrm F\longrightarrow \mathrm F_0$ need not be strict, this family of curves $f$ need not be a family of basic log curves.} log curves parametrised by $\mathrm F$.
\[\begin{tikzcd} \mathrm C\ar[bend left]{rr}{f} \rar\dar &\mathrm C_0\dar\rar & \mathrm X
\\\mathrm F\rar&  \mathrm F_0 \end{tikzcd}\]

 As $\mathrm F$ is explodable, $\rm C$ is too, so $f$ is an explodable family of curves. Moreover, the corresponding exploded family of curves
\[\begin{tikzcd} \expl \rm C\dar \rar{\expl f} & \expl \mathrm X 
\\ \expl \rm F \end{tikzcd}\]
contains every isomorphism class of stable holomorphic curve in $\expl \rm X$. This is because each such stable holomorphic curve $g$ in $\expl \rm X$ is contained in a canonical family  of holomorphic curves that is the explosion of a stable basic log curve $g'$ over a log point. A representative of this stable basic curve is contained in our original family $f_0$ of curves, so, a curve isomorphic to $g$ is contained in our family $\expl f$, because the map  $\expl \rm F\longrightarrow \expl F_0$ is surjective.

\end{proof}

\begin{prop}\label{isotropic GW image} Suppose that $\ex X$ is a Calabi--Yau exploded manifold with a refinement isomorphic to the explosion of some projective, compact log Calabi--Yau 3-fold.  Let $\mathcal M_{\bullet}(\ex X)$ be the moduli stack of connected stable holomorphic curves in $\ex X$ with $m$ labeled ends, contact data $\nu(1),\dotsc, \nu(m)$,  fixed genus, and representing some fixed homology class. Then, 
 \[ev(\mathcal M_{\bullet}(\ex X))\subset \prod_{k=1}^m(\ex X_{\nu(k)},\omega_{\nu(k)})\] is  logarithmically proper and isotropic. 
\end{prop}

\begin{proof}

Lemma \ref{moduli refinement} and  Corollary \ref{inverse image} imply that it is enough to prove that $ev(\mathcal M_\bullet (\ex X'))$ is logarithmically proper and isotropic for some refinement $\ex X'$ of $\ex X$. We take $\ex X'$ to be a refinement isomorphic to $\expl \rm X$ for a smooth, projective, proper log Calabi--Yao 3-fold $\rm X$. We have that the corresponding moduli stack of basic stable log curves in $\rm X$ is proper, \cite[Corollary 4.2]{GSlogGW}, so  Proposition \ref{explodable family} implies that there exists a proper explodable family of stable log curves $f$  in $\rm X$ whose explosion  contains every isomorphism class of curve in $\mathcal M_\bullet(\expl\rm X)$.
\[\begin{tikzcd}\rm C\dar \rar{f} & \rm X
\\ \rm F\end{tikzcd}\]
 In particular this implies that $ev(\expl \mathrm F)=ev(\mathcal M_\bullet(\expl\mathrm X))$.
 
To verify that $ev(\expl \mathrm F)$ is logarithmically proper, we need that $ev$ is the explosion of some logarithmic map $\mathrm{ev}$. Assume that each vector $\nu(k)\in\totb{\expl \rm X}$ corresponds to contact with only one component $D_k$ of the divisor $D\subset X$. This can be achieved by taking a further log modification of $\rm X$ if necessary. Then $\ex X'_{\nu(k)}$ is the explosion of the log scheme $\mathrm D_k$, given by $D_k$, relative to its intersection with other components of $D$.  Evaluation at marked points gives an algebraic evaluation map 
\[\mathrm {ev}:\mathrm F\longrightarrow \prod_k \mathrm D_k\]
 compatible with our evaluation map to $\prod_k\ex X'_{\nu(k)}$
in the sense that 
\[\expl \mathrm{ev}=ev\]
In particular, each marked point section gives a commutative diagram of log morphisms
\[\begin{tikzcd}\mathrm C_k\dar \rar{f} &\mathrm S_k\dar
\\  \rm F\rar{\mathrm{ev}_k} & \mathrm D_k \end{tikzcd}\]
where $\mathrm C_k\subset \rm C$ is the locus of the $k$th marked point, and $\mathrm S_k\subset \rm X$ is the closed stratum of $\rm X$ over $D_k$. As a map of underlying schemes, $\mathrm {ev}_k$ and $f$ agree, and the map of log structures is determined by the following commutative diagram of short exact sequences
\[\begin{tikzcd}  0 & 0
\\ \mathbb N\uar & \mathbb N\lar{d_k}\uar
\\ M_{\mathrm C_k}\uar &\lar{f^\flat} M_{\mathrm S_k}\uar
\\ M_{\mathrm F}\uar & \lar{\mathrm{ev}_k^\flat} M_{\mathrm D_k}\uar
\\ 0\uar & 0\uar \end{tikzcd}\]
Then the required logarithmic evaluation map  $\rm {ev}$ is the product of these maps $\mathrm {ev}_k$.

Proposition \ref{explodable family} then implies that $ev\left(\mathcal M_\bullet(\expl \rm X)\right )=ev(\expl \rm F)=\expl\mathrm {ev}(\expl \rm F)$ is logarithmically proper, which implies that $ev(\mathcal M_{\bullet}(\ex X))$ is also logarithmically proper. 
 Lemmas \ref{Res is pullback} and \ref{Res vanishes} imply that $ev(\expl \mathrm F)$ is isotropic, so $ev(\mathcal M_{\bullet}(\ex X))$ is also isotropic.

\end{proof}

\begin{remark}We can now describe log Gromov--Witten invariants of Calabi--Yau 3-folds as canonical lagrangian correspondences. Let $\rm X=(X, D)$ be a proper projective log Calabi-Yau 3-fold, and let $\mathcal M_\bullet (\rm X)$ be the moduli stack of basic stable log curves in $\rm X$ representing a fixed homology class,  with specified genus, and contact data $\nu(1),\dotsc,\nu(m)$, where $\nu(k)$ indicates contact of order $d_k$ with a single component $D_k$ of the divisor $D$. We then have the evaluation map
\[\mathrm{ev}:\mathcal M_\bullet (\mathrm X)\longrightarrow \prod_{k=1}^m \mathrm D_k\]
constructed in the proof of Proposition \ref{isotropic GW image}. As the image of $\rm ev$ is isotropic, each irreducible component has dimension bounded by $m$, and those components  contained in the boundary of $\prod_k\mathrm D_k$ have dimension strictly less than $m$.   The virtual fundamental class $[\mathcal M_\bullet(X)]$ constructed in \cite{GSlogGW} is an element of the (rational) Chow group  $A_m(\mathcal M_\bullet(\mathrm X);\mathbb Q)$, where we regard $\mathcal M_\bullet(\mathrm X)$ as a Deligne--Mumford stack over the category of complex schemes, (not using the log structure). When we forget the log structure on $\mathrm D_k$ we obtain $D_k$ and we have the evaluation map
\[\mathrm{ev}:\mathcal M_\bullet (\mathrm X)\longrightarrow \prod_{k=1}^m  D_k\]
and we can push forward $[\mathcal M_\bullet(\mathrm X)]$ to obtain 
\[\mathrm{ev}_*[\mathcal M_\bullet(\mathrm X)]\in A_{m}\left(\mathrm{ev}(\mathcal M_\bullet(\mathrm X));\mathbb Q\right)\] as a $m$--dimensional element of the rational Chow group of 
\[\mathrm{ev}(\mathcal M_\bullet(\mathrm X))\subset\prod_{k=1}^m D_k\ .\]
This $m$--dimensional Chow group is freely generated by the irreducible lagrangian components of $\mathrm{ev}(\mathcal M_\bullet(\mathrm X))$, and there are no non-trivial rational equivalences, so the pushforward of the virtual fundamental class is uniquely represented by a lagrangian cycle in $\Lag{\prod_k\mathrm D_k}\otimes\mathbb Q$.

Using logarithmic modifications, we can obtain similar results when we relax the conditions on contact data. For a component of the moduli stack $\mathcal M (\mathrm X)$ comprised of curves with some special points intersecting multiple components of $D$, we can choose a logarithimic modification $\mathrm X'\longrightarrow \mathrm X$ so that each special point has contact with only one component of the divisor in $\mathrm X'$. Then, the above argument applies to the push forward of the virtual fundamental class of the corresponding component of $\mathcal M(\mathrm X')$. Morally speaking $[\mathcal M(\mathrm X')]$ is a logarithmic modification of $[\mathcal M(\mathrm X)]$, \cite{ilgw}, so this gives information about $[\mathcal M(\mathrm X)]$.
 \end{remark}

Our goal now is to prove that the pushforward of the virtual fundamental class of the moduli stack of curves in $\ex X$ is represented by a canonical holomorphic lagrangian correspondence. We prove this using the exploded manifold construction of virtual fundamental class. 
 
The moduli stack $\mathcal M(\ex X)$ of stable holomorphic curves in $\ex X$ is a substack of the moduli stack of stable, not-necessarily-holomorphic curves $\mathcal M^{st}(\ex X)$ in $\ex X$; \cite[Section 2,7]{evc}. The virtual fundamental class $[\mathcal M(\ex X)]$ of $\mathcal M(\ex X)$ is constructed in \cite[Definition 4.7]{vfc}, using the Kuranishi structrure from \cite[Theorem 7.3]{evc}. The virtual fundamental class $[\mathcal M(\ex X)]$ is a canonically oriented,  weighted branched family of curves in $\mathcal M^{st}(\ex X)$, which can be constructed as close as we like to the holomorphic curves $\mathcal M(\ex X)\subset\mathcal M^{st}(\ex X)$. Our evaluation map $ev$ is defined on $\mathcal M^{st}(\ex X)$ and the virtual fundamental class. We pull back differential forms using $ev$ to integrate them over $[\mathcal M(\ex X)]$, as in \cite[Definition 5.8]{vfc}, and we can also use $ev$ to push forward the virtual fundamental class, representing this pushforward as a closed differential form, supported as close as we like to $ev(\mathcal M(\ex X))$;  \cite[Definition 5.15]{vfc}. The cohomology we use for these operations is called refined cohomology, and denoted by $\rh^*$; see \cite[Section 9]{dre}. 

Let $\ex X$, $\mathcal M_\bullet (\ex X)$ and $\nu(k)$ be as in Proposition \ref{isotropic GW image}, and let $\mathcal M^{st}_\bullet(\ex X)$ be the moduli stack of stable curves in $\ex X$, with genus, homology class, and contact data the same as $\mathcal M_{\bullet}(\ex X)$. We have an evaluation map
\[ev:\mathcal M^{st}_\bullet(\ex X)\longrightarrow \prod_{k=1}^m(\ex X_{\nu(k)},\omega_{\nu(k)})\]
such that  the image of $\mathcal M(\ex X)\subset\mathcal M^{st}_\bullet(\ex X)$ is logarithmically proper and isotropic. Moreover, as $\ex X$ is a Calabi--Yau 3-fold,  the real dimension  of $[\mathcal M_\bullet(\ex X)]$ is $2m$, and therefore half the dimension $\prod_{k=1}^m\ex X_{\mu(k)}$; this dimension calculation follows from \cite[Theorem 1.2]{reg} and \cite[Theorem 6.8]{evc}. 

\begin{lemma}\label{vfc sequence} Let $[\mathcal M_\bullet(\ex X)]_i$ be any sequence of virtual fundamental classes constructed using a sequence of weighted branched sections  converging to $\dbar$ in $C^1$. Then, for any open neighbourhood $U$ of $ev(\mathcal M_\bullet(\ex X))$ within any refinement of $\prod_k\ex X_{\nu(k)}$, there exists some $N\in\mathbb N$ such that for all $i>N$,   $[\mathcal M_\bullet(\ex X)]_i$ is  within $ev^{-1}U\subset\mathcal M^{st}(\ex X)$, and  the resulting map 
\[\rh^{2m}(U)\longrightarrow \mathbb R\]
\[\theta\mapsto \int_{[\mathcal M_\bullet]_i(\ex X)}ev^*\theta\]
is independent of $i>N$. Moreover this map is defined independent of any choices in the construction of the virtual fundamental classes $[\mathcal M_\bullet(\ex X)]_i$. 
\end{lemma}
 
\begin{proof}

We first argue that $ev^{-1}(U)$ contains an open neighbourhood of the moduli stack of holomorphic curves within $\mathcal M_\bullet^{st}$. This is not obvious, because $U$ is only a open subset of some refinement of $\prod_k\ex X_{\nu(k)}$. Given a family of curves $\hat f$ in $\mathcal M^{st}_\bullet$ parametrised by $\ex F(\hat f)$, the evaluation map $ev$ defines a map $\ex F(\hat f)\longrightarrow \prod_k \ex X_{\nu_k}$, and any refinement of $\prod_k\ex X_{\nu(k)}$ induces a refinement of $\ex F(\hat f)$. In particular, this implies that the inverse image of $U$ in $\ex F(\hat f)$ is a open subset of some refinement of $\ex F(\hat f)$. The complement of $ev^{-1}(U)$ does not contain any holomorphic curves, and has closed image in $\totb{\ex F(\hat f)}$. The closure of the complement of $ev^{-1}(U)$ consists of all curves topologically equivalent\footnote{The topology on the $\mathcal M_\bullet^{st}$ is not Hausdorff, but comes from a map to a Hausdorff topological space; see \cite[Section 2.6]{evc}. Two curves are topologically equivalent each closed substack either contains both or neither of these curves.   } to curves in the complement of $ev^{-1}(U)$. As all stable curves topologically equivalent to a holomorphic curve are themselves holomorphic curves, the closure of the complement of $ev^{-1}(U)$ contains no holomorphic curves. So, there is an open neighborhood $\mathcal U$ of the holomorphic curves in $\mathcal M_\bullet^{st}$ such that $\mathcal U$ is contained in $ev^{-1}(U)$.

As the virtual fundamental classes $[\mathcal M_\bullet]_i$ are constructed using a sequence of weighted branched sections converging to $\dbar$, and the moduli stack of holomorphic curves $\mathcal M_\bullet$ is compact, for $i$ sufficiently large, $[\mathcal M_\bullet]_i$ is contained in $\mathcal U$. Moreover, there exists some $N\in \mathbb N$ such that for all $i,j> N$, the two virtual fundamental classes $[\mathcal M_\bullet]_i$ and $[\mathcal M_\bullet]_j$ are cobordant within $\mathcal U$; see \cite[Lemma 4.5]{vfc} and \cite[Corollary 7.5]{evc} for the construction of cobordisms. So, the resulting map
\[\rh^{2m}(U)\longrightarrow \mathbb R\]
 defined by 
\[\theta\mapsto \int_{[\mathcal M_\bullet(\ex X)]_i}ev^*\theta\]
is independent of $i>N$. Similarly, \cite[Corollary 7.5]{evc} implies that this map is independent of any choices in the construction of our virtual fundamental classes.

\end{proof}

As explained in Section \ref{refined cohomology section},  we can also integrate differential forms over lagrangian correspondences, and the Poincare dual $PD(\alpha)$ of a lagrangian correspondence $\alpha$ is defined in $\rh^*_c(U)$, where $U$ is an arbitrarily small refined neighborhood of $\alpha$.

\begin{thm}\label{GW lagrangian} There exists a unique lagrangian correspondence 
\[\eta_\bullet\in \Lag{\prod_k\ex X_{\nu(k)}}\otimes \mathbb Q \] supported within $ev(\mathcal M_\bullet(\ex X))$ such that for any neighborhood $U$ of $ev(\mathcal M_\bullet(\ex X))$ within any refinement of $\prod_k\ex X_{\nu(k)}$, and all constructions of virtual fundamental class $[\mathcal M_\bullet(\ex X)]$ sufficiently close to $\mathcal M_\bullet\subset\mathcal M^{st}_\bullet$ we have
\[\int_{\eta_\bullet}\theta=\int_{[\mathcal M_\bullet(\ex X)]}ev^*\theta\]
for all $\theta\in \rh^{2m}(U)$.
\end{thm}

\begin{proof}

Let $\ex X'$ be any refinement of $\ex X$ isomorphic to $\expl \rm Y$ for $\rm Y$ smooth, and let $U\subset  \prod_k \ex X_{\nu(k)}'$ be an open neighbourhood of $ev(\mathcal M_\bullet(\ex X'))$ such that  $\totl U$ retracts onto $\totl {ev(\mathcal M_\bullet(\ex X'))}$. So, we have $H^{2m}(\totl{U};\mathbb R)=H^{2m}(\totl {{ev(\mathcal M_\bullet(\ex X'))}};\mathbb R)$, and the universal coefficient theorem then gives us that $H^{2m}(\totl{U};\mathbb R)$ is isomorphic to $\hom (H_{2m}(\totl {ev(\mathcal M_\bullet(\ex X'))}),\mathbb R)$. Then, \cite[Corollary 4.2]{dre} gives that $H^*(\totl{U};\mathbb R)$ is isomorphic to the de Rham cohomology of $ U$ defined in \cite[Definition 1.3]{dre}. Use $H^*(U)$ to denote this de Rham cohomology of $U$. As $H_{2m}(\totl {ev(\mathcal M_\bullet(\ex X'))})$ is freely generated by the lagrangian subvarieties, we also have that $\hom(H^{2m}(U),\mathbb R)$ is isomorphic to $H_{2m}(\totl {ev(\mathcal M_\bullet(\ex X'))};\mathbb R)$, and each element of $H_{2m}(\totl {ev(\mathcal M_\bullet(\ex X'))};\mathbb R)$ is represented by a unique element of $\Lag{\prod_k\ex X_{\nu(k)}}\otimes\mathbb R$ supported on ${ev(\mathcal M_\bullet(\ex X'))}$. Constructing our virtual fundamental class $[\mathcal M_\bullet (\ex X)]$ with image in $U$, the map
\[\theta\mapsto \int_{[\mathcal M_\bullet(\ex X)]}ev^*\theta\] 
defines a homomorphism from $H^{2m}(U)$ to $\mathbb R$, so there exists a unique \[\eta_\bullet\in \Lag{\prod_k\ex X_{\nu(k)}}\otimes\mathbb R\] such that 
\[\int_{[\mathcal M_\bullet(\ex X)]}ev^*\theta=\int_{\eta_\bullet}\theta\]
for all closed forms  $\theta\in\Omega^{2m}U$. Moreover, Lemma \ref{vfc sequence} implies that, so long as we construct $[\mathcal M_\bullet(\ex X)]$ using a small enough perturbation of $\dbar$, $\eta_\bullet$ is independent of the construction of virtual fundamental class. For any smaller open neighbourhood $U'$ of ${ev(\mathcal M_\bullet(\ex X))}$ in a further refinement of $\ex X'$, we can pull back forms from $U$ to $U'$, so $\eta_\bullet$ is the only option for a lagrangian cycle satisfying the analogous property for forms on $U'$. It follows that $\eta_\bullet$ satisfies the above property for all differential forms defined on any neighbourhood of ${ev(\mathcal M_\bullet(\ex X))}$ in any refinement of $\ex X$. As each element of $\rh^*(U)$ can be represented by a differential form on some refinement of $U$, we have that $\eta_\bullet$ satisfies the required property. 

It remains to show that $\eta_\bullet\in \Lag{\prod_k\ex X_{\nu(k)}}\otimes\mathbb Q\subset \Lag{\prod_k\ex X_{\nu(k)}}\otimes\mathbb R$. Let $\theta$ be Poincare dual to some proper $2m$ dimensional submanifold $N\subset U$ intersecting ${ev(\mathcal M_\bullet(\ex X))}$ transversely exactly once on some irreducible lagrangian component of ${ev(\mathcal M_\bullet(\ex X))}$. The weight assigned to this component is $\int_{[\mathcal M_\bullet]}ev^*\theta$.  We can also construct $[\mathcal M_\bullet(\ex X)]$ transverse to $N$ using \cite[Remark 3.18]{vfc}  and \cite[Theorem 7.3]{evc}, so that this integral coincides with the weighted count of points in $[\mathcal M_\bullet]\cap ev^{-1}N$. As $[\mathcal M_\bullet]$ is constructed using rational weights, this is a rational number. It follows that $\eta_\bullet\in\Lag{\prod_k\ex X_{\nu(k)}}\otimes\mathbb Q$.

\end{proof}

\subsection{ Gromov--Witten generating functions as lagrangian correspondences}\label{generating function section}

\

Let $\ex X$ be an exploded  Calabi--Yau 3-fold with a choice of $0\in\totb {\ex X}$ and a refinement isomorphic to the explosion of a proper projective log Calabi--Yau 3-fold $\rm X$. Define $\rh_*(\ex X)$ to be the limit of the homology of $ X'$ over all smooth log modifications $\rm X'\longrightarrow X$. Each curve in $\ex X$ represents a well-defined element $\beta\in\rh_2(\ex X)$. The curve class $\beta$ does not quite determine the contact data of $\bf p$ of a curve. Instead we have $\sum_k \mathbf p(kv)=\beta\cdot D_v$ for all primitive integral vectors $v\in\totb{\ex X}$, where $D_v$ is the divisor corresponding to $v$ in some smooth log modification $\rm X'$ of $\rm X$ such that $\beta$ represents a homology class in $\rm X'$.

Given contact data $\bf p$, a homology class $\beta$ and a genus $g$,  Theorem \ref{GW lagrangian} gives a canonical, $\Aut \bf p$--invariant rational lagrangian correspondence 
\[\eta_{g,\beta,\bf p}\in \Lag{\ex X^{\bf p}}\otimes \mathbb Q\] 
satisfying the condition that 
\begin{equation}\int_{\eta_{g,\beta,\bf p}}\theta=\left(\prod_v \abs {v}^{\mathbf p(v)}\right)\int_{[\mathcal M_{g,\beta,\bf p}]}ev^*\theta\end{equation}
for all $\theta\in \rh^*(\ex X^{\bf p})$, where $[\mathcal M_{g,\beta,\bf p}]$ indicates the virtual fundamental class of the moduli stack of stable, connected, genus $g$ holomorphic curves in $\ex X$ representing the curve class $\beta$, and with labeled contact data $\bf p$ --- so this moduli stack is an $\abs{\Aut \bf p}$--fold cover of  a component of $\mathcal M(\ex X)$, where $\abs{\Aut\bf p}=\prod_v \mathbf p(v)!$ and the square in the diagram below is a fibre product diagram. The extra factor of $\prod_v \abs {v}^{\mathbf p(v)}$ is because we want $\eta_{g,\beta,\bf p}$ for all $\bf p$ to encode the pushforward of $[\mathcal M_{g,\beta}(\ex X)]$ to $\prod_{\bf p}\mathcal X^{\bf p}$ rather than $\prod_{\bf p}\ex X^{\bf p}/\Aut \bf p$. This factor occurs because the fibres of the map $\mathcal X^{\bf p}\longrightarrow \ex X^{\bf p}/\Aut \bf p$ are $\prod_v (B_{\mathbb Z_{\abs v}})^{\mathbf p(v)}$. These weights ensure that that the tropical gluing formula for Gromov--Witten invariants, \cite{gfgw}, uses the star product over the stacks $\mathcal X^{\bf p}$, and ensure compatibility with the conjectured relationship between Donaldson Thomas invariants and Gromov--Witten invariants.

\[\begin{tikzcd}& \mathcal M_{g,\beta}(\ex X)\dar \ar{dl}& \lar\mathcal M_{g,\beta,\bf p}(\ex X)\dar
\\ \exp(\End(\ex X)):=\coprod_{\bf p}\mathcal X^{\bf p}\rar & \coprod_{\bf p}\ex X^{\bf p}/\Aut\mathbf p & \lar \ex X^{\bf p} \end{tikzcd}\]

As $\eta_{g,\beta,\bf p}$ is $\Aut \bf p$--invariant, this lagrangian cycle determines a (rational) lagrangian cycle in $\mathcal X^{\bf p}$, as in Definition \ref{stack star product}. If we sum up the corresponding lagrangian cycles for all $\bf p$ such that  $\sum_k \mathbf p(kv)=\beta\cdot D_v$, we get a lagrangian cycle in the ring  $\Lag{\exp(\End(\ex X))}\otimes \mathbb Q$.

\[\eta_{g,\beta}\in \Lag{\exp(\End(\ex X))}\otimes \mathbb Q\]

Each curve in $\mathcal M_{g,\beta,\bf p}$ should be regarded as having topological Euler characteristic $2-2g-\abs{\bf p}$ where $\abs{\bf p}:=\sum_v \mathbf p(v)$. We can sum over the contribution of curves of different genus to get
\[\eta_{\beta}:=\sum_g \hbar^{2g-2+\abs{\bf p}}\eta_{g,\beta}\in \left(\Lag{\exp(\End(\ex X))}\otimes \mathbb Q\right)((\hbar))  \]
where, for each $\bf p$ we have
\[\eta_{\beta,\bf p}:=\sum_g \hbar^{2g-2+\abs{\bf p}}\eta_{g,\beta,\bf p}\ .\]

 Gromov compactness in this context  does not imply that  $\eta_{\beta,\bf p}$ is a finite sum of lagrangian subvarieties weighted by coefficients in $\mathbb Q((\hbar))$, only that the coefficient of $\hbar^{n}$ in  $\eta_{\beta,\bf p}$  is finite. A consequence of Conjecture \ref{connected integrality} below,  is that we also expect that $\eta_{\beta, \bf p}$ to be a finite sum of lagrangian subvarieties weighted by coefficients $\mathbb Q((\hbar))$. 

As  $\ex X$ has a refinement that is the explosion of some projective $(X,D)$, there exists an energy grading 
\[E:\rh_2(\ex X)\longrightarrow \mathbb R\] (defined by integrating a real symplectic form on $X$ over $\beta$) such that, for $\beta\neq 0$,  there are no holomorphic curves representing the class $\beta$ unless $E(\beta)\geq 1$. Let $M\subset \rh_2(\ex X)$ be the submonoid comprised of $0$ together with elements $\beta$ with $E(\beta)\geq 1$. Define the commutative ring $\mathcal F$ to be the completion of the monoid ring $\left(\Lag{\exp(\End(\ex X))}\otimes \mathbb Q\right)((\hbar))[M]$ using the energy grading and the grading coming from the exponent of $\hbar$. This ring $\mathcal F$ can be thought of as a generalisation of a Fock space. 
Gromov compactness in this context \cite{cem,GSlogGW} implies that $\eta_{g,\beta}$ is nonzero for only finitely many $(g,\beta )$ with  $g\leq R$ and $E(\beta)\leq R$. So,  we can sum  the contributions of different homology classes to obtain the formal series encoding Gromov--Witten invariants of $\ex X$ counting connected, non-constant stable holomorphic curves. 
\[\eta:=\sum_{\beta\neq 0} t^\beta \eta_\beta \in \mathcal F\]
As the energy grading of $\eta$ is bounded below by $1$, it makes sense to exponentiate $\eta$ to define a partition function $Z_{GW}(\ex X)$. This partition function encodes the Gromov--Witten invariants of $\ex X$ coming from possibly disconnected stable curves, none of which have any constant components.
\[Z_{GW}(\ex X):=\exp \eta:=\sum_{n=0}^\infty \frac 1{n!} \eta^n\in\mathcal F\]
We similarly define $Z_{GW}(\ex X)_{\beta,\chi}\in \Lag{\exp(\End(\ex X))}\otimes \mathbb Q$ representing the pushforward of the virtual fundamental class of the moduli stack of possibly disconnected stable curves in $\ex X$ with topological Euler characteristic $\chi$ and representing the homology class $\beta\in \rh_2(\ex X)$. We can also sum over different topological Euler characteristics to  define
\[Z_{GW}(\ex X)_\beta:=\sum_\chi \hbar^{-\chi} Z_{GW}(\ex X)_{\chi,\beta} \in \Lag{\exp(\End(\ex X))}\otimes \mathbb Q((\hbar))\ .\]
 So we have
\[Z_{GW}(\ex X)=\sum_{\chi,\beta} \hbar^{-\chi}t^\beta Z_{GW}(\ex X)_{\beta,\chi}=\sum_\beta t^\beta Z_{GW}(\ex X)_\beta\ .\]

\subsection{ Gromov--Witten Donaldson-Thomas correspondence}

\

Pandharipande--Thomas invariants, \cite{PT_2009}, are integer invariants of 3-folds that count stable pairs consisting of a sheaf supported on an embedded curve, and a section. Donaldson--Thomas invariants, \cite{Thomas_2000}, are related invariants counting ideal sheaves supported in dimension at most $1$. Both these invariants  can be generalised to logarithmic 3-folds $\rm X$; see  \cite[Section 2]{logGWDT} and \cite{logDT}. In particular, for each $\beta\in H_2(\rm X)$, and integer  $ \chi$, there is a moduli space $PT_{\beta,\chi}(\rm X)$ of stable pairs and a moduli space $DT_{\beta,\chi}(\rm X)$ of ideal sheaves such that the sheaf has holomorphic Euler characteristic\footnote{Note that the holomorphic and topological Euler characteristics are not the same. The holomorphic Euler characteristic of the ideal sheaf of a smoothly embedded curve is half the topological Euler characteristic.} $\chi$, and  represents the homology class $\beta$.  Moreover, both $DT_{\beta,\chi}(\rm X)$ and $PT_{\beta,\chi}(\rm X)$ admit  integral virtual fundamental classes $[DT_{\beta,\chi}(\rm X)]$ and $[PT_{\beta,\chi}(\rm X)]$. In the case that $\rm X=(X,D)$ is a logarithmic Calabi--Yau 3-fold, the dimension of of both of these virtual fundamental classes  is $\beta\cdot D$, and they may be pushed forward to a positive integral sum of lagrangian cycles within a Hilbert scheme of points on components of $D$. 

Let $\ex X$ be an exploded manifold with a refinement isomorphic to the explosion of a log Calabi--Yau 3-fold $\rm X$. Let us construct generating functions $Z_{PT}(\ex X)$ and $Z'_{DT}(\ex X)$ conjecturally related to $Z_{GW}(\ex X)$ by a lagrangian correspondence constructed in Section \ref{correspondence section}. Given a curve class $\beta\in \rh_2(\ex X)$, we can choose a smooth logarithmic modification $\mathrm X'=(X',D)$ of $\rm X$ such that $\beta$ is a homology class on $X'$. Each component $D_v$ of the divisor corresponds to a primitive integral vector $v\in \totb{\ex X}$, and we need that 
\[\beta(v):=\beta\cdot D_v\geq 0\] for $PT_{\beta,\chi}$ or $DT_{\beta,\chi}$ to be nonempty.

There are natural evaluation maps
\[ev: PT_{\beta,\chi}(\ex X)\longrightarrow \prod_{v}D^{[\beta(v)]}_v\]
\[ev: DT_{\beta,\chi}(\ex X)\longrightarrow \prod_{v}D^{[\beta(v)]}_v\]
discussed in \cite[Section 3]{logGWDT}. For sheaves whose support is in general position with respect to $D$, this evaluation map is given by restriction of the sheaf to $D$, and  \cite[Theorem 1.6.1]{ThomasThesis} leads us to expect that the image of these evaluation maps is isotropic. We accordingly consider the pushforward of virtual fundamental classes as lagrangian cycles.
\[ev_*[PT_{\beta,\chi}(\ex X)] \in \Lag{\prod_{v} D^{[\beta(v)]}_v}=\Lag{\prod_{v}(D_v^\circ)^{{[\beta(v)]}}}^-\]

By restricting the interior $D_v^\circ $ of $D_v$, we remove all dependence on a choice of logarithmic modification. Use the notation
\[D^{[\beta]}:=\prod_v (D_v^\circ)^{{[\beta(v)]}} .\]
This is a smooth complex symplectic variety, which has the advantage\footnote{The disadvantage of restricting to the interior like this is that the star product of these invariants becomes less natural, and the star product is used in the gluing formula for such relative invariants.  The proper way to proceed would be to define these sheaf invariants directly for the exploded manifold $\ex X$, and have evaluation spaces $\ex X^{[\beta]}$ constructed as Hilbert scheme analogues of the evaluation stacks $\mathcal X^{\bf p}$. This would involve extending \cite{logHS} to work in this marginally more general setting. } that it does not depend on the choice of logarithmic modification $\rm X'$. 

Summing over the contributions of different holomorphic Euler characteristics, we get
\[Z_{PT}(\ex X)_{\beta}:=\sum_{\chi} ev_*[PT_{\beta,\chi}]q^{\chi-\beta\cdot D/2}\in \Lag{D^{[\beta]}}^-\otimes \mathbb Z[[q^{\frac 12},q^{-\frac 12}]] \]

We can sharpen \cite[Conjecture 3.2]{PT_2009} and \cite[Conjecture 5.3.1]{logGWDT} to the following.

\begin{conj}\label{integral} $Z_{PT}(\ex X)_{\beta}$ is a finite sum of lagrangian cycles with coefficients the Laurent expansions of rational functions in $q^{\frac 12}$. Moreover, the coefficient of each lagrangian cycle is either invariant under $q^{\frac 12}\mapsto q^{-
\frac 12}$ or $q^{\frac 12}\mapsto -q^{\frac 12}$. \end{conj}
This last invariance condition follows from the conjectured relationship between $Z_{PT}(\ex X)_\beta$ and $Z_{GW}(\ex X)_\beta$ after the change of variables $q^{\frac 12}=ie^{i\hbar/2}$; see Conjecture \ref{GWPT} below. It encodes the fact that both generating functions are real. 

We can similarly define a partition function for Donaldson--Thomas invariants.
\[Z_{DT}(\ex X)_{\beta}:=\sum_{\chi\in\mathbb Z} ev_*[DT_{\beta,\chi}]q^{\chi-\beta\cdot D/2} \in \Lag{D^{[\beta]}}^-\otimes \mathbb Z[[q^{\frac 12},q^{-\frac 12}]]\]
Noting that $Z_{DT}(\ex X)_0\in \mathbb Z[[q^{\frac 12},q^{-\frac 12}]]$, define
\[Z_{DT}'(\ex X)_{\beta}=\frac{Z_{DT}(\ex X)_{\beta}}{Z_{DT}(\ex X)_{0}}\ .\]
As a slight sharpening of \cite[Conjecture 5.3.1]{logGWDT} and \cite[Conjecture 3.3]{PT_2009}, we have the following.
\begin{conj}\[Z'_{DT}(\ex X)_{\beta}=Z_{PT}(\ex X)_{\beta}\]
within $\Lag{D^{[\beta]}}^-\otimes \mathbb Z[[q^{\frac 12},q^{-\frac 12}]]$.
\end{conj}

Let us now state the conjectured correspondence between Gromov--Witten invariants and Pandharipandre--Thomas or Donaldson--Thomas invariants. 

 Because $Z_{PT}(\ex X)$ is defined as a lagrangian correspondence in the Hilbert scheme of a noncompact smooth scheme rather than the corresponding compact log scheme or exploded manifold,  the fist thing we need to do is replace the correspondence $Z_{GW}(\ex X)$ on an exploded stack with a correspondence $Z_{GW}(\ex X)^\circ$ on a stack over the category of schemes. 
Recall that $\ex X$ has a cone structure on $\totb{\ex X}$ determined by a choice of $0\in\totb {\ex X}$. The set of points $\ex X^\circ $ with tropical part $0$ is smooth (non-compact) Calabi--Yau 3-fold. This agrees with the interior $\rm X^\circ $ when  a refinement of $\ex X$ is isomorphic to $\expl \rm X$. Similarly, $\ex X_v$, $\mathcal X_v$, $\ex X^{\bf p}$ and $\mathcal X^{\bf p}$ all have an induced cone structure on their tropical parts, and the set of points with tropical part $0$ is a smooth scheme or Deligne--Mumford stack with a natural inclusion 
\[(\ex X^{\bf p})^\circ \subset  \ex X^{\bf p} \text{ or } (\mathcal X^{\bf p})^\circ\subset \mathcal X^{\bf p}\ .\]
In particular,  if $D_v$ is a component of the divisor in $(X',D)$ then $\ex X_v^\circ =D_v^\circ$, and $(\ex X^{\bf p})^\circ=\prod_v (D_v^\circ)^{\mathbf p(v)}$.  Recall that $Z_{GW}(\ex X)_\beta$ is determined by $\Aut \bf p$--invariant lagrangian correspondences
\[Z_{GW}(\ex X)_{\beta,\bf p} \in \Lag{\ex X^{\bf p}}\otimes Q((\hbar))\]
which we can restrict to $(\ex X^{\bf p})^\circ $ to define
\[Z_{GW}(\ex X)^\circ _{\beta,\bf p} \in \Lag{(\ex X^{\bf p})^\circ }^-\otimes Q((\hbar))\]
which then defines 
\[Z_{GW}(\ex X)^\circ_\beta \in \Lag{\coprod_{\bf p}(\mathcal X^{\bf p})^\circ}^{-}\otimes Q((\hbar))\]
and 
\[Z_{GW}(\ex X)^\circ :=\sum_{\beta} t^\beta Z_{GW}(\ex X)^\circ_\beta \ . \]

We now define a unitary correspondence between $\coprod_{\bf p} (\mathcal X^{\bf p})^\circ$ and $\coprod_{\beta\in M} D^{[\beta]}$. Analogously to Example \ref{HSD}, define a lagrangian cycle 
\[\Delta_{\bf p} \in \Lag{(\ex X^{\bf p})^\circ , D^{[\beta]}}\]
\[\begin{tikzcd} \Delta_{\bf p}\rar\dar & D^{[\beta]}\dar{h}
 \\  (\ex X^{\bf p})^\circ \rar{f} &  \prod_v S^{\beta(v)} (D_v^\circ)\end{tikzcd}\]
such that the above diagram is a fibre product diagram, with $h$ the Hilbert--Chow morphism, and $f$ the map sending points in $\ex (X^{\bf p})^\circ $ to the corresponding $0$--cycles, where a point in $\ex X_v$ is sent to the corresponding point in $D_v^\circ$, weighted by $|v|$.

Each $\Delta_{\bf p}$ is $\Aut \bf p$ invariant, and hence also defines a lagrangian cycle in $\Lag{ (\mathcal X^{\bf p})^\circ, D^{[\beta]}}$.  As in Example \ref{HSD2c}, we can then define a lagrangian correspondence 
\[\mathcal L\in \Lag{\coprod_{\bf p} (\mathcal X^{\bf p})^\circ,\coprod_{\beta\in M} D^{[\beta]}}\otimes \mathbb Z[i]\]
by 
\[\mathcal L:=\sum_{\bf p}\left(\prod_v i^{\mathbf p(v)(|v|-1)}\right)\Delta_{\bf p}\ \]
with adjoint 
\[\mathcal L^\dagger:=\sum_{\bf p}\left(\prod_v i^{\mathbf p(v)(|v|-1)}\right)\Delta_{\bf p}^\dagger
\ \ \ \text{ in }\ \ \ 
 \Lag{\coprod_{\beta\in M} D^{[\beta]},\coprod_{\bf p} (\mathcal X^{\bf p})^\circ}\otimes \mathbb Q[i] \]
Example \ref{HSD2c} implies that $\mathcal L^\dagger$ is the inverse of $\mathcal L$.

The original conjecture \cite{MNOP_2006} relating Gromov--Witten invariants and Donaldson--Thomas invariants has the generalisations \cite[Conjecture 3.3]{PT_2009} and \cite[Conjecture 5.3.2]{logGWDT}, which we sharpen to the following.

\begin{conj}\label{GWPT} With the change of variables $q^{\frac 12}=ie^{i\hbar/2}$ we have
 \[Z_{GW}(\ex X)^\circ =\mathcal L\star Z_{PT}(\ex X) \text{ and }  \mathcal L^\dagger \star Z_{GW}(\ex X)^\circ =Z_{PT}(\ex X)\ .
\]
In particular, 
\[Z_{GW}(\ex X)^\circ _{\beta,\bf p}= \left(\prod_v i^{\mathbf p(v)(|v|-1)}\right)\Delta_{\bf p} \star_{D^{[\beta]}}Z_{PT}(\ex X)_\beta\ .\]
\end{conj}
Combining this with Conjecture \ref{integral}, we get the following integrality conjecture for $Z_{GW}$

\begin{conj}\label{integrality}  $i^{\sum_v \bf p(v)(\abs v-1)}Z_{GW}(\ex X)_{\beta,\bf p}$ is a finite sum of lagrangian cycles with coefficients that are both Laurant series in $\mathbb Z [[q^\frac 12,q^{-\frac 12}]]$ and rational functions in $q^{\frac 12}=ie^{i\hbar/2}$. As $Z_{GW}(\ex X)$ is real, we get that these rational functions are invariant under 
\[q^{\frac 12}\mapsto (-1)^{1+\sum_v\mathbf p(v)(\abs v-1)}q^{-\frac 12}\ .\]
\end{conj}

As a corollary of Conjecture \ref{integrality}, we can deduce the same integrality properties for the Gromov--Witten invariants $\eta$ counting connected curves, so long as each non-trivial curve has non-empty contact data. This is because, once we label contact data, there are no automorphisms swapping identical components of a connected curve.

\begin{conj}\label{connected integrality} In the case that every nontrivial curve in $\ex X$ has nonempty contact data, the connected Gromov--Witten invariant $\eta_{\beta,\bf p}$ satisfies the same integrality condition:
\[i^{\sum_v \bf p(v)(\abs v-1)}\eta_{\beta,\bf p}\in \mathbb Z[[q^{\frac 12},q^{-\frac 12}]]\]
is a rational function in $q^{\frac 12}$ where $q^\frac 12=ie^{i\hbar/2}$.
\end{conj}

\begin{remark} Conjecture \ref{connected integrality} is false without the assumption that nontrivial curves have nonempty contact data. A counter example is provided by $\ex X$ containing an embedded sphere with normal bundle $O(-1)\oplus O(-1)$. The count of connected degree $d$ covers of such a sphere is $\frac 1{d(2\sin{d\hbar/2})^2}=\frac{\pm 1}{d (q^{d/2} \pm q^{-d/2})^2}=\pm \frac 1d q^d \pm \frac 2{d}q^{2d}\pm \dotsc$. 
\end{remark}

\begin{remark} We could also expect a version of Gopakumar--Vafa integrality, \cite{GV1,GV2,IonelGV} to hold,  with an explicit formula for the contribution of multiple covers of curves. Something of this nature is likely true, but the naive generalisation of Gopakumar--Vafa integrality does not hold, with  counter examples given by the curve counts in Example \ref{T3 example} below, particularly equation (\ref{trivalent vertex}) with $n=2$.

\end{remark}

\begin{example} \label{T3 example} The exploded manifold $\ex X=\ex T^3$ has tropical part  $\totb{\ex T^3}=\mathbb R^3$. If we   subdivide  $\mathbb R^3$ into a toric fan, the resulting  refinement of $\ex T^3$ is the explosion of the corresponding toric $3$-fold. As $\ex X^\circ=(\mathbb C^*)^3$, the reader can think of $\ex T$ as an exploded manifold replacement for $\mathbb C^*$ throughout this example.   

The curve class $\beta$ of a curve in $\ex X$ is uniquely determined by contact data. Choose two integral vectors $v_1$, $v_2$ in $\mathbb Z^3\setminus 0$ with $v_1\wedge v_2\neq 0$ and consider curves with contact data $\mathbf p$ the characteristic function of $\{v_1,v_2,-v_1-v_2\}$.  There is a free and transitive action of $\ex T^3$ on the moduli space $\mathcal M_{0,\bf p}(\ex X)$ of genus $0$ curves with contact data $\bf p$. Moreover, each of these vectors $v$ correspond  an action of $\ex T$ on $\ex T^3$ of weight $v$,  $\mathcal X_v$ is the stack quotient by this action, and $\ex  X_v$ is the quotient by the corresponding free action of weight $v/\abs v$. In particular, $\mathcal X_v^\circ$ and $\ex X_v^\circ$ are the quotient of $\mathbb C^3$ by the $\mathbb C^*$ actions of weight $v$ or $v/\abs v$ respectively. The induced action of $\ex T^3$ on $\ex X^{\bf p}=\ex T^6$ is only free if $\abs{v_1\wedge v_2}=1$. Otherwise, we need to pass to a subgroup of index $n=\abs{v_1\wedge v_2}$ to get a free action, so the evaluation map restricted to $\mathcal M_{0,\bf p}$ is a $n$-fold cover of its image, which is a lagrangian subvariety  $V\subset \ex X^{\bf p}$. In this case, the moduli space of higher genus curves has the same image, and there are no disconnected curves that contribute to $Z_{GW}(\ex X)_{\bf p}$.  We can then apply \cite[Theorem 1.1]{3d} to give that
\begin{equation}\label{trivalent vertex}\eta_{\bf p}=Z_{GW}(\ex X)_{\bf p}=\frac{(i^{-1}q^{\frac 12})^n-(i^{-1}q^{\frac 12})^{-n}}{i} V=2\sin(n\hbar/2)V:=[n]_q V\ .\end{equation}

Note that if $n=\abs{v_1\wedge v_2}$ is odd, we obtain that $|v|$ is odd for $v=v_1,v_2,-(v_1+v_2)$, so $i^{\sum_v\mathbf p(v)(\abs v-1)}$ is real, and in this case $Z_{GW}(\ex X)_{\bf p}$ is a real rational function in $q^{\frac 12}$, as expected. On the other hand if $n=\abs{v_1\wedge v_2}$ is even, $|v|$ is even for an even number of the vectors $v_1,v_2,-(v_1+v_2)$, so $i^{\sum_v\mathbf p(v)(\abs v-1)}$ is imaginary,  and in this case $Z_{GW}(\ex X)_{\bf p}$ is a purely imaginary rational function in $q^{\frac 12}$, as expected.

\end{example}

\begin{example}
Let $\ex X=\expl (X,D)$ where $(X,D)$ is obtained from $(\mathbb CP^1,\{\pm 1\})^3$ by taking the non-toric blowup along the line where $z_2=1$ and $z_3=0$. In this case too, the homology class of a curve is determined by its contact data.  Many curves in $\ex X$ are trapped in the `wall' where $z_3=0$ by their choice of contact data. In particular, if we have contact data which precludes intersection with $z_3=\{\pm 1\}$, we can write this contact data in terms of vectors in $\mathbb Z^2\setminus 0$, where $(a,b)$ indicates contact of order $\abs a$ with $z_1= \frac{a}{\abs a} $, and order $\abs b$ with the strict transform of $z_2= \frac b{\abs b} $. For $(a,b)$ with $a\neq 0$, $\ex X_{(a,b)}$ is isomorphic to a refinement of $\ex T\times\expl(\mathbb CP^1,\{\pm 1\})$, with $z_3$ providing the coordinate on the second component. Curves trapped in the wall $z_3=0$ must therefore map to the lagrangian subvariety where $\{z_3=0\}\subset \ex X_{(a,b)}$. The evaluation spaces $\ex X_{0,\pm 1}$, on the other hand are isomorphic to the nontoric blowup of $(\mathbb CP^1,\{\pm 1\})^2$ at the point $(1,0)$, with the coordinates $(z_2,z_3)$ from $\ex X$. In this case, the wall $\{z_3=0\}$ within  $\ex X_{(0,\pm 1)}$ breaks up into two lagrangian subvarieties $\{z_3=0\}=E^+\cup E^-$. In many cases, choosing contact data resticts the image of the evaluation map to be one of these subvarieties $E^\pm$, so a calculation of numerical Gromov--Witten invariants, as in \cite[Sections 11 and 12]{tropological}, will determine $Z_{GW}(\ex X)_{\bf p}$.

For example, when $\bf p$ is the indicator function of $(\pm n, 0)$,  \cite[Lemma 12.5]{tropological} implies that 
\begin{equation}\eta_{\bf p}=Z_{GW}(\ex X)_{\bf p}= \frac{(-1)^{n+1}i}{(i^{-1}q^{\frac 12})^n-(i^{-1}q^{\frac 12})^{-n}}\{z_3=0\}:=\frac {(-1)^{n+1}}{[n]_q}\{z_3=0\}\ .\end{equation}
As
\[\frac {(-1)^{n+1}}{[n]_q}=\frac {i(iq^{\frac 12})^n}{1-(-q)^n}=i^{n+1}q^{n/2}(1+(-q)^n+(-q)^{2n}+\dotsb )\]
this satisfies Conjecture \ref{connected integrality}. Other Gromov--Witten invariants enumerating curves with contact data on the wall $\{z_3\}$ can be calculated using the techniques of \cite{tropological}, which features an algebra of operators $\mathcal W_{v,\ell}$ satisfying the commutation relation
\[\mathcal [W_{v_1,\ell},W_{v_2,\ell}]=[v_1\wedge v_2]_q W_{v+w,\ell}+(n_{\ell}\wedge v)\delta_{v+w}\]
compatible with the integrality of Conjecture \ref{integrality}. 

If we take a further non-toric blowup of $(X,D)$ along the strict transform of the line where $z_1=1$ and $z_3=0$, we get  a logarithmic model of the topological vertex of Aganagic, Klemm, Mari\~no and Vafa, \cite{topologicalvertex}. Again, we can choose contact data so that curves are trapped in the wall $\{z_3=0\}$, and the techniques of \cite[Section 13]{tropological} can be used to verify Conjecture \ref{integrality} for these curves.

\end{example}

\bibliographystyle{plain}
\bibliography{../ref}
 \end{document}